\documentclass{amsart}
\usepackage{amsfonts, amsmath, amssymb, amscd, amsthm, graphicx,enumerate}
\usepackage{hyperref,xypic}
\usepackage{mathrsfs}

\makeatletter
\def\blfootnote{\xdef\@thefnmark{}\@footnotetext}
\makeatother

\usepackage{tikz-cd}
\usetikzlibrary{arrows}

\newtheorem{thm}{Theorem}[section]
\newtheorem{cor}[thm]{Corollary}
\newtheorem{lem}[thm]{Lemma}
\newtheorem{prop}[thm]{Proposition}

\newtheorem{ques}[thm]{Question}

\newtheorem{defn}[thm]{Definition}
\newtheorem{observation}[thm]{Observation}
\newtheorem{sa}[thm]{Standing Assumption}
\newtheorem{rem}[thm]{Remark}
\newtheorem{ex}[thm]{Example}

\newtheorem{convention}[thm]{Convention}

\newfont{\eufm}{eufm10}

\newcommand{\e}{\varepsilon }
\renewcommand{\phi}{\varphi}

\renewcommand{\l}{{\ell}}

\newcommand{\N}{\mathbb N}
\newcommand{\Z}{\mathbb Z}
\newcommand{\R}{\mathbb R}
\newcommand{\llangle }{\langle\hspace{-.7mm}\langle }
\newcommand{\rrangle }{\rangle\hspace{-.7mm}\rangle }

\newcommand{\Lab }{{\bf Lab}}

\newcommand{\Aut} {\mathrm{Aut}}
\def\Mod {\mathrm{Mod}}
\def\AJSJ {\mathbb A_{\mathrm{JSJ}}}
\def\AMOD {\mathbb A_{\mathrm{M}}}

\def\Isom {\mathrm{Isom}}
\def\Hom {\mathrm{Hom}}
\def\Area {\mathrm{Area}}
\def\rel {\mathrm{rel}}
\def\Cay {\mathrm{Cay}}

\def\Fix {\mathrm{Fix}}

\def\Stab {\mathrm{Stab}}
\def\ufin {\mathrm{ufin}}

\def\mc {\mathcal}
\def\onto {\twoheadrightarrow}

\def\co{\colon\thinspace} 

\def\QHAH {QH$_{(\mc{A},\mc{H})}$}

\makeatletter
\newtheorem*{rep@theorem}{\rep@title}
\newcommand{\newreptheorem}[2]{%
\newenvironment{rep#1}[1]{%
 \def\rep@title{#2 \ref{##1}}%
 \begin{rep@theorem}\it}%
 {\end{rep@theorem}}}
\makeatother

\newreptheorem{theorem}{Theorem}
\newreptheorem{lemma}{Lemma}
\newreptheorem{corollary}{Corollary}
\newreptheorem{definition}{Definition}

\newtheorem{thmA}{Theorem}

\newtheorem{corA}[thmA]{Corollary}
\newtheorem{defA}[thmA]{Definition}

\begin{document}

\title[Homomorphisms to acylindrically hyperbolic groups, I]{Homomorphisms to acylindrically hyperbolic groups I: Equationally noetherian groups and families}
\author{D. Groves}
\address{Department of Mathematics, Statistics and Computer Science, University of Illinois at Chicago, 322 Science and Engineering Offices (M/C 249), 851 S. Morgan St., Chicago IL 60607, USA}
\email{groves@math.uic.edu}
\thanks{The first author was partially supported by a grant from the Simons Foundation (\#342049 to Daniel Groves) and by NSF grant DMS-1507076.}

\author{M. Hull}
\address{Department of Mathematics, University of Florida, 358 Little Hall, Gainesville, FL 32611, USA}
\email{mbhull@ufl.edu}
\date{}

\begin{abstract}
We study the set of homomorphisms from a fixed finitely generated group $G$ into a family of groups $\mc{G}$ which are `uniformly acylindrically hyperbolic'.  Our main results reduce this study to sets of homomorphisms which do not diverge in an appropriate sense.  As an application, we prove that any relatively hyperbolic group with equationally noetherian peripheral subgroups is itself equationally noetherian.
\end{abstract}

\maketitle

\setcounter{tocdepth}{1}
\tableofcontents

\section{Introduction}

In this paper, we are interested in the following problem.  Suppose that $\mc{G}$ is a family of groups, and $G$ is a finitely generated group.  What is the structure of the set $\Hom(G,\mc{G})$ of all homomorphisms from $G$ to elements of $\mc{G}$? We are concerned with the situation where $\mc{G}$ consists of a {\em uniformly acylindrically hyperbolic} family of groups (see Definition \ref{def:UAH} below).  A specific example is where $\mc{G} = \{ \Gamma \}$ is a single acylindrically hyperbolic group, though there are many other interesting cases, as we explain below.  The study in this paper builds on many previous authors' work, for example \cite{Ali,Gro,groves_rh,JalSel,KM1, KM2,ReiWei,sela:dio1,sela:DioX}.  With the exception of \cite{JalSel}, all of these previous works consider a fixed group to be the target, rather than a family of groups.

{\em Acylindrically hyperbolic groups} (defined by Osin in \cite{Osi13}, see Definition \ref{def:AHG} below) are a large class of groups that have recently been the study of intense activity in geometric group theory.  On the one hand, the definition is powerful enough that there is a rich theory of these groups, while on the other hand it is flexible enough that there any many examples of acylindrically hyperbolic groups.  Examples include (see \cite[Section 8]{Osi13} for more discussion):
\begin{enumerate}
\item Any non-elementary hyperbolic group;
\item Any non-elementary relatively hyperbolic group;
\item The mapping class group of almost any surface of finite type;
\item The outer automorphism of a finitely generated free group of rank at least $2$;
\item Directly indecomposable right-angled Artin groups, and more generally any ${\mathrm{CAT}}(0)$ group which contains a rank 1 isometry;
\item Any group admitting a non-elementary action on a simplicial tree which is $k$--acylindrical for some $k$.
More generally, a group admitting a non-elementary $(k,C)$--acylindrical action on a simplicial tree for some $k$ and $C$. In particular, this includes the fundamental group of any irreducible, compact $3$--manifold with a non-trivial JSJ-decomposition.
\end{enumerate}

In this paper, we are concerned with {\em families} of acylindrically hyperbolic groups which are `uniform' in the sense given in Definition \ref{def:UAH} below.  The first example of such a `family' is a single acylindrically hyperbolic group, and this is an important case.  Perhaps the most natural example of a uniform family is the set of groups acting $k$--acylindrically on simplicial trees for some fixed $k$.  A particularly important class is the set of fundamental groups of compact $3$--manifolds, since a closed, orientable, irreducible $3$--manifold is either finitely covered by a torus bundle over the circle or else the JSJ decomposition is $4$--acylindrical by \cite[Lemma 2.4]{WiltonZalesskii}. We expect the results in this paper to have many applications to the study of $\Hom(G,\mc{M}_3)$ where $G$ is an arbitrary finitely generated group and $\mc{M}_3$ is the set of fundamental groups of (compact) $3$--manifolds.

One motivation for studying sets of homomorphisms is topological, since homotopy classes of maps between aspherical spaces are in bijective correspondence with homomorphisms between their fundamental groups. $3$--manifold groups are of particular interest here, since many $3$--manifolds are aspherical. We expect that studying sets of homomorphisms to either a fixed $3$--manifold group or to the family of all $3$--manifold groups can be used to better understand (homotopy classes of) maps between $3$--manifolds.

Another motivation which is more closely related to logic comes from the study of equations in groups or \emph{algebraic geometry over groups} \cite{BMR1}. Understanding sets of solutions to systems of equations in a group $\Gamma$ is the first step towards understanding the elementary theory of $\Gamma$, and these sets of solutions naturally correspond with sets of homomorphisms to $\Gamma$. We explain this connection below in the more general context of a family of groups.

Let $\mc{G}$ be a family of groups and let $F_n$ be the free group of rank $n$. An \emph{equation} is an element $W\in F_n$ which we identify with a freely reduced word in the generators $x_1^{\pm 1},..., x_n^{\pm 1}$. For $\Gamma\in \mc{G}$ and $(g_1,...,g_n)\in\Gamma^n$,  let $W(g_1,...,g_n)$ denote the element of $\Gamma$ obtained by replacing each $x_i$ with $g_i$ and each $x_i^{-1}$ with $g_i^{-1}$ in the word $W$. $(g_1,...,g_n)$ is called a \emph{solution} to the equation $W$ if $W(g_1,...,g_n)=1$. For a \emph{system of equations} $S\subseteq F_n$, we define

\[
V_\mc{G}(S)=\{(g_1,...,g_n)\in\Gamma^n\;|\; \Gamma\in\mc{G} \text{ and } W(g_1,...,g_n)=1 \;\forall W\in S\}.
\]

That is, $V_{\mc{G}}(S)$ is the \emph{algebraic set} of solutions to the system of equations $S$. Observe that $(g_1,..., g_n)\in\Gamma^n$ belongs to $S$ if and only if the map which sends each $x_i$ to $g_i$ extends to a homomorphism $F_n/\llangle S\rrangle\to \Gamma$. Hence, there is a natural bijective correspondence between $V_{\mc{G}}(S)$ and $\Hom(F_n/\llangle S\rrangle, \mc{G})$.

The most important property about a family $\mc{G}$ of groups for the purposes of this paper is that of the family $\mc{G}$ being {\em equationally noetherian}, which can be thought of as a group-theoretic version of the conclusion of the Hilbert Basis Theorem.

\begin{defA} \label{d:en family}
We say that a family $\mc{G}$ is an {\em equationally noetherian family of groups} if for any $n$ and any $S\subseteq F_n$, there exists a finite $S_0\subseteq S$ such that $V_\mc{G}(S_0)=V_\mc{G}(S)$.
\end{defA}

This definition generalizes the well-studied notion of what it means for a group to be equationally noetherian (the case where $\mc{G}$ consists of a single group). This definition can also be translated to a statement in terms of sets of homomorphisms as follows. The family $\mc{G}$ is equationally noetherian if and only if for any finitely generated group $G$, there is a finitely presented group $P$ and an epimorphism $P\onto G$ which induces a bijection between $\Hom(G, \mc{G})$ and $\Hom(P, \mc{G})$. Another characterization of the equationally noetherian property in terms of sequences of homomorphisms is given by Theorem \ref{eqnfamily}.

 The equationally noetherian property has been very important in the study of equations over groups, model theory of groups, and many other questions: see, for example, \cite{BMR1,  Gro, GrovesWilton09, JalSel, KM1, OuldHoucine, sela:dio1, sela:dio7, sela:DioX}. Similarly, the question of whether a family of groups $\mc{G}$ is equationally noetherian is of fundamental importance for the study of the set $\Hom(G,\mc{G})$. In particular, this property is more or less essential in order to have a meaningful description of sets of the homomorphisms to a group or family of groups.

We note that to say that a family $\mc{G}$ is equationally noetherian is {\em much} stronger than to say that 
each element of $\mc{G}$ is equationally noetherian.  The following collections all have the property that they are not equationally noetherian as a family while each individual member is equationally noetherian:  (i) Finite groups; (ii) hyperbolic groups; and (iii) linear groups (see Example \ref{l:finite not en}).  There are many other such examples.

One important consequence of a family $\mc{G}$ being equationally noetherian is that any sequence of surjective maps
\[	\Gamma_1 \onto \Gamma_2 \onto \dots	\onto \Gamma_n \onto \dots \]
where each $\Gamma_i \in \mc{G}$ eventually consists of isomorphisms (see Theorem \ref{thm:eqnhopf}).  This can be thought of as an extension of the Hopfian property to families of groups; In particular, it implies that each element of $\mc{G}$ is Hopfian (see Corollary \ref{c:limithopfian}).

Let $\mc{G}$ be a uniformly acylindrically hyperbolic collection of groups (see Definition \ref{def:UAH}).  Since the subgroups of $\mc{G}$ which act with bounded orbits on the associated $\delta$--hyperbolic spaces could be completely arbitrary, it is impossible to say anything about $\Hom(G,\mc{G})$ in general since this is as hard as understanding the set of homomorphisms from $G$ to {\em all} groups.  Therefore, we first restrict attention to the case that a sequence of homomorphisms behaves in an interesting way with respect to the $\delta$--hyperbolic space.  The right notion is for a sequence to {\em diverge}, as defined in Definition \ref{def:divergent}.  The following is our main technical result; in light of Theorem \ref{eqnfamily}, it says that if $\mc{G}$ fails to be equationally noetherian then that failure is witnessed by some non-divergent sequence of homomorphisms.  In the following, $\omega$ is a non-principal ultrafilter which is fixed once and for all throughout this paper (see Section \ref{ss:EN and limit}).

\begin{thmA}\label{t:EN up to non-div}
 Suppose that $\mc{G}$ is uniformly acylindrically hyperbolic family. Furthermore, suppose that whenever $G$ is a finitely generated group and $\eta_i \co G \to \mc{G}$ is a non-divergent sequence of homomorphisms, $\eta_i$ $\omega$--almost surely factors through the limit map $\eta_\infty$.  Then $\mc{G}$ is equationally noetherian.
\end{thmA}

This result allows a reduction of the study of the set of $\Hom(G,\mc{G})$ to the study of non-divergent sequences of homomorphisms.  With extra assumptions, it is possible to say more.

Answering a question from \cite{CK}, in \cite[Theorem 9.1]{sela:DioX} Sela proved that if $A$ and $B$ are equationally noetherian then their free product $A \ast B$ is also equationally noetherian.  In Corollary \ref{t:G^ast EN} we provide a generalization of this:  if $\mc{G}$ is an equationally noetherian family of groups then the set of all free products of elements of $\mc{G}$, denoted by $\mc{G}^\ast$, is itself an equationally noetherian family.

\begin{corA} \label{t:G^ast EN}
If $\mc{G}$ is an equationally noetherian family of groups, then $\mc{G}^{\ast}$ is an equationally noetherian family of groups.
\end{corA}

We remark that in \cite{JalSel}, Jaligot and Sela undertake much of the same work that we do in the case of free products.  They suggested that their work should be able to be generalized, in particular to $k$--acylindrical actions on trees.  We believe that Theorem \ref{t:EN up to non-div} is the appropriate generalizing of (the beginning of) \cite{JalSel}.
In order to make further progress in this direction, one needs more control over the elliptic elements.  The reason for this is that when considering free products, it is clear that a non-divergent sequence has a limiting simplicial action on a tree with trivial stabilizers.  Even in the case of $k$--acylindrical actions on trees, the limiting action is $k$--acylindrical and simplicial in the non-divergent case, but it need not be any nicer than that.  For example, it seems very hard to retain finite generation of edge stabilizers under non-divergent limits.

In \cite{sela:DioX}, Sela proved that torsion-free hyperbolic groups are equationally noetherian.  In \cite[Theorem 5.16]{Gro}, the first author proved that a torsion-free group which is hyperbolic relative to abelian subgroups is equationally noetherian.  Subgroups of equationally noetherian groups are equationally noetherian, so the strongest possible result in this direction about relatively hyperbolic groups would be that if the peripheral subgroups of a relatively hyperbolic group are equationally noetherian then the whole group is.  Using  Theorem \ref{t:EN up to non-div}, we prove that this is indeed the case. 

\begin{thmA}\label{thm:rel hyp eqn}
If $\Gamma$ is hyperbolic relative to equationally noetherian subgroups, then $\Gamma$ is equationally noetherian. 
\end{thmA}

Note that Theorem \ref{thm:rel hyp eqn} uses the machinery developed in this paper for families of acylindrically hyperbolic groups to prove a theorem about a single relatively hyperbolic group which is much more general than previous results in this direction.

Since groups acting geometrically on a $\mathrm{CAT}(0)$ space with isolated flats are hyperbolic relative to virtually abelian subgroups \cite{hruskakleiner}, this theorem answers \cite[Question I.(8).(i)]{sela:diorp} and \cite[Question I.8.(ii)]{sela:diorp}. In the case where $\Gamma$ is torsion-free and the peripheral subgroups are abelian, these questions were answered by the first author in \cite{Gro}.

The following immediate consequence of Theorem \ref{thm:rel hyp eqn} has nothing {\it a priori} to do with equationally noetherian groups.
\begin{corA}
 Suppose that $\Gamma$ is hyperbolic relative to linear groups.  Then $\Gamma$ is Hopfian.
\end{corA}

As stated above, we expect that the results in this paper will have many further applications, in particular to the study of groups acting $k$--acylindrically on simplicial trees. In a sequel to this paper, we show that when $\mc{G}$ is both uniformly acylindrically hyperbolic and equationally noetherian, then there is a Makanin--Razborov diagram which parameterizes $\Hom(G, \mc{G})$ modulo homomorphisms which factor through non-divergent $\mc{G}$--limit groups.

We now provide a brief outline of the contents of this paper.  In Section \ref{s:Prelim} we provide some preliminary definitions and results about (families of) acylindrically hyperbolic groups.  In Section \ref{ss:EN and limit} we introduce the notion of an equationally noetherian family of groups, relate this to the well-studied notion of equationally noetherian groups and also to different kinds of limit groups.  In Section \ref{s:limits} we explain how a {\em divergent} sequence of homomorphisms leads to a limiting action on an $\R$--tree.  We also explain that the {\em Rips machine} applies to such actions and thus yields a graph of groups decomposition.  In Section \ref{s:JSJ} we prove the existence of a {\em JSJ decomposition} of a divergent $\mc{G}$--limit group, which encodes all of the graph of groups decompositions that may arise through the limiting construction.  We define {\em modular} groups of automorphisms, which is used to develop an analogue of Sela's `Shortening Argument' in this setting.  In Section \ref{s:short quotient} we study {\em shortening quotients}, and prove a certain descending chain condition for shortening quotients (Theorem \ref{dcc}).  This allows us to prove Theorem \ref{t:EN up to non-div}.  In Section \ref{s:rel hyp} we apply the machinery developed in previous sections, and particularly Theorem \ref{t:EN up to non-div}, to prove Theorem \ref{thm:rel hyp eqn}.

\begin{rem}
Our approach in this paper is motivated by the methods of Sela from \cite{sela:dio1} (and many other papers, for example  \cite{Gro, groves_rh, ReiWei}).  As such, although our setting is quite different to these earlier works, many of the technical details remain the same. Rather than add a considerable amount to the length of this paper, we have not attempted to make this paper self-contained.  Instead, we occasionally refer to other papers for certain technical details, especially \cite{Guirardel:Rtrees}, \cite{GuiLev2} and \cite{ReiWei}, only indicating the changes required in the setting of this paper.  We apologize to the reader who is not already familiar with these earlier papers.
\end{rem}

\subsection{Acknowledgements}

We would like to thank Henry Wilton for providing the simplified proof for Example \ref{l:finite not en}.

\section{Preliminaries} \label{s:Prelim}
Throughout this paper, all actions of groups on metric spaces are by isometries and all metric spaces are geodesic metric spaces. 

\begin{defn}
A geodesic metric space is called \emph{$\delta$--hyperbolic} if for any geodesic triangle with sides $p$, $q$, and $r$, $p$ belongs to the closed $\delta$--neighborhood of $q\cup r$. In this case, we refer to $\delta$ as the \emph{hyperbolicity constant} of the space. A metric space is simply called \emph{hyperbolic} if it is $\delta$--hyperbolic for some $\delta\geq0$. 
\end{defn}
See \cite[III.H]{BH} for many of the basic properties about $\delta$--hyperbolic spaces.

For a group $\Gamma$ acting on a metric space $X$ and a constant $K$, let 
$$\Fix_K(\Gamma)=\{x\in X\;|\; d(x, gx)\leq K\;\forall\;g\in \Gamma \}.$$ 
The following is a standard fact about hyperbolic metric spaces.
\begin{lem}\label{lem:quasicenter}
Suppose a group $\Gamma$ acts with bounded orbits on a $\delta$--hyperbolic metric space $X$ and $\e>0$. Then $\Fix_{4\delta+2\e}(\Gamma)\neq\emptyset$.
\end{lem}
\begin{proof}
The {\em $\epsilon$--quasi-centers} for a bounded set $Y \in X$ are the points in the collection
\[	C_\epsilon(Y) = \left\{ y \in X \mid Y \subset B(y,r_Y+\epsilon) \right\}	,	\]
where $r_Y = \inf \left\{ \rho > 0 \mid \exists z \in X\co  \  Y \subseteq B(z,\rho) \right\}$.

Let $x\in X$. By assumption, the $\Gamma$--orbit $\Gamma\cdot 
x$ is bounded, hence for all $\e>0$ the set of $\e$--quasi-centers is non-empty and has diameter at most $4\delta+2\e$ by \cite[Chapter III.$\Gamma$, Lemma 3.3]{BH}. Since the set of $\e$--quasi-centers of $\Gamma\cdot x$ is clearly $\Gamma$--invariant, the result follows.
\end{proof}

The following notion was introduced by Bowditch \cite{Bow2}.
\begin{defn}\label{acy}
Let $\Gamma$ be a group acting on a metric space $(X, d)$. The action is called {\em acylindrical} if for all $\e$ there exist $N_{\e}$ and $R_{\e}$ such that for all $x, y\in X$ with $d(x, y)\geq R_{\e}$,
\[
\left| \{g\in \Gamma \mid d(x, gx)\leq \e, d(y, gy)\leq \e\}\right| \leq N_{\e}.
\]
\end{defn}
We refer to the functions $N_\e$ and $R_\e$ as the \emph{acylindricity constants} of the action.

\begin{rem}
 The idea of an acylindrical action is intended (see \cite{Bow2}) to be related to Sela's notion of a $k$--acylindrical action on a tree \cite{sela:acyl}.  However, unfortunately the notion in Definition \ref{acy} does not generalize this notion from trees, but rather the notion of a $(k,C)$--acylindrical action.
\end{rem}

Given a group $\Gamma$ acting on a hyperbolic metric space $X$, an element $g\in \Gamma$ is called \emph{loxodromic} if for some (equivalently, any) point $x\in X$, the map $\mathbb Z\to X$ defined by $n\to g^nx$ is a quasi-isometric embedding. Such an element $g$ admits a bi-infinite quasi-geodesic axis on which $g$ acts as a non-trivial translation. We denote the limit points on $\partial X$ of the axis of $g$ by $g^{+\infty}$ and $g^{-\infty}$. Given two loxodromic elements $g$ and $h$, we say they are \emph{independent} if $\{g^{+\infty}, g^{-\infty}\}\cap\{h^{+\infty}, h^{-\infty}\}=\emptyset$. The action of $\Gamma$ on $X$ is called \emph{non-elementary} if $\Gamma$ contains two independent loxodromic elements. By a result of Gromov \cite[8.2.F]{gromov:hyp}, this is equivalent to $\Gamma$ containing infinitely many pairwise independent loxodromic elements.
\begin{defn}\cite[Definition 1.3]{Osi13} \label{def:AHG}
A group $\Gamma$ is {\em acylindrically hyperbolic} if $\Gamma$ admits a non-elementary, acylindrical action on a $\delta$--hyperbolic metric space.
\end{defn}

A group $\Gamma$ acting on a metric space $X$ is called \emph{elliptic} if some (equivalently, any) $\Gamma$--orbit is bounded.
\begin{thm}\cite[Theorem 1.1]{Osi13}
Let $\Gamma$ be a group acting acylindrically on a $\delta$--hyperbolic metric space. Then $\Gamma$ satisfies exactly one of the following:
\begin{enumerate}
\item $\Gamma$ is elliptic.
\item $\Gamma$ is virtually cyclic and contains a loxodromic element.
\item $\Gamma$ contains infinitely many pairwise independent loxodromic elements.
\end{enumerate}
\end{thm}

In particular, this theorem implies that if $\Gamma$ is acylindrically hyperbolic and $H\leq \Gamma$, then either $H$ is elliptic, virtually cyclic, or acylindrically hyperbolic. A special case of this theorem is that every element of $\Gamma$ is either elliptic or loxodromic, a result first proved by Bowditch \cite[Lemma 2.2]{Bow2}.  The following result is stated in \cite{DGO} in the more general context of loxodromic WPD elements.

\begin{lem}\cite[Lemma 6.5, Corollary 6.6]{DGO}\label{E(h)}
Let $\Gamma$ be a group acting acylindrically on a $\delta$--hyperbolic metric space and let $h\in \Gamma$ be a loxodromic element. Then $h$ is contained in a unique maximal virtually cyclic subgroup denoted $E_{\Gamma}(h)$. Furthermore, the following are equivalent for any $g\in \Gamma$:
\begin{enumerate}
\item $g\in E_{\Gamma}(h)$
\item $g^{-1}h^mg=h^k$ for some $m, k\in\Z\setminus\{0\}$.
\item $g^{-1}h^ng=h^{\pm n}$ for some $n\in\N$.
\end{enumerate}
In addition, for some $r\in\N$,
\[
E_{\Gamma}^+(h)=\{g\in \Gamma \mid \exists n\in \N \; g^{-1}h^ng=h^n\}=C_{\Gamma}(h^r).
\]
\end{lem}

In this article, we are interested in studying the set of homomorphisms from a finitely generated group $G$ to a family of groups. This perspective appeared  in the work of Jaligot--Sela \cite{JalSel} who studied homomorphisms to free products of groups.  The following definition introduces the main class(es) of groups which we study.

\begin{defn} \label{def:UAH}
Let $\mc{G}$ be a family of pairs $(\Gamma,X_\Gamma)$ where $\Gamma$ is a group and $X_\Gamma$ is a (geodesic) metric space upon which $\Gamma$ acts isometrically. We say that $\mc{G}$ is \emph{uniformly acylindrically hyperbolic} if there exists $\delta \ge 0$ and functions $R_\e$, $N_\e$ such that for each $(\Gamma, X_\Gamma)\in\mc{G}$, the space $X_\Gamma$ is a $\delta$--hyperbolic metric space on which $\Gamma$ acts acylindrically with acylindricity constants $R_\e$ and $N_\e$. Further, we require that the action of at least one $\Gamma$ on $X_\Gamma$ is non-elementary. We say that $\delta$ is the \emph{hyperbolicity constant} and $R_\e$, $N_\e$ are the \emph{acylindricity constants} for the family $\mc{G}$.
\end{defn}

\begin{rem}
We usually refer to elements of $\mc{G}$ just via the group $\Gamma$, leaving the metric space $X_\Gamma$ implicit.  However, it is possible to apply the techniques in this paper to a single group and a collection of spaces upon which this group acts in a uniformly acylindrically hyperbolic way.   An example of this approach appears in \cite{sela:acyl} where Sela studies the set of all $k$--acylindrical actions of a fixed finitely generated group on simplicial trees.  In this paper, however, we are mostly interested in the groups in $\mc{G}$, so we usually implicitly fix a single space $X_\Gamma$ for each group.  Therefore, we often use such expressions as `let $\Gamma \in \mc{G}$' and variants thereof.
\end{rem}

The requirement that some element of $\mc{G}$ admits a non-elementary action is included in Definition \ref{def:UAH} to ensure that $\Gamma$ is acylindrically hyperbolic (with an acylindrically hyperbolic action on $X_\Gamma$) if and only if $\{ (\Gamma, X_\Gamma) \}$ is uniformly acylindrically hyperbolic. Note that a uniformly acylindrically hyperbolic family $\mc{G}$ may contain groups which are not acylindrically hyperbolic. For example, $\mc{G}$ may consist of all subgroups of a fixed acylindrically hyperbolic group (all acting on the same space).

More examples of acylindrically hyperbolic families of groups include: any finite collection of acylindrically hyperbolic groups, the family of all groups which decompose as a non-trivial free product, more generally the class of all groups which admit a $k$--acylindrical action on a tree for a fixed $k$. This last example includes the family of all fundamental groups of compact $3$--manifolds. Indeed, if a compact $3$--manifold $M$ is reducible then $\pi_1(M)$ splits as a free product corresponding to the Kneser--Milnor decomposition of $M$ and we set $X_{\pi_1(M)}$ to be the Bass--Serre tree dual to this decomposition. If $M$ is irreducible, then we set $X_{\pi_1(M)}$ to be the Bass--Serre tree dual to the JSJ decomposition of $M$ (which may be a point if this decomposition is trivial). As noted in the introduction, the action of $\pi_1(M)$ on $X_{\pi_1(M)}$ is $4$--acylindrical. On the other hand, the family of all hyperbolic groups is not uniformly acylindrically hyperbolic because this would violate Lemma \ref{ufin} below.

The next result follows from the proof of \cite[Lemma 6.8]{Osi13}. In \cite{Osi13} it is stated for a single acylindrically hyperbolic group, but it is easy to see from the proof that the constant only depends on the hyperbolicity and acylindricity constants.
\begin{lem}\label{ufin}
Let $\mc{G}$ be a uniformly acylindrically hyperbolic family of groups. Then there exists $K$ such that for any $\Gamma\in\mc{G}$ and any loxodromic $h\in \Gamma$, every finite subgroup of $E_\Gamma(h)$ has order at most $K$.
\end{lem}

The following is \cite[Lemma 3.6]{Osi13}; it is clear from the proof that all the constants depend only on the hyperbolicity and acylindricity constants for $\Gamma$, so the result holds for our uniformly acylindrically hyperbolic family $\mc{G}$. 
\begin{lem}\label{acy2}
Let $\mc{G}$ be a uniformly acylindrically hyperbolic family.  For all $\e > 0$ there exist $N$ and $R$ such that for any $(\Gamma, X_\Gamma) \in \mc{G}$ and for any $x, y\in X_\Gamma$ with $d(x, y)\geq R$,
\[
|\{g\in \Gamma\;|\; d(x, gx)\leq \e, d(y, gy)\leq d(x,y)+\e\}|\leq N.
\] 
\end{lem}

\section{Equationally noetherian families and limit groups} \label{ss:EN and limit}

In this section we discuss what it means for a group to be equationally noetherian and some of the basic consequences of this property.  We then define what it means for a {\em family} $\mc{G}$ of groups to be equationally noetherian (see Definition \ref{d:en family} below).  We also introduce $\mc{G}$--limit groups, and relate properties of them to $\mc{G}$ being equationally noetherian.

Let $\Gamma$ be a group, $n$ a natural number and $F_n=F(x_1,...,x_n)$ the free group of rank $n$. Given $W\in\Gamma\ast F_n$ and $(g_1,...,g_n)\in\Gamma^n$, let $W(g_1,...,g_n)$ be the element of $\Gamma$ obtained by replacing each occurrence of $x_i$ with $g_i$ and each occurrence of $x_i^{-1}$ with $g_i^{-1}$.  For a subset $S\subseteq \Gamma\ast F_n$, define $V_\Gamma(S)$ by 
\[
V_\Gamma(S)=\{(g_1,...,g_n)\in\Gamma^n\;|\; W(g_1,...,g_n)=1 \;\forall W\in S\}.
\]

We think of $W\in F_n\ast\Gamma$ as defining an equation in the variables $x_1,..,x_n$, with constants from the group $\Gamma$, $S$ as a system of equations and $V_\Gamma(S)$ as the algebraic set of solutions to the system of equations.  There is a canonical copy of $F_n$ in $\Gamma \ast F_n$, and  an equation $W$ is {\em coefficient-free} if $W \in F_n$.

For any $(g_1,...,g_n)\in\Gamma^n$, there is an associated homomorphism $\eta\colon \Gamma\ast F_n\to \Gamma$ which is the identity on $\Gamma$ and which sends each $x_i$ to $g_i$. Then $(g_1,...,g_n)\in V_\Gamma(S)$ if and only if $S\subseteq \ker(\eta)$. In particular, for any homomorphism $\eta\colon \Gamma\ast F_n\to\Gamma$, $(\eta(x_1),...\eta(x_n))\in V_\Gamma(\ker(\eta))$. Note also that $(g_1,...,g_n)\in V_\Gamma(S)$ if and only if $x_i \mapsto g_i$ induces a homomorphism $(\Gamma\ast F_n)/\llangle S\rrangle\to \Gamma$. Hence there is a one to one correspondence between the algebraic set of solutions $V_\Gamma(S)$ to the system of equations $S$ and $\Hom(G, \Gamma)$ where $G=(\Gamma\ast F_n)/\llangle S\rrangle$. 

\begin{defn}
A group $\Gamma$ is \emph{equationally noetherian} if for any natural number $n$ and any $S\subseteq \Gamma\ast F_n$, there exists a finite $S_0\subseteq S$ such that $V_\Gamma(S_0)=V_\Gamma(S)$. 
\end{defn}

 The equationally noetherian property is a group-theoretic version of the conclusion of the Hilbert Basis Theorem.  The Hilbert Basis Theorem easily implies that all finitely generated linear groups are equationally noetherian (see, for example \cite{BMR1}). Using linearity, Guba showed that free groups are equationally noetherian \cite{Guba}. Sela showed that torsion-free hyperbolic groups are equationally noetherian \cite{sela:dio7}, and Sela's methods have been expanded to show that hyperbolic groups with torsion \cite{ReiWei}, toral relatively hyperbolic groups \cite{Gro}, and free products of equationally noetherian groups \cite{sela:DioX} are all equationally noetherian. Some solvable groups, including free solvable groups and rigid groups were shown to be equationally noetherian by Gupta and Romanovski{\u\i} \cite{GR, Romanovskii}.  In work in preparation, the first author proves that the mapping class group of a surface of finite type is equationally noetherian.

It is well-known that every equationally noetherian group is Hopfian. Since being equationally noetherian is closed under taking subgroups, any group which contains a non-Hopfian subgroup is not equationally noetherian. In particular there exist acylindrically hyperbolic groups, such as $\Z\ast \mathrm{BS}(2, 3)$, which are not equationally noetherian.

Next we define $\mc{G}$--limit groups, where $\mc{G}$ is a family of groups. Other versions of these definitions are often stated only for \emph{stable} sequences of homomorphisms. There is no real loss of generality here, since one can always turn a sequence of homomorphisms into a stable sequence by passing to a subsequence. However, we prefer to use to language of ultrafilters instead of frequently passing to subsequences.

\begin{defn}
An \emph{ultrafilter} is a finitely additive probability measure $\omega\colon 2^{\mathbb N}\to\{0, 1\}$. An ultrafilter $\omega$ is called \emph{non-principal} if $\omega(F)=0$ for any finite $F\subset\N$.  We say a subset $A\subseteq \mathbb N$ is \emph{$\omega$--large} if $\omega(A)=1$ and \emph{$\omega$--small} if $\omega(A)=0$. Given a statement $P$ which depends on some index $i$, we say that $P$ holds \emph{$\omega$--almost surely} if $P$ holds for an $\omega$--large set of indices. 
\end{defn}
For the remainder of this paper, we fix a non-principal ultrafilter $\omega$. When $G$ is a group and $\mc{G}$ is a family of groups, we denote $\bigcup\limits_{\Gamma\in\mc{G}} \Hom(G, \Gamma)$ by $\Hom(G, \mc{G})$.

\begin{defn}
Let $\mc{G}$ be a family of groups, $G$ a finitely generated group, and $(\phi_i)$ a sequence from $\Hom(G, \mc{G})$. The {\em $\omega$--kernel} of the sequence $(\phi_i)$ is 
$$\ker^{\omega}(\phi_i)=\{g\in G\;|\; \phi_i(g)=1 \; \text{$\omega$--almost surely}\}.$$ 

A {\em $\mc{G}$--limit group} is a group of the form $L=G/\ker^{\omega}(\phi_i)$ for some sequence $(\phi_i)$ from $\Hom(G, \mc{G})$. Let $\phi_{\infty}\colon G\onto L= G / \ker^{\omega}(\phi_i)$ denote the natural quotient map. We refer to $\phi_\infty$ as the {\em limit map} associated to the sequence $(\phi_i)$. We refer to the sequence $(\phi_i)$ as a {\em defining sequence} of homomorphisms for the $\mc{G}$--limit group $L$.
\end{defn}
If $\mc{G} = \{ \Gamma \}$, a single group, then we talk about {\em $\Gamma$--limit groups}.

\begin{defn}
A $\mc{G}$--limit group $L = G / \ker^{\omega}(\phi_i)$ where $\phi_i \co G \to \Gamma_i$ is \emph{strict} if either $\omega$--almost surely the groups $\Gamma_i$ are distinct (i.e. non-isomorphic), or else $\omega$--almost surely the $\Gamma_i$ are equal to a fixed $\Gamma \in \mc{G}$ and the homomorphisms $\phi_i$ and $\phi_j$ are $\omega$--almost surely distinct modulo conjugation in $\Gamma$. 
\end{defn}

Note that non-strict $\mc{G}$--limit groups are isomorphic to subgroups of elements of $\mc{G}$. Some strict $\mc{G}$--limit groups may also be (isomorphic to) subgroups of elements of $\mc{G}$.

The follow equivalences are well-known, for example $(1)\implies (4)$ appears in \cite[Proposition 3]{BMR1}. However, we are not aware of a reference in the literature which includes all of them. 
\begin{thm}\label{eqn}
Let $\omega$ be a non-principal ultrafilter.The following are equivalent for any group $\Gamma$:
\begin{enumerate}
\item For any natural number $n$ and any $S\subseteq F_n$, there is a finite subset $S_0\subseteq S$ such that $V_\Gamma(S_0)=V_\Gamma(S)$.
\item\label{factors} For any finitely generated group $G$ and any sequence of homomorphisms $(\phi_i \colon G\to \Gamma)$, $\phi_i$ $\omega$--almost surely factors through the limit map $\phi_\infty$.
\item For any finitely generated group $G$ and any sequence of homomorphisms $(\phi_i\colon G\to \Gamma)$, some $\phi_i$ factors through the limit map $\phi_\infty$.
\item $\Gamma$ is equationally noetherian.
\end{enumerate}
\end{thm}

\begin{proof}
$(2)\implies (3)$ and $(4)\implies (1)$ are trivial.

$(1)\implies (2)$: First, it suffices to assume $G=F_n$. Indeed, since $G$ is finitely generated we can fix some quotient map $h\colon F_n\twoheadrightarrow G$ and consider the sequence $(\eta_i)$ where $\eta_i =\phi_i\circ h$. If $\eta_i$ factors through $\eta_\infty$, then $\ker(\phi_\infty\circ h)=\ker(\eta_\infty)\subseteq\ker(\eta_i)=\ker(\phi_i\circ h)$, and hence $\ker(\phi_\infty)\subseteq \ker(\phi_i)$. 

Now assume $G=F_n$ and let $S=\ker^\omega(\phi_i)=\ker(\phi_\infty)$. By assumption, there exists finite subset $S_0\subseteq S$ such that $V_\Gamma(S_0)=V_\Gamma(S)$. But since $S_0$ is finite and $\omega$ is finitely additive, $\omega$--almost surely $S_0\subseteq \ker(\phi_i)$, hence $V_\Gamma(\ker(\phi_i))\subseteq V_\Gamma(S_0)=V_\Gamma(S)$ $\omega$--almost surely. For such $i$, $(\phi_i(x_1),...,\phi_i(x_n))\in V_\Gamma(\ker(\phi_i))\subseteq V_\Gamma(S)$, so if $W\in S=\ker(\phi_\infty)$, then $\phi_i(W)=W(\phi_i(x_1),...,\phi_i(x_n))=1$ in $\Gamma$, and hence $\ker(\phi_\infty)\subseteq \ker(\phi_i)$. Thus $\phi_i$ $\omega$--almost surely factors through $\phi_\infty$.

$(3)\implies (4)$: Let $S=\{W_1, W_2,...\}\subseteq \Gamma\ast F_n$, and let $S_i=\{W_1,...,W_i\}\subseteq S$. Suppose that $V_\Gamma(S_i)\neq V_\Gamma(S)$ for all $i$. Then there exists $\phi_i\colon \Gamma\ast F_n\to \Gamma$ such that $\phi_i$ restricts to the identity map on $\Gamma$ and $S_i\subseteq\ker(\phi_i)$ but $S\not\subset \ker(\phi_i)$. Notice that $S\subseteq \ker(\phi_\infty)=\ker^\omega(\phi_i)$ since $W_i\in \ker(\phi_j)$ for all $j\geq i$. By assumption, some $\phi_i|_{F_n}$ factors through $\phi_\infty|_{F_n}$, since $\phi_i$ and $\phi_\infty$ are both the identity on $\Gamma$ if follows that $\phi_i$ factors through $\phi_{\infty}$. But this implies that $S\subseteq \ker(\phi_\infty)\subseteq \ker(\phi_i)$, a contradiction. Therefore, $\Gamma$ is equationally noetherian.
\end{proof}

We now recall from the introduction how to extend these notions to define what it means for a family of groups $\mc{G}$ to be equationally noetherian.  Given $S\subseteq F_n$, let

\[
V_\mc{G}(S)=\{(g_1,...,g_n)\in\Gamma^n\;|\; \Gamma\in\mc{G} \text{ and } W(g_1,...,g_n)=1 \;\forall W\in S\}.
\]

\begin{repdefinition}{d:en family}
We say that a family $\mc{G}$ is an {\em equationally noetherian family of groups} if for any $n$ and any $S\subseteq F_n$, there exists a finite $S_0\subseteq S$ such that $V_\mc{G}(S_0)=V_\mc{G}(S)$.
\end{repdefinition}

\begin{rem}
In order to consider equations with coefficients for a family of groups, one would need to consider families of groups $\mc{G}$ such that each element of $\mc{G}$ is equipped with a fixed embedding of some fixed coefficient group $\Gamma$. We do not use this perspective in this paper.
\end{rem}

\begin{thm}\label{eqnfamily}
The following are equivalent for any family of groups $\mc{G}$:
\begin{enumerate}
\item  $\mc{G}$ is equationally noetherian.
\item  For any finitely generated group $G$ and any sequence of homomorphisms $(\phi_i)$ from $\Hom(G, \mc{G})$, $\phi_i$ factors through the limit map $\phi_\infty$ $\omega$--almost surely.
\item  For any finitely generated group $G$ and any sequence of homomorphisms $(\phi_i)$ from $\Hom(G, \mc{G})$, some $\phi_i$ factors through the limit map $\phi_\infty$.
\end{enumerate}
\end{thm}
\begin{proof}
The proof is identical to Theorem \ref{eqn} with $\Gamma$ replaced by groups $\Gamma_i$ from $\mc{G}$.
\end{proof}

The following are obvious.

\begin{lem}
 Any union of finitely many equationally noetherian families is equationally noetherian.
\end{lem}

\begin{lem}
If $\mc{G}$ is an equationally noetherian family of groups, then the collection of all subgroups of elements of $\mc{G}$ is equationally noetherian.
\end{lem}

\begin{lem}
If $\mc{G}$ is equationally noetherian and $\mc{G}^\prime\subseteq \mc{G}$, then $\mc{G}^\prime$ is equationally noetherian.
\end{lem}

The Hilbert Basis Theorem implies that for any field $K$, the linear group $GL_n(K)$ is equationally noetherian (see \cite[Theorem B1]{BMR1}). Combining this observation with the previous lemmas, we get some natural geometric families of groups, such as the family of fundamental groups of finite volume hyperbolic $3$--manifolds, more generally the family of all Kleinian groups, and the family of all subgroups of a fixed (linear) Lie group.

Theorem \ref{thm:eqnhopf} below shows that when $\mc{G}$ is equationally noetherian, there is no infinite descending chain of $\mc{G}$--limit groups. When $\mc{G}$ consists of a single group, this is \cite[Theorem 2.7]{OuldHoucine}. The proof for a family of group is essentially the same.  We first record the following elementary observation.

\begin{lem}\label{l:newgenset}
Suppose $L=G/\ker^\omega(\phi_i)$ and $\{y_1,...,y_n\}$ is a finite generating set for $L$. Let $\{x_1,...,x_n\}$ be a free generating set for $F_n$. Then there exists $\phi_i^\prime\in \Hom(F_n, \mc{G})$ such that $L=F_n/\ker^\omega(\phi_i^\prime)$ and $\phi^\prime_\infty(x_j)=y_j$ for $1\leq j\leq n$. 
\end{lem}
\begin{proof}
Let $\tilde{Z}$ be a finite generating set for $G$ and let $Z=\phi_\infty(\tilde{Z})$. Since $Z$ generates $L$, each $y_j$ can be written as a word $W_j$ in $Z$. Let $\widetilde{W_j}$ be the word obtained by replacing each letter $z$ in $W_j$ with some $\tilde{z}\in\phi_\infty^{-1}(z)\cap \tilde{Z}$. Define $\beta\colon F_n\to G$ by $\beta(x_j)=\widetilde{W_j}$, and let $\phi^\prime_i=\phi_i\circ\beta$. By definition, $\phi_\infty\circ\beta$ induces an injective map $F_n/\ker^\omega(\phi_i^\prime)\to L$, and this map is also surjective because its image contains the generating set $\{y_1,..., y_n\}$. Hence $F_n/\ker^\omega(\phi_i^\prime)$ is isomorphic to $L$. Furthermore, it follows from the construction that $\phi^\prime_\infty=\phi_\infty\circ \beta$, so $\phi^\prime_\infty(x_j)=W_j=y_j$.

\end{proof}

\begin{thm}\label{thm:eqnhopf}
Suppose $\mc{G}$ is an equationally noetherian family of groups $(\alpha_j\colon L_j\onto L_{j+1})$ is a sequence of surjective homomorphisms where each $L_j$ is a $\mc{G}$--limit group. Then $\alpha_j$ is an isomorphism for all but finitely many $j$.
\end{thm}

\begin{proof}
Let $(\phi_i^j\colon F_n\to\Gamma_i^j)$ be a defining sequence of maps for $L_j$, that is each $\Gamma_i^j\in\mc{G}$ and $F_n/\ker^\omega(\phi_i^j)=L_j$. We further assume that these maps are chosen such that the following diagram commutes for all $j\geq 1$:

\ 
\centerline{
\xymatrix{
F_n \ar@{>}[r] ^{\phi_\infty^j} \ar@{>}[dr]_{\phi^{j+1}_\infty} & L_j \ar@{>}[d]^{\alpha_j} \\
& L_{j+1}
}}\\\\

We can find such defining sequences by Lemma \ref{l:newgenset}. In particular, this means that $\ker(\phi_\infty^1)\subseteq\ker(\phi_\infty^2)\subseteq...$, that is the kernels of the maps $\phi^j_\infty$ form an increasing sequence. Let $S=\cup_{j=1}^\infty\ker(\phi_\infty^j)$, and let $S_0$ be a finite subset of $S$ such that $V_{\mc{G}}(S_0)=V_{\mc{G}}(S)$. Since $S_0$ is finite, there is some $k$ such that $S_0\subseteq \cup_{j=1}^k\ker(\phi_\infty^j)=\ker(\phi_\infty^k)$. We show that $S\subseteq\ker(\phi_\infty^k)$ which implies that for all $j\geq k$, $\alpha_j$ is an isomorphism.

Since $S_0\subseteq\ker(\phi_\infty^k)$ and $S_0$ is finite, $\omega$--almost surely $S_0\subseteq \ker(\phi_i^k)$. For each such $i$,
\[
V_\mc{G}(\ker(\phi_i^k))\subseteq V_\mc{G}(S_0)=V_{\mc{G}}(S)
\]
Thus, for such $i$ $(\phi_i^k(x_1),...,\phi_i^k(x_n))\in V_{\mc{G}}(S)$ hence if $W\in S$, $\phi_i^k(W)=W(\phi_i^k(x_1),...,\phi_i^k(x_n))=1$. Since this holds $\omega$--almost surely, $W\in\ker(\phi_\infty^k)$. Thus $S\subseteq\ker(\phi_\infty^k)$.
\end{proof}

Note that since each element of $\mc{G}$ is a $\mc{G}$--limit group, the above theorem also implies that there is no infinite descending chain of proper epimorphisms of the form 
\[
\Gamma_1\onto\Gamma_2\onto \cdots
\]
with each $\Gamma_i\in\mc{G}$. This can be viewed as a generalized Hopfian property for families of groups, hence the above generalizes the well-known fact that every equationally noetherian group is Hopfian. Similarly, applying Theorem \ref{thm:eqnhopf} in the case where each $L_i$ is isomorphic to a fixed group yields the following corollary (cf. \cite[Corollary 2.9]{OuldHoucine}).

\begin{cor}\label{c:limithopfian}
If $\mc{G}$ is an equationally noetherian family of groups, then every $\mc{G}$--limit group is Hopfian.  In particular, every element of $\mc{G}$ is Hopfian.
\end{cor}

There are, however, examples of Hopfian groups which are not equationally noetherian, for example any simple group which contains $BS(2, 3)$. Similarly, there are families of groups which satisfy the above generalization of the Hopfian property but are not equationally noetherian families. One such example is the family of all finite groups.  We thank Henry Wilton for pointing out the following proof  (we had in mind a more complicated proof).

\begin{ex}\label{l:finite not en}
The family of all finite groups is not equationally noetherian.
\end{ex}
\begin{proof}
In an answer to \cite{YC-Mathoverflow}, Yves de Cornulier (user YCor) constructs a non-Hopfian group as a limit of virtually free groups.  It is straightforward to see, using the fact that virtually free groups are residually finite, that such a group is also a limit of finite groups. Hence the family of all finite groups cannot be equationally noetherian by Corollary \ref{c:limithopfian}.
\end{proof}

 It follows that any family of groups which contains all finite groups is not equationally noetherian; for example, the collection of all residually finite groups, the collection of all hyperbolic groups, and the collection of all linear groups are not equationally noetherian.

\begin{ex}
The collection of all torsion-free hyperbolic groups is not equationally noetherian.
\end{ex}
\begin{proof}
There is a sequence of torsion-free hyperbolic groups $\Gamma_1\onto\Gamma_2\onto...$ such that the direct limit of this sequence $\Gamma_\infty$ is an infinite, non-abelian group with the property that every proper, non-trivial subgroup of $\Gamma_\infty$ is infinite cyclic  \cite{gromov:hyp, Ols}. In particular, $\Gamma_\infty$ is not hyperbolic, thus the family of all torsion-free hyperbolic groups is not equationally noetherian by Theorem \ref{thm:eqnhopf}.
\end{proof}

These examples show that a family of groups being equationally noetherian is a much stronger property than each of the groups in the family being equationally noetherian. 

\begin{ex}  
Suppose that $G$ is a Hopfian group which is not equationally noetherian.  For example, $G$ might be an infinite simple group containing $BS(2,3)$, as mentioned above.  Simplicity clearly implies $G$ is Hopfian, while $G$ cannot be equationally noetherian since it contains non-equationally noetherian (even non-Hopfian) subgroups.

If $\Gamma=G\ast\mathbb Z$, then $\Gamma$ is hyperbolic relative to a Hopfian subgroup, but $\Gamma$ is not equationally noetherian. This example indicates that our techniques (inspired by Sela's) will not be able to prove that a group  hyperbolic relative to Hopfian subgroups is Hopfian. 
\end{ex}

\begin{ques}
Suppose that $G$ is a Hopfian group and $\Gamma = G \ast \mathbb Z$.  Is $\Gamma$ Hopfian?
\end{ques}

Next we make a few more observations about $\mc{G}$--limit groups in the case where $\mc{G}$ is equationally noetherian. A group $G$ is called \emph{fully residually $\mc{G}$} if for any finite subset $F\subseteq G\setminus\{1\}$, there exists $\phi\in \Hom(G, \mc{G})$ such that $\ker(\phi)\cap F=\emptyset$. The group $G$ is called \emph{residually $\mc{G}$} if this holds whenever $F$ consists of a single element of $G\setminus\{1\}$.

The following is well-known when $\mc{G}$ consists of a single group, and the proof is the same in general.
\begin{lem}
Let $\mc{G}$ be a family of groups.  Then any finitely generated fully residually $\mc{G}$ group is a $\mc{G}$--limit group. If $\mc{G}$ is an equationally noetherian family then any $\mc{G}$--limit group is fully residually $\mc{G}$.
\end{lem}
\begin{proof}
If $L$ is finitely generated and fully residually $\mc{G}$, then there exists a sequence $(\phi_i)$ from $\Hom(L, \mc{G})$ such that $\phi_i$ is injective on the ball of radius $i$ in $L$, hence $\ker^{\omega}(\phi_i)=\{1\}$. Thus $L=L/\ker^{\omega}(\phi_i)$ is a $\mc{G}$--limit group.

Now suppose $L=G/\ker^\omega(\phi_i)$, for maps $\phi_i \co G \to \Gamma_i$ (with $\Gamma_i \in \mc{G}$). Then by Theorem \ref{eqn}.\eqref{factors} $\phi_i$ $\omega$--almost surely factors through $\phi_\infty$, hence there are maps $\eta_i \co L \to \Gamma_i$ such that $\phi_i=\eta_i\circ\phi_\infty$. Then $\ker(\phi_\infty)=\ker^\omega(\phi_i)=\ker^\omega(\eta_i\circ\phi_\infty)$, thus $\ker^\omega(\eta_i)=\{1\}$. It follows that for any finite subset $F\subseteq L\setminus\{1\}$, there exists some $\eta_i$ such that $F\cap \ker(\eta_i)=\emptyset$, hence $L$ is fully residually $\mc{G}$.
\end{proof}

In particular when $\mc{G}$ is an equationally noetherian family a finitely generated group $L$ is a $\mc{G}$--limit group if and only if there exists a sequence $(\phi_i)$ from $\Hom(L, \mc{G})$ with $\ker^\omega(\phi_i)=\{1\}$.

Lemma \ref{lem:relfpen} below is needed in Sections \ref{s:short quotient} and \ref{s:rel hyp}.
Given a group $G$ and subgroups $P_1,..., P_n$ of $G$, $G$ is said to be \emph{finitely generated relative to the subgroups $\{P_1,...,P_n\}$} if there is a finite set $S\subseteq G$ such that $S\cup P_1\cup...\cup P_n$ generates $G$. In this case, there is a natural surjective homomorphism $F(S)\ast P_1\ast...\ast P_n\onto G$ (where $F(S)$ is the free group on $S$) which is the identity when restricted to $S$ and to each $P_i$. If the kernel of this homomorphism is normally generated by a finite set, then $G$ is said to be \emph{finitely presented relative to subgroups $\{P_1,...,P_n\}$}.

\begin{lem}\label{lem:relfpen}
Suppose that the sequence $(\phi_i)$ from $\Hom(G, \mc{G})$ is such that the limit group $L=G/\ker^\omega(\phi_i)$ is finitely presented relative to subgroups $\{P_1,...,P_n\}$. Suppose further that for each $P_j$, there is subgroup $\widetilde{P}_j \le G$ which $\phi_\infty$  maps onto $P_j$ such that $\phi_i|_{\widetilde{P}_j}$ factors through $\phi_\infty|_{\widetilde{P}_j}$ $\omega$--almost surely. Then $\phi_i$ factors through $\phi_\infty$ $\omega$--almost surely.
\end{lem}
\begin{proof}
Let $S$ be a finite generating set of $G$ and let $F(S)$ denote the free group on $S$. Then there are natural surjective maps $\alpha\colon F(S)\ast(\ast_{i=1}^n \widetilde{P}_i)\onto G$, $\beta\colon F(S)\ast(\ast_{i=1}^n P_i)\onto L$, and $\gamma\colon F(S)\ast(\ast_{i=1}^n \widetilde{P}_i)\onto F(S)\ast(\ast_{i=1}^n P_i)$ which fit into the commutative diagram:
\[
\begin{CD}
F(S)\ast(\ast_{j=1}^n \widetilde{P}_j) @>{\alpha}>>G\\
@V{\gamma}VV @VV{\phi_\infty}V\\
 F(S)\ast(\ast_{j=1}^n P_j)@>{\beta}>> L
\end{CD}
\]

The assumption that $L$ is finitely presented relative to $\{P_1,...,P_n\}$ means that $\ker(\beta)$ is normally generated by a finite set $R$. Let $\widetilde{R}$ be a finite subset of $F(S)\ast(\ast_{j=1}^n \widetilde{P}_j)$ which $\gamma$ maps onto $R$.  Also, $\ker(\gamma)$ is normally generated by $\bigcup\limits_{j=1}^n \ker(\phi_\infty|_{\widetilde{P}_j})$. Thus, $\ker(\beta\circ\gamma)=\ker(\phi_\infty\circ\alpha)$ is normally generated by $\widetilde{R}\bigcup\left(\bigcup\limits_{j=1}^n \ker(\phi_\infty|_{\widetilde{P}_j})\right)$

Clearly $\ker(\phi_\infty\circ\alpha)=\ker^\omega(\phi_i\circ\alpha)$.  Since $\widetilde{R}$ is finite, $\widetilde{R}\subseteq \ker(\phi_i\circ\alpha)$ $\omega$--almost surely. By assumption, $\omega$--almost surely $\ker(\phi_\infty|_{\widetilde{P}_j})\subseteq \ker(\phi_i|_{\widetilde{P}_j})$.

Since $\ker(\phi_\infty\circ\alpha)$ is normally generated by $\widetilde{R}\bigcup(\bigcup_{j=1}^n \ker(\phi_\infty|_{\widetilde{P}_j}))$, we see $\ker(\phi_\infty\circ\alpha)\subseteq\ker(\phi_i\circ\alpha)$ $\omega$--almost surely. Since $\alpha$ is surjective this implies that $\omega$--almost surely $\ker(\phi_\infty)\subseteq\ker(\phi_i)$. Hence $\omega$--almost surely $\phi_i$ factors through $\phi_\infty$.
\end{proof}

\section{Limit actions and the Rips machine}\label{s:limits}

Throughout this section, let $\mc{G}$ be a fixed  uniformly acylindrically hyperbolic family of groups and for each $\Gamma\in\mc{G}$, let $X_\Gamma$ denote the associated hyperbolic metric space on which $\Gamma$ acts acylindrically. In this section we consider {\em divergent} $\mc{G}$--limit groups (see Definition \ref{def:divergent}), which we show come equipped with a limiting action on an $\R$--tree.  We investigate the basic properties of this action.

\subsection{Limiting $\R$--trees}

Recall that we fixed a non-principal ultrafilter $\omega$.  We next discuss some basic properties of $\omega$--limits.  See \cite{vdD-W} for more information.  Given a sequence of real numbers $(a_i)$, we define $\lim^\omega(a_i)=a$ if for all $\e>0$, $\omega \left( \{i\;|\; |a-a_i|<\e\} \right) = 1$. Similarly, $\lim^\omega(a_i)=\infty$ if for all $N>0$, $\omega \left( \{i\;|\; a_i>N\} \right) = 1$. It is easy to see that the $\omega$--limit of the sequence $(a_i)$ is equal to the ordinary limit of some subsequence.

Let $(X_i, d_i)$ be a sequence of metric spaces together with fixed base points $o_i\in X_i$.  On the product $\prod X_i$ we define an equivalence relation $\sim$ by saying $(x_i)\sim(y_i)$
if and only if $\lim^{\omega}d_i(x_i, y_i)=0$. Then the \emph{ultra-limit of the metric spaces $X_i$ with basepoints $o_i$} is defined as
\[
\lim\phantom{ }^\omega(X_i, o_i)= \frac{ \left\{(x_i)\in\prod X_i\;|\; \lim\phantom{ }^\omega d_i(o_i, x_i)<\infty \right\} }{\sim}	.
\]
This space has the metric $d$ defined by $d((x_i), (y_i))=\lim^\omega d_i(x_i, y_i)$. It is straightforward to check that this is a well-defined metric. For a sequence $(x_i)\in\prod X_i$, we say the sequence is \emph{visible} if $\lim^\omega d_i(o_i, x_i)<\infty$. In this case we often denote the image of $(x_i)$ in the ultra-limit by $\lim^\omega x_i$.

It is well-known that if each $X_i$ is a geodesic metric space then $\lim^\omega(X_i, o_i)$ is a geodesic metric space (see, for example, the proof of \cite[Proposition 4.2(b)]{vdD-W} which works in this generality).  The following lemma goes back in spirit to a construction of Paulin \cite{Paulin} (cf. \cite{bestvina:degenerations}, \cite{bridson-swarup}).  Using the modern construction involving asymptotic cones it is straightforward.

\begin{lem}\label{lem:hypultralimit}
Suppose each $X_i$ is a $\delta_i$--hyperbolic metric space and $\lim^\omega \delta_i<\infty$. Then $\lim^\omega (X_i, o_i)$ is a $\lim^\omega \delta_i$--hyperbolic metric space.
\end{lem}

Suppose $G$ is generated by a finite set $S$. Let $(\phi_i\colon G\to \Gamma_i)$ be a sequence of homomorphisms where $\Gamma_i \in \mc{G}$, and suppose that $X_i := X_{\Gamma_i}$ has metric $d_i$.  Define a {\em scaling factor}
\[
\| \phi_i \|=\inf_{x\in X_i}\max_{s\in S} d_i(x, \phi_i(s)x).
\]

\begin{defn} \label{def:divergent}
The sequence $(\phi_i)$ is called {\em divergent} if $\lim^{\omega}\| \phi_i \|=\infty$. In this case, the associated $\mc{G}$--limit group $L=G / \ker^\omega(\phi_i)$ is called a \emph{divergent $\mc{G}$--limit group}.
\end{defn}

Note that if $\mc{G} = \{ \Gamma \}$, where $\Gamma$ is a single hyperbolic group acting on its Cayley graph, then a sequence of homomorphisms $(\phi_i \co G \to \Gamma)$ is divergent if and only if the $\phi_i$ do not lie in finitely many conjugacy classes.  Thus, in this case the notion of strict $\Gamma$--limit groups and divergent $\Gamma$--limit groups are the same.  On the other hand, in general the difference between strict and divergent $\mc{G}$--limit groups causes many complications.  We note that this is unavoidable, since for any group $H$ there is an acylindrically hyperbolic group $\Gamma$ which contains $H$ as an elliptic subgroup (for example, $\Gamma=\Z\ast H$). This means that every $H$--limit group $L$ occurs as a non-divergent $\Gamma$--limit group, and if $L$ is a strict $H$--limit group it is also be a strict $\Gamma$--limit group. Hence, without additional assumptions on $\mc{G}$ non-divergent $\mc{G}$--limit groups can be completely arbitrary.

Fix a divergent $\mc{G}$--limit group $L$, with defining sequence $(\phi_i \co G \to \Gamma_i)$.
 For an element $g\in L$, let $\widetilde{g}$ denote some element of $\phi_\infty^{-1}(g)$. The element $\widetilde{g}$ is not unique, but the statements which use this notation never depend on the choice. Similarly, for a subgroup $H\leq L$, let $\widetilde{H}$ be a subgroup of $G$ which maps surjectively onto $H$. When $H$ is finitely generated we also choose $\widetilde{H}$ to be finitely generated, by lifting a finite generating set and considering the subgroup generated by these lifts.

\begin{defn} \label{def:SE and LSE}
Given a subgroup $H\leq L$, we call $H$ \emph{stably elliptic} (with respect to the defining sequence $(\phi_i)$) if the action of $\phi_i(\widetilde{H})$ on $X_i$ is $\omega$--almost surely elliptic, and $H$ is \emph{locally stably elliptic} (with respect to $(\phi_i)$) if this holds for every finitely generated subgroup of $H$. 

Let $\mathcal{SE}(L)$ and $\mathcal{LSE}(L)$ denote the set of stably elliptic and locally stably elliptic subgroups of $L$  (with respect to $(\phi_i)$).  Let $\mathcal{NLSE}(L)$ denote the set of subgroups of $L$ which are not locally stably elliptic  (with respect to $(\phi_i)$).
\end{defn}
The sets $\mathcal{SE}(L)$, $\mathcal{LSE}(L)$, and $\mathcal{NLSE}(L)$ all depend on both $L$ and on the defining sequence $(\phi_i)$, but since the defining sequence is usually fixed we suppress this from the notation.

If $H$ is finitely generated then stably elliptic and locally stably elliptic are equivalent notions. However, for non-finitely generated groups stably elliptic implies locally stably elliptic, but there may be locally stably elliptic subgroups which are not stably elliptic.

The following theorem is a standard application of well-known methods which go back to the work of Paulin \cite{Paulin}.
\begin{thm} \label{t:limiting R-tree}
Suppose that $G$ is a finitely generated group and that $(\phi_i)$ is a divergent sequence from $\Hom(G, \mc{G})$.  Then $G$ admits a non-trivial, minimal action on an $\R$--tree T. Furthermore, the action of $G$ on $T$ induces a non-trivial, minimal action of the corresponding $\mc{G}$--limit group $L=G/\ker^\omega(\phi_i)$ on $T$.

\end{thm}
\begin{proof}
Recall that $\| \phi_i \|=\inf\limits_{x\in X_i}\max\limits_{s\in S} d_i(x, \phi_i(s)x)$, where $S$ is a finite generating set for $G$. In most cases we are interested in, we can choose a point $o_i\in X_i$ which realizes this infimum. If no such point exists, we let $o_i\in X_i$ such that $\max\limits_{s\in S} d_i(o_i, \phi_i(s)o_i)\leq \|\phi_i\|+\frac1i$.

Since each $X_i$ is $\delta$--hyperbolic, after rescaling the metric by $\frac{1}{\|\phi_i\|}$ we get a space which is $\delta/\|\phi_i\|$--hyperbolic. We denote the rescaled space by $X_i/\|\phi_i\|$.  By Lemma \ref{lem:hypultralimit}, $\lim^\omega (X_i/ \|\phi_i \|, o_i)$ is $0$--hyperbolic and hence an $\mathbb R$--tree. The action of $G$ on $\lim^\omega (X_i/ \|\phi_i \|, o_i)$  is defined by $g\lim^\omega x_i=\lim^\omega\phi_i(g)x_i$; it is straightforward to check that this is a well-defined isometric action. Now given a point $x=\lim^\omega x_i\in \lim^\omega (X_i/ \|\phi_i \|, o_i)$ , for each $x_i$ there is a generator of $G$  $s_i\in S$ such that $d_i(x_i, \phi_i(s_i)x_i)\geq\|\phi_i\|$. Since $S$ is finite, $\omega$--almost surely $s_i$ is equal to some fixed $s\in S$. It follows that $d(x, sx)\geq 1$, hence $sx\neq x$. This shows that the action of $G$ on $\lim^\omega (X_i/ \|\phi_i \|, o_i)$  is non-trivial.

We now get the desired tree $T$ by choosing a minimal $G$--invariant subtree of $\lim^\omega (X_i/ \|\phi_i \|, o_i)$. (Note that unless $\lim^\omega (X_i/ \|\phi_i \|, o_i)$ is a line, the action of $G$ on $\lim^\omega (X_i/ \|\phi_i \|, o_i)$ cannot be minimal since $G$ is finitely generated and $\lim^\omega (X_i/ \|\phi_i \|, o_i)$ is an $\R$--tree of uncountable valence).  

Finally, note that each element of $\ker^\omega(\phi_i)$ acts trivially on $T$. Hence the action of $G$ on $T$ induces an action of $L=G/\ker^\omega(\phi_i)$ on $T$. 
\end{proof}

\begin{defn} \cite[Definition 3.1]{BF:stable} \label{def:stable}
When a group $G$ is acting on an $\R$--tree $T$, an arc of $T$ is \emph{stable} if the stabilizer of every non-degenerate subarc is equal to the stabilizer of the whole arc.  Otherwise, the arc is {\em unstable}.
\end{defn}

The following lemma is the key ingredient which allows us to apply the Rips machine to the actions of divergent limit groups on the associated $\R$--trees. Parts of it are proved in \cite[Theorem 2.2.1]{Val}, and our proof is similar to the proof of \cite[Theorem 1.16]{ReiWei}. 

\begin{convention}
In the following lemma, and for the remainder of this section, \emph{uniformly finite} means that it is bounded by a constant which depends only on $\delta$ and the acylindricity functions $N_\epsilon$ and $R_\epsilon$.  
\end{convention}
In the next section (see Standing Assumption \ref{sa:uf}), we make a choice of number $C$ so that all uniformly finite numbers from this section are less than $C$. 

\begin{lem}(Stability Lemma) \label{l:stability lemma}
Let $L$ be a divergent $\mc{G}$--limit group and $T$ the corresponding $\mathbb R$--tree given by Theorem \ref{t:limiting R-tree}. Then the action of $L$ on $T$ satisfies:
\begin{enumerate}
\item If $H\leq L$ stabilizes a non-trivial arc of T, then $H$ is (uniformly finite)-by-abelian.
\item If $H\leq L$ preserves a line in $T$ and fixes its ends, then $H$ is (uniformly finite)-by-abelian. 
\item Stabilizers of tripods are uniformly finite. Consequently if $T$ is not a line, then the kernel of the action of $L$ on $T$ is uniformly finite.
\item The stabilizer of an unstable arc is uniformly finite.
\item Every finitely generated element of $\mathcal{SE}(L)$ fixes a point in $T$.
\item If $H \in \mathcal{LSE}(L)$ fixes a non-degenerate arc in $T$ then $H$ is uniformly finite.
\end{enumerate}
\end{lem}

\begin{rem}
If the elements of $\mc{G}$ are torsion-free, then $L$ is torsion-free which allows one to replace ``uniformly finite" with ``trivial" everywhere in this lemma. 
\end{rem}

\begin{proof}
Throughout the proof, we suppose that $L = G / \ker^\omega (\phi_i)$ for some finitely generated group $G$ and some divergent sequence of homomorphisms $(\phi_i \co G \to \Gamma_i)$ with $\Gamma_i \in \mc{G}$. Let $(X_i,d_i)$ be the $\delta$--hyperbolic space upon which $\Gamma_i$ acts acylindrically. 

(1)  Let $H\leq L$ be a subgroup which stabilizes a non-trivial arc $[x, y]$ of $T$, where $x=\lim^\omega x_i$ and $y=\lim^\omega y_i$. Fix $\e_1=28\delta$, and let $N_1=N_{\e_1}$ and $R_1=R_{\e_1}$ be given by Definition \ref{acy}. Let $g_1,...,g_{N_1+1}, h_1,...,h_{N_1+1}$ be elements of $H$, and choose lifts $\widetilde{g_j}, \widetilde{h_j}$ to $G$. Let $F=\{g_j, h_j\;|\;1\leq j\leq N_1+1\}$ and $\tilde{F}=\{\tilde{g_j}, \tilde{h_j}\;|\;1\leq j\leq N_1+1\}$. 
We have
\[	\lim\phantom{ }^\omega \frac{d_i(x_i,y_i)}{\| \phi_i \|} \ne 0 .	\]
On the other hand, for each $j \in \{ 1 ,\ldots , N_1+1 \}$ we have
\[	\lim\phantom{ }^\omega \frac{d_i(x_i,\phi_i(\widetilde{g_j})x_i)}{\| \phi_i \|} =	\lim\phantom{ }^\omega \frac{d_i(x_i,\phi_i(\widetilde{h_j})x_i)}{\| \phi_i \|} = 0	.	\]
Therefore, for each $\tilde{f}\in\tilde{F}$, $\omega$--almost surely we have
\begin{align*}
d_i(x_i, \phi_i(\tilde{f})x_i) & < \frac{1}{100}   d_i(x_i, y_i), \\
d_i(y_i, \phi_i(\tilde{f})y_i) & < \frac{1}{100} d_i(x_i, y_i), \mbox{ and }\\
R_1 & < \frac{1}{100} d_i(x_i, y_i).
\end{align*}

Now choose points $p_i, q_i\in[x_i, y_i]$ such that for any $\tilde{f}\in \tilde{F}$, $\omega$--almost surely we have 
\begin{align*}
 d_i(p_i, x_i) &\geq d_i(x_i, \phi_i({\tilde{f}})x_i)+2\delta ,\\
 d_i(p_i, y_i) &>d_i(q_i, y_i)\geq d_i(y_i, \phi_i(\tilde{f})y_i)+2\delta, \mbox{ and }\\
 d_i(p_i, q_i) & \geq R_1.
\end{align*}
From this it follows that $\omega$--almost surely, $d_i(p_i, \phi_i([\widetilde{g}_j, \widetilde{h}_j])p_i)\leq \e_1$ and $d_i(q_i, \phi_i([\widetilde{g}_j, \widetilde{h}_j])q_i)\leq\e_1$ for all $1\leq  j\leq N_1+1$, hence by the definition of acylindricity the commutators are not all distinct. It follows that $H$ contains at most $N_1$ commutators.  Note that $g^h = g[g,h]$, so if there are at most $N_1$ commutators in $H$ then any element of $H$ has at most $N_1$ distinct conjugates.  Thus, $H$ is {\em $N_1$--BFC} in the language of \cite{Wie}, which means that by \cite[Theorem 4.7]{Wie} there is some $K$ depending only on $N_1$ so that $|H^\prime| \le K$.  The statement (1) follows immediately.

(2) Let $Y$ be the line in $T$ that is preserved by $H$. Choose a finite set $F\subset H$ as in the proof of (1). For such $F$, we can choose $x, y\in Y$ such that $d(x, y)$ is sufficiently large compared to the translation length of any element of $F$. The proof can then be completed in the same way as (1).

(3) Suppose $H$ stabilizes a non-trivial tripod with vertices $x=\lim^\omega x_i$, $y=\lim^\omega y_i$, and $z=\lim^\omega z_i$. Let $\e_2 = 14\delta$, and let $N_2 = N_{\e_2}$ and $R_2 = R_{\e_2}$ be the associated acylindricity constants from Definition \ref{acy}. Choose a geodesic $[x_i,y_i]$ in $X_i$ and let $c_i$ be the point on $[x_i, y_i]$ that minimizes the distance to $z_i$. Then for any $h \in H$ and any lift $\widetilde{h}$ of $h$, $\omega$--almost surely we have (see, for example, \cite[Lemma 4.1]{RipsSela94})
\[	d_i(c_i, \phi_i(\widetilde{h})c_i)\leq 10\delta.	\] 
As in the proof of (1), we can choose $p_i\in[x_i, c_i]$ such that $d(x_i, p_i)\geq d(x_i, \phi_i(\widetilde{h}x_i))+2\delta$ and $d(p_i, c_i)\geq \max\{R_2, 12\delta\}$. Then $d(\phi_i(\widetilde{h})p_i, [x_i, c_i])\leq 2\delta$ and $d_i(p_i, c_i)-10\delta\leq d_i(\phi_i(\widetilde{h})p_i, c_i)\leq d_i(p_i, c_i)+10\delta$. Hence $d_i(p_i, \phi_i(\widetilde{h})p_i)\leq 14\delta$.
Repeating this argument with $N_2+1$ elements of $H$, we find repetition amongst the elements and by acylindricity we have $|H|\leq N_2$.

(4) This is the same as \cite[Theorem 1.16(4)]{ReiWei} (which in turn is very similar to \cite[Proposition 4.2]{RipsSela94}). Suppose $[a, b]\subset[x, y]$ and $g\in \Stab_L([a,b])$ but $gx\neq x$. Then $(x, y, gx)$ are the endpoints of a non-trivial tripod in $T$. If $h\in \Stab_L([x, y])$, then $hgx=[h, g]ghx=[h, g]gx$. Now $[h, g]$ belongs to the commutator subgroup of $\Stab_L([a, b])$, which is uniformly finite by (1). Hence the $\Stab_L([x, y])$--orbit of $gx$ is uniformly finite, and so a uniformly finite index subgroup of $\Stab_L([x, y])$ fixes $gx$ and hence fixes the whole tripod spanned by $(x, y, gx)$. Thus $\Stab_L([x, y])$ is uniformly finite by (3).

(5) Let $H=\langle h_1,..., h_n\rangle $ be a finitely generated element of $\mathcal{SE}(L)$, and let $o_i\in X_i$ be the basepoint. Fix $\widetilde{h_j}\in\phi_\infty^{-1}(h_j)$ and let $\widetilde{H}=\langle\widetilde{h_1},...,\widetilde{h_n}\rangle$.  Let $M_j$ be the word length of $\widetilde{h_j}$ with respect to the given generators of $G$ and let $M=\max_{1\leq j\leq n}\{ M_j\}$. Since each generator of $G$ moves $o_i$ by at most $\|\phi_i\|+\frac1i$, by induction on word length we have $d_i(o_i, \phi_i(\widetilde{h_j})o_i)\leq M_j(\|\phi_i\|+\frac1i)\leq M(\|\phi_i\|+\frac1i)$.

Since each $\phi_i(\widetilde{H})$ is $\omega$--almost surely elliptic, the orbit $\phi_i(\widetilde{H})\cdot o_i$ is bounded and hence has a $1$--quasi-center $q_i$. Since $\phi_i(\widetilde{H})$ preserves the set of $1$--quasi-centers, Lemma \ref{lem:quasicenter} implies that the $\phi_i(\widetilde{H})$--orbit of $q_i$ is bounded by $4\delta+2$. If the sequence $(q_i)$ is visible in the ultra-limit, then the point $\lim^\omega q_i$ is fixed by $H$ (note that this is true even if $H$ is not finitely generated). In case  the sequence $(q_i)$ is not visible we work with another sequence $(p_i)$ constructed below.

Without loss of generality, we suppose the indices are chosen such that $\omega$--almost surely 
\[	d_i(o_i, \phi_i(\widetilde{h_1})o_i)=\max_{1\leq j\leq n}\{d_i(o_i, \phi_i(\widetilde{h_j})o_i)\}.\]
Let $p_i$ be the point on $[o_i, q_i]$ such that $d_i(o_i, p_i)=d_i(o_i, \phi_i(\widetilde{h}_1)o_i)+2\delta+1$ or set $p_i=q_i$ if no such point exists. In particular, $d_i(o_i, p_i)\leq M\|\phi_i\|+M+2\delta+1$, so $\lim^\omega \frac{d_i(o_i, p_i)}{\|\phi_i\|}\leq M$, thus $p=\lim^\omega(p_i)\in \lim^\omega(X/\|\phi_i\|, o_i)$. Now fix some $1\leq j\leq n$. Consider a geodesic quadrilateral in $X_i$ with vertices (in cyclic order) being $\{ o_i, q_i, \phi_i(\widetilde{h_j})q_i , \phi_i(\widetilde{h_j})o_i \}$.  Note that the side between $q_i$ and $\phi_i(\widetilde{h_j})q_i$ has length at most $4\delta + 2$.  Moreover, this geodesic quadrilateral is 
$2\delta$--slim, so there is a point $p_i^\prime \in [o_i,q_i]$ such that $d_i(p_i^\prime, \phi_i(\widetilde{h_j})p_i) \le 6\delta + 2$ (since $\phi_i(\widetilde{h_j})p_i$ cannot be near $[o_i,\phi_i(\widetilde{h_j})o_i]$ by choice of $p_i$).

Now, $d_i(q_i, \phi_i(\widetilde{h_j})q_i)\leq 4\delta+2$ and $d_i(p_i, q_i)=d_i(\phi_i(\widetilde{h_j})p_i, \phi_i(\widetilde{h_j})q_i)$, so combining these inequalities we get:
\[	\left| d_i(p_i,q_i) - d_i(p_i^\prime, q_i) \right| < 10\delta + 4	\]

Since $p_i$ and $p_i^\prime$ both belong to $[o_i, q_i]$, we get that $d_i(p_i, p_i^\prime)\leq 10\delta+4$, and hence $d_i(p_i, \phi_i(\widetilde{h_j})p_i)\leq 16\delta+6$. Therefore, the point $p$ is fixed by  each $h_j$, and thus fixed by $H=\langle h_1,...,h_n\rangle$. Finally, note that $T$ is invariant under the action of $H$, so since $H$ fixes $p$ it also fixes the unique geodesic from $p$ to $T$. In particular, the intersection of this geodesic and $T$ is a point of $T$ which is fixed by $H$.

(6)- Suppose that $H \in \mathcal{LSE}(L)$ and $H$ fixes a non-trivial arc $[x, y]$ in $T$. Let $F$ be a finite subset of $H$. Then $\langle F\rangle \in \mathcal{SE}(L)$, so by the proof of (5), $\omega$--almost surely we can find $p_i\in X_i$ such that $p=\lim^\omega p_i\in \lim^\omega(X/\|\phi_i\|, o_i)$ and $d_i(p_i, \phi_i(\widetilde{h})p_i)\leq 16\delta+6$ for all $h\in F$. Hence $\langle F\rangle$ fixes the point $p$, and also the unique arc in $\lim^\omega(X/\|\phi_i\|, o_i)$ from $p$ to $[x, y]$. At least one of the two arcs $[p, x]$ and $[p, y]$ must be non-degenerate, suppose it is $[p, y]$ and $y=\lim^\omega(y_i)$.

Let $\e_3 = 16\delta+6$ and let $R_3 = R_{\e_3}$ and $N_3 = N_{\e_3}$ be given by Lemma \ref{acy2}. Now as in (1), $\omega$--almost surely we can choose $q_i\in [p_i, y_i]$ such that for all $h\in F$, $d_i(q_i, p_i)\geq d_i(p_i, \phi_i(\widetilde{h})p_i)+2\delta$, $d_i(q_i, p_i)\geq d_i(q_i, y_i)\geq d_i(y_i, \phi_i(\widetilde{h})y_i)+2\delta$ and $d_i(p_i, q_i)\geq R_3$. For a fixed $h\in F$ and lift $\widetilde{h}$ of $h$,  $\phi_i(\widetilde{h})q_i\in \phi_i(\widetilde{h})[p_i, y_i]$. Let $q_i^\prime$ be the closest point on $[p_i, y_i]$ to $\phi_i(\widetilde{h})q_i$, so $d_i(q_i^\prime, \phi_i(\widetilde{h})q_i)\leq 2\delta$. Now 
\[
d_i(q_i, \phi_i(\widetilde{h})q_i)\leq d_i(q_i, q^\prime_i)+2\delta\leq d_i(q_i, p_i)+2\delta
\]
Since $d_i(p_i, \phi_i(\widetilde{h})p_i)\leq 16\delta+6$, we can apply Lemma \ref{acy2} to get that $|F|\leq N_3$. Finally, since the size of every finite subset of $H$ is bounded by $N_3$, we get that $|H|\leq N_3$.
\end{proof}

\subsection{From $\R$--trees to simplicial trees}
The {\em Rips machine} (see \cite{BF:stable,Guirardel:Rtrees,RipsSela94}) allows one to upgrade group actions on $\R$--trees to group actions on (related) simplicial trees.

The Stability Lemma \ref{l:stability lemma} allows us to apply Guirardel's version of  the Rips machine \cite{Guirardel:Rtrees}.  We begin by recalling some terminology from \cite{Guirardel:Rtrees}.

\begin{defn} \cite[Definition 1.2]{Guirardel:Rtrees}
A {\em graph of actions} of $\R$--trees is given by the following data:
\begin{itemize}
\item A group $G$, and an isometric $G$--action without inversions on a simplicial tree $S$ (the tree $S$ is called the {\em skeleton} of the graph of actions);
\item An $\R$--tree $Y_v$ (called a {\em vertex tree}) assigned to each vertex $v$ of $S$;
\item An attaching point $p_e \in Y_v$ for each oriented edge $e$ of $S$ with terminal vertex $v$.
\end{itemize}
Moreover, we require that this data is $G$--equivariant: $G$ acts on the disjoint union of the vertex trees so that $Y_v \mapsto v$ is $G$--equivariant, and moreover for every edge $e$ in $S$ and $g \in G$ we have $p_{g.e} = g. p_e$.
\end{defn}
To a graph of actions $\mathfrak{G}= \left( S, (Y_v), (p_e) \right)$ we can assign an $\R$--tree $T_{\mathfrak{G}}$, equipped with an isometric $G$--action,  see \cite{Gui04}.  Briefly, $T_\mathfrak{G}$ is obtained by taking the quotient of the disjoint union of the $Y_v$ by identifying the two attaching points of each edge of $S$.

\begin{defn}
Suppose that $G$ is a group, and that $T$ is an $\R$--tree equipped with an isometric $G$--action.  We say that $T$ {\em splits as a graph of actions $\mathfrak{G}$} if there is a $G$--equivariant isometry from $T$ to $T_\mathfrak{G}$.
\end{defn}

The tree $T_\mathfrak{G}$ associated to a graph of actions is covered by the vertex trees in a particularly nice way, as encoded in the next definition.

\begin{defn} \cite[Definition 1.4]{Guirardel:Rtrees}
A {\em transverse covering} of an $\R$--tree $T$ is a covering of $T$ by a family of subtrees ${\mc Y} = (Y_v)_{v \in V}$ such that
\begin{enumerate}
\item Each $Y_v$ is a closed subtree of $T$;
\item Each arc of $T$ is covered by finitely many subtrees from $\mc{Y}$; and
\item For $v_1 \ne v_2 \in V$, $|Y_{v_1} \cap Y_{v_2} | \le 1$.
\end{enumerate}
\end{defn}

\begin{lem} \cite[Lemma 4.7]{Gui04}\label{l:goa-tran}
Suppose that the group $G$ acts on the $\R$--tree $T$.  If $T$ splits as a graph of actions $\mathfrak{G}$ then the image in $T$ of the non-degenerate vertex trees from $T_{\mathfrak{G}}$ give a transverse covering of $T$.

Conversely, if $T$ has a $G$--invariant transverse covering then there is a graph of actions $\mathfrak{G}$ whose non-degenerate vertex trees consist of the (non-degenerate) subtrees of the transverse covering, so that $T \simeq T_{\mathfrak G}$.
\end{lem}

We are interested in finding graph of actions decompositions of actions of group on $\R$--trees, and in particular it is important to understand the different types of vertex trees that arise.

\begin{defn} [see \cite{Guirardel:Rtrees}]
Suppose that $G$ is a group acting isometrically on an $\R$--tree $T$, and that $Y$ is a subtree of $T$.  Let $G_Y$ be the (set-wise) stabilizer of $Y$.  

We say that $Y$ is of {\em Seifert type} if the action of $G_Y$ on $Y$ has kernel $K_Y$, and the group $G_Y/K_Y$ admits a faithful action on $Y$ which is dual to an arational measured foliation on a closed $2$--orbifold with boundary.

We say that $Y$ is of {\em axial type} if $Y$ is a line and the image of $G_Y$ in $\Isom(\R)$ is a finitely generated group with dense orbits in $Y$.

We say that $Y$ is {\em simplicial} if the action of $G_Y$ on $Y$ is simplicial.
\end{defn}

For a group $G$ acting on an $\R$--tree $T$, recall from Definition \ref{def:stable} that an arc $J$ of $T$ is called stable if for every subarc $I\subseteq J$, $\Stab_G(I)=\Stab_G(J)$. Arcs which are not stable are called \emph{unstable}.  
\begin{defn}
Let $G$ be a group acting on an $\R$--tree $T$. We say that $T$ satisfies the \emph{ascending chain condition} if for any nested sequence of arcs $I_1\supseteq I_2\supseteq...$ whose intersection is a point the corresponding sequence of stabilizers $\Stab_G(I_1)\subseteq \Stab_G(I_2)\subseteq...$ is eventually constant.
\end{defn} 

\begin{thm}\cite[Main Theorem]{Guirardel:Rtrees}\label{Rmach}
Let $G$ be a group acting non-trivially and minimally on an $\mathbb R$ tree $T$ by isometries. Assume that:
\begin{enumerate}
\item $T$ satisfies the ascending chain condition.
\item For any unstable arc $J$,
\begin{enumerate}
\item $\Stab_G(J)$ is finitely generated.
\item $\Stab_G(J)$ is not conjugate to a proper subgroup of itself.
\end{enumerate}
\end{enumerate}
Then either (i) $G$ splits over the stabilizer of an unstable arc or over the stabilizer of an infinite
 tripod; or (ii) $T$ has a decomposition into a graph of actions where each vertex action
is either simplicial, axial, or of Seifert type.
\end{thm}

When the $\R$--tree $T$ (equipped with a $G$--action) admits a graph of actions as in the conclusion of Theorem \ref{Rmach}, then the set of vertex trees forms a $G$--invariant transverse covering of $T$ by Lemma \ref{l:goa-tran}. Starting with a transverse covering, we next describe how to construct a particularly nice skeleton of the associated the graph of actions. Our construction is a variation of the one given by Guirardel in the proof of \cite[Lemma 4.7]{Gui04}, so we first describe Guirardel's construction and then our modification.

If $T$ has a transverse covering by a family of subtrees $\mc{Y}$, then Guirardel constructs the skeleton $S$ of the associated graph of actions as the tree with vertices $V(S)=V_0(S)\cup V_1(S)$, where $V_1(S)$ is the set of subtrees $Y\in\mc{Y}$ and $V_0(S)$ is the set of points of $T$ which lie in at least 2 distinct subtrees. The edges of $S$ are pairs of the form $(x, Y)$ where $x\in V_0(S)$, $Y\in V_1(S)$, and $x\in Y$ \cite[Lemma 4.7]{Gui04}.

Under this construction, when $Y\in \mc{Y}$ is a simplicial subtree of $T$, it gets collapsed to a point when passing to the skeleton $S$.  Additionally, the stabilizer of the edge $e = (x,Y)$ corresponds to the stabilizer of a single point in $Y$, and we do not have good control over such stabilizers when $Y$ is simplicial. Hence, we modify the above construction in order to both keep the simplicial trees intact and also collapse edges of the form $(x, Y)$ with $Y$ a simplicial tree.

First let us assume that all simplicial pieces of the transverse covering are `maximal' in the sense that if $Y_u$ and $Y_v$ are both simplicial for some $u$ and $v$ and $Y_u \cap Y_v \ne \emptyset$ then $Y_u = Y_v$ and $u = v$.  Moreover, let us assume that simplicial trees $Y_u$ are not points.

We now build the skeleton $S$ associated to a transverse covering ${\mc Y} = (Y_v)_{v \in V}$. The vertex set $V(S)=V_0(S)\sqcup V_1(S)\sqcup V_2(S)$, where $V_0(S)$ is the set of points which lie in the intersection of at least 2 distinct trees in $\mc{Y}$, $V_1(S)$ is the set of non-simplicial trees $Y\in\mc{Y}$ and $V_2(S)$ is the set of branch points of simplicial trees $Y\in\mc{Y}$. In addition to the ``internal" branch points of simplicial trees, here we consider any point which lies in the intersection of a simplicial tree $Y$ and a non-simplicial tree to be a branch point in $Y$. The edge set $E(S)=E_0(S)\sqcup E_1(S)\sqcup E_2(S)$, where $E_0(S)$ consists of edges of the form $(x,Y)$ where $x \in V_0(S)$, $Y \in V_1(S)$ and $x \in Y$, $E_1(S)$ consists of edges in simplicial trees $Y\in\mc{Y}$, $E_2(S)$ consists of edges of the form $(x,y)$ where $x \in V_0(S)$, $y \in V_2(S)$ and $x = y$.

The {\em collapsed skeleton} $S^c$ is obtained from $S$ by collapsing the edges in $E_2(S)$. If a vertex $v$ in $S^c$ is the image of a vertex from $V_1(S)$, there there is an associated non-simplicial subtree $Y_v$ of $T$ associated to $v$. For any other vertex $v$ of $S^c$, there is an associated point $p_v\in T$ which is either a branching point in a simplicial subtree of $T$ or is the intersection of 2 non-simplicial subtrees of $T$.

\begin{observation}
Suppose that the group $G$ acts on the $\R$--tree $T$ and that $T$ splits as a graph of actions $\mathfrak{G}$ with $S^c$ the associated collapsed skeleton constructed above. If $e = (x,Y)\in E_0(S)$ then $G_e = G_Y \cap \Stab(x)$, those elements of the group $G_Y$ which fix $x$. If $e\in E_1(S)$ is an edge in a simplicial tree $Y\in\mc{Y}$ with endpoints $x, y$, then $G_e=Stab_G([x, y])$.
\end{observation}

The following result is similar to \cite[Lemma 4.9]{Gui04}.
\begin{lem}
Suppose that the non-degenerate $\R$--tree $T$ is equipped with a nontrivial minimal action of the group $G$, and that $(Y_v)_{v \in V}$ is a $G$--equivariant transverse covering.  Let $S$ be the associated skeleton and $S^c$ the collapsed skeleton.  

The $G$--action on $S^c$ is minimal, and $S^c$ is a point if and only if the transverse covering consists of a single non-simplicial tree.
\end{lem}
\begin{proof}
Suppose that $S' \subset S^c$ is a $G$--invariant subtree. Associated to each vertex of $S'$ is a subtree (possibly a point) of $T$. Let $T'$ be the union of these subtrees of $T$; additionally, if $u, v$ are adjacent vertices in $S^\prime$ which correspond to points $p_u$, $p_v$ in a simplicial subtree $Y\in\mc{Y}$, then add the edge $[p_u, p_v]$ to $T^\prime$. Note that $T^\prime$ is equal to a union of subtrees which belong to $\mc{Y}$ and arcs which are contained in elements of $\mc{Y}$.  Since $S'$ is connected and edges in $S^c$ imply either nonempty intersection of subtrees or adjacent vertices in a simplicial subtree of $T$, $T'$ is also connected and hence is a subtree of $T$.  It is clear that $T'$ is $G$--invariant, which by minimality of $T$ implies that $T' = T$.

Suppose that $v$ is a vertex in $S^c$. Then the subtree $Y$ of $T$ associated to $v$ is contained in $T' = T$. If $Y$ is a non-simplicial tree, then $Y\in\mc{Y}$ and hence $Y$ cannot be covered by other elements of $\mc{Y}$ by the definition of a transverse cover, so we must have $v\in S'$. If $Y$ is a point in the intersection of two non-simplicial subtrees of $T$, then the vertices corresponding to these non-simplicial subtrees both belong to $S'$ and are adjacent to $v$ in $S^c$, hence $v\in S'$. If $Y$ is a branch point in a simplicial tree belonging to $\mc{Y}$, then $Y$ is the endpoint of an edge $e$ of this simplicial tree (since we assumed that simplicial subtree in $\mc{Y}$ are not points). Since $e$ is contained in $T^\prime$, the vertices of $S^c$ corresponding to the endpoints of $e$ must belong to $S'$, hence $v\in S'$. The final assertion about when $S^c$ can be a point is clear from the construction.
\end{proof}

Recall that a divergent $\mc{G}$--limit group $L$ acts on a $\R$--tree $T$ and the stabilizers of unstable arcs are uniformly finite by the Stability Lemma \ref{l:stability lemma}. Clearly such an action satisfies the hypothesis of Theorem \ref{Rmach}, so we have the following.

\begin{thm} \label{t:Rmach applies}
Suppose that $G$ is a finitely generated group and that $(\phi_i)$ is a divergent sequence from $\Hom(G, \mc{G})$ with $L = G / \ker^\omega (\phi_i)$.  Then the associated action of $L$ on the limiting $\R$--tree $T$ satisfies the hypotheses of Theorem \ref{Rmach}.

Consequently, either 
\begin{enumerate}
\item $L$ splits over a uniformly finite subgroup which is the stabilizer of either an unstable arc  or of an infinite tripod in $T$; or
\item The $L$--action on $T$ admits a graph of actions decomposition where each vertex tree is either simplicial, of axial type, or of Seifert type.
\end{enumerate}
\end{thm}

In the first case above, the splitting is over a (uniformly) finite edge group, and we are able to ignore this case in this paper.  We now proceed to describe the splitting of $L$ obtained in the second case of this theorem in more detail, summarizing the above construction.

\begin{prop}\label{prop:splitting}
Let $L$ be a divergent $\mc{G}$--limit group, and $T$ the associated limiting $\R$--tree. Suppose that $L$ does not split over the  stabilizer of an unstable arc or over the stabilizer of an infinite tripod in $T$. Then the action of $L$ on the limit tree $T$ splits as a graph of actions with collapsed skeleton $S^c$ and the following hold:

\begin{enumerate}
\item $V(S^c)=U_0\sqcup U_1$ such that for each $v\in U_0$ there is an associated point $p_v\in T$ and for each $v\in U_1$ there is an associated non-simplicial subtree $Y_v$ of $T$.
\item If $v\in U_0$, then $\mathrm{Stab}_L(v)=\mathrm{Stab}_L(p_v)$.
\item If $v\in U_1$, then the action of $\mathrm{Stab}_L(v)=L_{Y_v}$ on $Y_v$ is either of axial type or of Seifert type. 
\item If $v\in U_1$ and the action of $L_{Y_v}$ on $Y_v$ is of Seifert type, the kernel of the action of $L_{Y_v}$ on $Y_v$ is uniformly finite and all infinite subgroups of $L_{Y_v}$ are either contained in $\mathcal{NLSE}(L)$ or they contain a cyclic subgroup of uniformly finite index. Moreover, those infinite subgroups that are not in $\mc{NLSE}(L)$ correspond to boundary components of the associated orbifold.
\item If $v\in U_1$ and the action of $L_{Y_v}$ on $Y_v$ is of axial type, then $L_{Y_v}$ has a subgroup of index at most 2 which is (uniformly finite)-by-abelian and contained in $\mathcal{NLSE}(L)$. 
\item Distinct vertices in $U_1$ are not adjacent in $S^c$.
\item If $e=(u, v)\in E(S^c)$ with $u, v\in U_0$, then $\mathrm{Stab}_L(e)=\mathrm{Stab}_L([p_u, p_v])$ is (uniformly finite)-by-abelian and contained in $\mathcal{NLSE}(L)$.
\item If $e=(u, v)\in E(S^c)$ with $u\in U_0$ and $v\in U_1$, then $\mathrm{Stab}_L(e)=\mathrm{Stab}_L(p_u)\cap L_{Y_v}$. Furthermore,
\begin{enumerate}
\item If the action of $L_{Y_v}$ is Seifert-type, then $\mathrm{Stab}_L(e)$ is either finite or it has a cyclic subgroup of uniformly finite index.
\item\label{eq:stab axial} If the action of $L_{Y_v}$ is axial-type, then $\mathrm{Stab}_L(e)$ is either uniformly finite or has a subgroup of  index at most 2 which is (uniformly finite)-by-abelian and contained in $\mathcal{NLSE}(L)$.
\end{enumerate}
\end{enumerate}
 Finally, if $H\leq L$ fixes a point in $T$ then $H$ fixes a point in $S^c$. In particular, every finitely generated member of $\mathcal{SE}(L)$ is elliptic with respect to the action of $L$ on $S^c$.
\end{prop}

\begin{proof}
Let $U_0$ be the image of $V_0\cup V_2$ in $S^c$ and $U_1$ the image of $V_1$ in $S^c$. Properties (1), (2), (3), and (6) follow from the construction of $S^c$ and Theorem \ref{t:Rmach applies}; it is also clear from the construction of $S^c$ that if $H\leq L$ fixes a point in $T$ it also fixes a point in $S^c$. If  the action of $L_{Y_v}$ on $Y_v$ is of Seifert type, then the kernel $K$ of the action of $L_{Y_v}$ on $Y_v$ fixes a tripod in $T$ and hence is uniformly finite by Lemma \ref{l:stability lemma}(3).  Furthermore, every infinite subgroup of $L_{Y_v}$ corresponds either to a boundary component of the underlying orbifold and hence contain a cyclic subgroup of uniformly finite index, or it contains an infinite cyclic subgroup which has no fixed point in $T$, hence this subgroup belongs to $\mathcal{NLSE}(L)$ by Lemma \ref{l:stability lemma}(5). Hence (4) holds.

(5) follows from Lemma \ref{l:stability lemma}(2), \ref{l:stability lemma}(5), and the fact that $L_{Y_v}$ has a finitely generated subgroups with no fixed points in $T$ since the image of $L_{Y_v}$ in $\Isom(\R)$ is a finitely generated subgroup with dense orbits. (7) follows from Lemma \ref{l:stability lemma}(1). (8a) Follows from the fact that points in $Y_v$ which are fixed by non-trivial subgroups of $L_{Y_v}$ correspond to either boundary components in the corresponding orbifold or to cone points in the orbifold, that is points with finite stabilizer. Finally, (8b) follows from Lemma \ref{aa-finite} proved in the next section.
\end{proof}

\section{JSJ decompositions and the Shortening Argument} \label{s:JSJ}

In the previous section, we showed how a divergent $\mc{G}$--limit group $L$ admits a natural action on an $\mathbb R$--tree which in turn produces a splitting of $L$ via the Rips machine. In this section, we prove the existence of a JSJ decomposition for $L$ (see Theorem \ref{t:JSJ}), a graph of groups decompositions which in a certain sense encodes all possible splittings for $L$ which can arise in this manner.
In Subsection \ref{ss:shortening moves}, we provide a collection of automorphisms of a divergent $\mc{G}$--limit group which are then used to make an adaptation of Sela's shortening argument to our setting.

\subsection{Virtually abelian subgroups}
For the rest of this section, fix a uniformly acylindrically hyperbolic family $\mc{G}$ and a $\mc{G}$--limit group $L$ with a defining sequence $(\phi_i)$ from $\Hom(G, \mc{G})$ for some finitely generated group $G$. As in Definition \ref{def:SE and LSE}, we let $\mathcal{SE}(L)$ and $\mathcal{LSE}(L)$ denote the set of stably elliptic and locally stably elliptic subgroups of $L$, {\em with respect to the defining sequence $(\phi_i)$}, and let $\mathcal{NLSE}(L)$ denote the set of subgroups which are not locally stably elliptic (with respect to $(\phi_i)$). We also use $K$ throughout this section as the constant from Lemma \ref{ufin} for the family $\mc{G}$. As before, uniformly finite means bounded by a constant which depends only on the hyperbolicity and acylindricity constants for the family $\mc{G}$.

The following proof is essentially the same as \cite[Lemma 1.21]{ReiWei}.
\begin{lem}\label{aasub}
Let $H\leq L$ be a subgroup so that $H \in \mathcal{NLSE}(L)$. 
Then the following are equivalent:
\begin{enumerate}
\item $H$ is virtually finite-by-abelian.
\item $\omega$--almost surely there exists loxodromic elements $g_i\in\Gamma_i$ such that for every finitely generated subgroup $J\leq H$ $\phi_i(\widetilde{J})\subseteq E_{\Gamma_i}(g_i)$.
\item $H$ has a (uniformly finite)-by-abelian subgroup $H^+$ of index at most 2 so that $[H^+: Z(H^+)]$ is uniformly finite.
\item $H$ is virtually abelian.
\item All finitely generated subgroups of $H$ are virtually abelian.
\end{enumerate}
\end{lem}
\begin{proof}
We begin by establishing some notation that is used throughout the proof. Let $H=\{h_1,...\}$, and let $\widetilde{h}_i\in\phi_\infty^{-1}(h_i)$. Let $H_0$ be a finitely generated subgroup of $H$ such that $\phi_i(\widetilde{H}_0)$ is $\omega$--almost surely not elliptic where $\widetilde{H}_0$ be a finitely generated subgroup of $G$ such that $\phi_\infty(\widetilde{H}_0)=H_0$.  Thus $\omega$--almost surely we can choose a loxodromic element $g_i \in \Gamma_i$ with $g_i \in \phi_i(\widetilde{H}_0)$. Let $H_k=\langle H_0, h_1,...,h_k\rangle$ and let $\widetilde{H}_k=\langle \widetilde{H}_0, \widetilde{h}_1,...,\widetilde{h}_k\rangle$. Let $\widetilde{H}=\langle \widetilde{h}_1,...\rangle$. We clearly have $\phi_\infty(\widetilde{H})=H$.

Suppose $H$ is virtually finite-by-abelian. Then each $H_k$ is virtually finite-by-abelian, and since $H_k$ is also finitely generated we get that $H_k$ is finitely presented. It follows that $\omega$--almost surely, $\phi_i(\widetilde{H}_k)$ is a quotient of $H_k$ and hence virtually finite-by-abelian, which implies that  $\phi_i(\widetilde{H_k})\subseteq E_{\Gamma_i}(g_i)$. This shows that $(1)\implies (2)$.

 Now suppose that $\omega$--almost surely $\phi_i(\widetilde{H_k})\subseteq E_{\Gamma_i}(g_i)$. Note that there are finitely many isomorphism types of elementary groups with all finite subgroups of size at most $K$. Thus, $\omega$--almost surely there is a fixed elementary group $E$ and homomorphisms $\phi_i^\prime\colon \widetilde{H_k}\to E$ such that $H_k=\widetilde{H_k}/\ker^\omega(\phi_i^\prime)$. Now $E$ has a subgroup $E^+$ of index at most 2 which is finite-by-$\Z$. Let $\widetilde{H}^+_{k, i}=(\phi_i^\prime)^{-1}(E^+)$. This is a subgroup of index at most $2$ in $\widetilde{H}_k$, and since $\widetilde{H}_k$ has only finitely many subgroups of index at most 2 there is a fixed subgroup $\widetilde{H}_k^+$ which occurs $\omega$--almost surely. Let $H_k^+=\phi_\infty(\widetilde{H}_k^+)$, clearly $[H_k\;:\;H_k^+]\leq 2$. Since for each $i$ we have $\widetilde{H}^+_{k, i}\subseteq\widetilde{H}^+_{k+1, i}$, it follows that $H_k^+\leq H_{k+1}^+$. Hence we can define $H^+=\bigcup_{k=1}^\infty H_k^+$, which is a subgroup of $H$ of index at most 2. The size of the commutator subgroup of $\phi_i^\prime(\widetilde{H}^+_k)$ is uniformly finite, so the commutator subgroup of each $H^+_k$ is uniformly finite and hence the commutator subgroup of $H^+$ is uniformly finite. Thus $H^+$ is (uniformly finite)-by-abelian.

Let $M=[E^+\;:\;Z(E^+)]$ and let $\widetilde{Z}_{k, i}=(\phi_i^\prime)^{-1}(Z(E^+))\cap\widetilde{H}^+_k$, a subgroup of $H_k^+$ of index at most $M$. Again, since $H_k^+$ is finitely generated there are only finitely many such subgroups, so $\omega$--almost surely this is equal to a fixed subgroup which we denote $\widetilde{Z}_k$. As $\phi_i^\prime(\widetilde{Z}_k)$ is central $\omega$--almost surely, $Z_k:=\phi_\infty(\widetilde{Z}_k)$ is a central subgroup of $H^+_k$ of index at most $M$. Hence $[H_k^+\;:\; Z(H_k^+)]\leq M$. 

After passing to a subsequence of $k$ we can assume that $[H_k^+\;:\; Z(H_k^+)]=M_0$. for some fixed $M_0\leq M$.  $H^+_k\leq H^+_{k+1}$, so the the index of $Z(H^+_{k+1})\cap H^+_k$ in $H^+_k$ is at most $M_0$.  Also $Z(H^+_{k+1})\cap H^+_k\leq Z(H^+_k)$, and  $[H_k^+\;:\; Z(H_k^+)]=M_0$, so $Z(H^+_{k+1})\cap H^+_k=Z(H^+_k)$. Thus $Z(H^+_k)\leq Z(H^+_{k+1})$, so we can define the subgroup $Z=\bigcup_{k=1}^\infty Z_k$. Since $Z$ is a central subgroup of $H^+$ of index $M_0\leq M$, we have $[H^+\;:\; Z(H^+)]\leq M$. It is easy to see that one can bound $M$ in terms of $K$, so we have shown that $(2)\implies (3)$.

Since $(3)\implies (1)$ is trivial, we have proved that $(1)$, $(2)$ and $(3)$ are equivalent. Also the implications $(3)\implies(4)\implies(5)$ are obvious. Finally, to conclude we note that the same proof which shows that $(1)\implies (2)$ can be applied to show that $(5)\implies (2)$.
\end{proof}

\begin{lem}\label{aa-finite}
Suppose $F \le H$ are virtually abelian subgroups of $L$ such that $F \in \mathcal{LSE}(L)$ and $H \in \mathcal{NLSE}(L)$.  Then $|F| \le K$.
\end{lem}
\begin{proof}
Let $F_0$ be a finitely generated subgroup of $F$. Then by Lemma \ref{aasub}, there exists $g_i\in\Gamma_i$ such that $\phi_i(\widetilde{F_0})\subseteq E_{\Gamma_i}(g_i)$ $\omega$--almost surely. Since $F_0$ is stably elliptic, $\phi_i(\widetilde{F_0})$ is elliptic $\omega$--almost surely, so by Lemma \ref{ufin} $\omega$--almost surely $|\phi_i(\widetilde{F_0})|\leq K$. Hence $|F_0|\leq K$. Since every finitely generated subgroup of $F$ has size at most $K$, $|F|\leq K$.
\end{proof}

\begin{lem}
Suppose $H \in \mathcal{NLSE}(L)$ and $H$ is virtually abelian. Then there exists $\widetilde{h}\in\widetilde{H}$ such that $\omega$--almost surely $\phi_i(\widetilde{h})$ is loxodromic. 
\end{lem}
\begin{proof}
There is a uniform bound on the order of finite subgroups of $H$. Since $H$ is virtually abelian, this means that $H$ must contain an element of infinite order.  Let $h$ be such an element. Then $\phi_i(\widetilde{h})\in E_{\Gamma_i}(g_i)$ has infinite order and hence is loxodromic $\omega$--almost surely.
\end{proof}

Since subgroups of acylindrically hyperbolic groups can be arbitrary, we can not make any claims about the local algebraic structure of limit groups over acylindrically hyperbolic groups. This is an important technical difference from the hyperbolic case, where there is rather strict control on finite subgroups, virtually abelian subgroups, and elements of infinite order. In our situation many of these properties still go through if we restrict to subgroups of our limit group which ``see" the hyperbolicity in an essential way. For that we introduce the following definition, which characterizes those subgroups which never ``see" the hyperbolicity of $\mc{G}$.
\begin{defn}
Let $R$ be a $\mc{G}$--limit group. We say a subgroup $H\leq R$ is \emph{absolutely elliptic} if for all  finitely generated groups $G_0$,  and all defining sequences $(\psi_i)$ from $\Hom(G_0, \mc{G})$ such that $R=G_0/\ker^\omega(\psi_i)$,  $H$ is locally stably elliptic with respect to the sequence $(\psi_i)$. 
\end{defn}

\begin{lem}\label{aa-intersect}
Suppose $A$ and $B$ are virtually abelian subgroups of $L$ which are not absolutely elliptic.
\begin{enumerate}
\item If $A\cap B$ is not absolutely elliptic, then  $\langle A, B\rangle$ is virtually abelian.
\item If $A\cap B$ is absolutely elliptic, then $|A\cap B|\leq K$.
\end{enumerate}
\end{lem}

\begin{proof}
Suppose $A\cap B$ is not absolutely elliptic. Let $(\psi_i)$ be a defining sequence of homomorphisms for $L$ such that $A\cap B \in \mathcal{NLSE}(L)$ (with respect to this sequence). Clearly this implies that $A$ and $B$ are also contained in $\mathcal{NLSE}(L)$ with respect to $\psi_i$. Choose $H\leq A\cap B$ such that $H$ is finitely generated and there is a loxodromic element $g_i$ so that $\omega$--almost surely $g_i\in\psi_i(\widetilde{H})$. Then by Lemma \ref{aasub}, for every finitely generated $A_0\leq A$ and $B_0\leq B$, $\phi_i(\widetilde{A_0}\cup\widetilde{B_0})\subseteq E_{\Gamma_i}(g_i)$ and hence $\psi_i(\langle \widetilde{A_0}, \widetilde{B_0}\rangle)\subseteq E_{\Gamma_i}(g_i)$ $\omega$--almost surely. Since every finitely generated subgroup of $\langle A, B\rangle$ is contained in a subgroup of the form $\langle A_0, B_0\rangle$ with $A_0\leq A$ and $B_0\leq B$ both finitely generated, we get that for all finitely generated subgroups $J\leq\langle A, B\rangle$, $\psi_i(\widetilde{J})\leq E_{\Gamma_i}(g_i)$. Hence $\langle A, B\rangle$ is virtually abelian by Lemma \ref{aasub}.

Now suppose $A\cap B$ is absolutely elliptic and $A \in \mathcal{NLSE}(L)$ with respect to a defining sequence $(\psi_i)$. Since $A\cap B$ is absolutely elliptic, $A\cap B \in \mathcal{LSE}(L)$ with respect to $\psi_i$. Hence $|A\cap B|\leq K$ by Lemma \ref{aa-finite}.
\end{proof}

\begin{lem}\label{max-aasub}
Let $H\leq L$ be a virtually abelian subgroup which is not absolutely elliptic. Let 
$$M_H=\left\langle\{a\in L\;|\; \langle a, H\rangle \text{ is virtually abelian }\}\right\rangle .$$ 
Then $M_H$ is the unique, maximal, virtually abelian subgroup of $L$ containing $H$.
\end{lem}

\begin{proof}
If $a$ and $b$ belong to $M_H$, then $H\subseteq \langle a, H\rangle \cap \langle b, H\rangle$, and hence  $\langle a, b, H\rangle$ is virtually abelian by Lemma \ref{aa-intersect}. By induction it follows that all finitely generated subgroups of $M_H$ are virtually abelian, which implies that $M_H$ is virtually abelian by Lemma \ref{aasub}. 
The maximality and uniqueness are now clear.
\end{proof}

\begin{lem}\label{aa-conj}
If $M$ is a maximal virtually abelian subgroup which is not absolutely elliptic and $g^{-1}Mg\cap M$ is infinite, then $g\in M$.
\end{lem}

\begin{proof}
By Lemma \ref{aa-intersect}, since $g^{-1}Mg\cap M$ is infinite it must be not absolutely elliptic, hence there exist $\psi_i\in \Hom(G, \mc{G})$ defining $L$ and $h\in g^{-1}Mg\cap M$ such that $\psi_i(\widetilde{h})$ is loxodromic $\omega$--almost surely. Let $M_0\leq M$ be a finitely generated subgroup which contains $h$. Since $g^{-1}M_0g$ is virtually abelian, $\psi_i(\widetilde{g}^{-1}\widetilde{M}_0\widetilde{g})\subseteq E_{\Gamma_i}(\psi_i(\widetilde{h}))$, and since $h\in M_0$, it follows that $\psi_i(\widetilde{g}^{-1}\widetilde{h}\widetilde{g})\in  E_{\Gamma_i}(\psi_i(\widetilde{h}))$, thus $\psi_i(\widetilde{g})\in E_{\Gamma_i}(\psi_i(\widetilde{h}))$ $\omega$--almost surely by Lemma \ref{E(h)}. Since also $\psi_i(\widetilde{M_0})\subseteq  E_{\Gamma_i}(\psi_i(\widetilde{h}_i))$ $\omega$--almost surely, we get that $\langle g, M_0\rangle$ is virtually abelian. It follows that every finitely generated subgroup of $\langle g, M\rangle$ is virtually abelian, so $\langle g, M\rangle$ is virtually abelian by Lemma \ref{aasub}. Therefore $g\in M$ by maximality.

\end{proof}

\begin{sa}\label{sa:uf}
Given $\mc{G}$, there exists a number $C$ so that $C \ge K$ for $K$ the constant from Lemma \ref{ufin} so that everything which is shown to be ``uniformly finite" in Lemma \ref{l:stability lemma}, Theorem \ref{t:Rmach applies}, Proposition \ref{prop:splitting} and Lemma \ref{aasub} is bounded by $C$.  We fix this constant $C$ for the remainder of the section.
\end{sa}

\subsection{JSJ-decompositions}

We are now ready to start showing that $L$ admits a JSJ decomposition. We use the language of Guirardel--Levitt \cite{GuiLev2} for JSJ decompositions. 

\begin{defn}
Given a group $G$, let $\mathcal A$ and $\mathcal H$ be two families of subgroups of $G$ such that $\mathcal A$ is closed under conjugation and taking subgroups. An $(\mathcal A, \mathcal H)$--tree is a simplicial tree $T$ together with an action of $G$ on $T$ such that each edge stabilizer belongs to $\mathcal A$ and each $H\in\mathcal H$ fixes a point in $T$. Similarly, a graph of groups decomposition of $G$ is called an {\em $(\mc{A},\mc{H})$--splitting} if the corresponding Bass--Serre tree is an $(\mc{A}, \mc{H})$--tree. This is also referred to as a splitting over $\mc{A}$ relative to $\mc{H}$.

An $(\mathcal A, \mathcal H)$--tree $T$ is called \emph{universally elliptic} if the edge stabilizers of $T$ are elliptic in every other $(\mathcal A, \mathcal H)$--tree. Also, given trees $T$ and $T^\prime$, we say $T$ \emph{dominates} $T^\prime$ if every vertex stabilizer of $T$ is elliptic with respect to $T^\prime$. 
\end{defn}

\begin{defn}
A graph of groups decomposition $\mathbb A$ of $G$ is called a \emph{$(\mathcal A, \mathcal H)$--JSJ decomposition of $L$} if the corresponding Bass--Serre tree $T$ is an $(\mathcal A, \mathcal H)$--tree  which is universally elliptic and which dominates every other universally elliptic $(\mathcal A, \mathcal H)$--tree.
\end{defn}

Generally speaking, an $(\mathcal A, \mathcal H)$--JSJ decomposition is not unique, see \cite{GuiLev2}. 

A subgroup $H\leq G$ is called \emph{rigid in $(\mathcal A, \mathcal H)$--trees} if this subgroup is elliptic in every $(\mathcal A, \mathcal H)$--tree. Subgroups which are not rigid are called \emph{flexible (in $(\mathcal A, \mathcal H)$--trees)}. When $\mc{A}$ and $\mc{H}$ are understood we simply say the subgroup is \emph{rigid} or \emph{flexible}. Typically, a description of a JSJ decomposition also includes a description of the flexible vertex groups.

\begin{defn}
A vertex group $H = G_v$ of an $(\mc{A},\mc{H})$--graph of groups decomposition of a group $G$ is called a \emph{QH-subgroup} if there is a normal subgroup $N\leq H$ (called the \emph{fiber of $H$}) such that $H/N$ is isomorphic to the fundamental group of a hyperbolic $2$--orbifold and the image of any incident edge group in $H/N$ is finite or conjugate to the fundamental group of a boundary component of the orbifold. 

A {\em \QHAH--subgroup} is a QH--subgroup so that for every essential simple closed curve $\gamma$ on the underlying orbifold, the corresponding subgroup $Q_\gamma$ belongs to $\mc{A}$, and if furthermore the intersection of $H$ and any element of $\mathcal H$ has image in $H/N$ that is finite or conjugate to the fundamental group of a boundary component of the orbifold.
\end{defn}

In the case where $L$ is a divergent $\mc{G}$--limit group, we consider the following families of subgroups.

\begin{defn}
Let $\mc{A}_{\ufin}$ denote the collection of all finite subgroups of $L$ of order at most $2C$ and let $\mc{A}_\infty$ denote the collection of all virtually abelian subgroups which are not absolutely elliptic. Let $\mc{A}=\mc{A}_{\ufin}\cup\mc{A}_\infty$.  Let $\mc{H}$ denote the collection of all finitely generated absolutely elliptic subgroups of $L$.
\end{defn}

In this section we are mostly concerned with the case where $L$ does not split over $\mc{A}_{\ufin}$ relative to $\mc{H}$. As is explained in Section \ref{s:short quotient}, the general case can be reduced to this case using a Linnell decomposition, which is an $(\mc{A}_{\ufin}, \mc{H})$--JSJ decomposition. 

 In all situations we consider, the fiber $N$ of a QH--subgroup is finite and so $Q_\gamma$ is virtually cyclic. Hence in case $L$ is a $\mc{G}$--limit group, in the above definition of \QHAH we are requiring curves of the underlying orbifold to correspond to subgroups which are not absolutely elliptic.

The collection $\mc{A}$ is clearly closed under conjugation, and $\mc{A}$ is closed under taking subgroups by Lemma \ref{aa-finite}. Our next goal is to show that a divergent $\mc{G}$--limit group admits a non-trivial $(\mc{A}, \mc{H})$--JSJ decomposition. Since $L$ is not (necessarily) finitely presented, we need Weidmann's $(k,C)$--acylindrical accessibility \cite{Weid-access} (a generalization of Sela's acylindrical accessibility \cite{sela:acyl}) to show that these JSJ decompositions exists. In order to do this, we need to show that $(\mc{A}, \mc{H})$--splittings can be modified to make them $(2,C)$--acylindrical. This can be done with the methods of \cite{ReiWei}; instead we use the Guirardel--Levitt \emph{tree of cylinders} construction.

\begin{defn}
A tree $T$ is {\em $(k, C)$--acylindrical} if the point-wise stabilizer of every path of length $\geq k+1$ has order $\leq C$.
\end{defn}

For $A, B\in\mc{A}_\infty$, let $A\sim B$ if $\langle A, B\rangle$ is virtually abelian. The fact that this is an equivalence relation follows easily from Lemma \ref{aa-intersect}. Similarly, it is easy to check that this equivalence relation is \emph{admissible} in the sense of \cite[Definition 7.1]{GuiLev2}. Now, $L$ acts on $\mc{A}_\infty/\sim$ by conjugation, and the stabilizer of an equivalence class $[A]$ is equal to the maximal virtually abelian subgroup $M_A$ containing $A$ by Lemma \ref{aa-conj}. In particular, each such stabilizer is small (contains no nonabelian free subgroup) and belongs to $\mc{A}_\infty$. Thus we can build the tree of cylinders $T_c$ for any $(\mc{A}_\infty, \mc{H})$ tree $T$. The properties of this construction are summarized in the following result, which is an immediate consequence of \cite[Proposition 7.12]{GuiLev2}. Note that in our setting, the tree of cylinders and the collapsed tree of cylinders are the same, see \cite[Lemma 7.3]{GuiLev2}.

\begin{prop}\cite[Proposition 7.12]{GuiLev2}\label{prop:toc}
For any $(\mc{A}_\infty, \mc{H})$--tree $T$, there exists a $(\mc{A}_\infty, \mc{H})$ tree $T_c$ called the \emph{tree of cylinders} such that 
\begin{enumerate}
\item $T_c$ is $(2, C)$--acylindrical.
\item $T$ dominates $T_c$ and any group which is elliptic in $T_c$ but not in $T$ is virtually abelian and not virtually cyclic.
\end{enumerate}
\end{prop}

A group is called \emph{$C$--virtually cyclic} if it maps onto either $\Z$ or $D_\infty$ with kernel of order at most $C$.

\begin{thm}\cite[First part of Theorem 8.7]{GuiLev2}\label{thm:acyJSJ}
Given $\mathcal A$ and $\mathcal H$, suppose there exist numbers $C$ and $k$ such that
\begin{enumerate}
\item $\mathcal A\cup\mc{H}$ contains all $C$--virtually cyclic subgroups.
\item $\mathcal A$ contains all subgroups of order $\leq 2C$.
\item Every $(\mathcal A, \mathcal H)$--tree $T$, dominates some $(k, C)$--acylindrical $(\mathcal A, \mathcal H)$ tree $T^\ast$ such that every subgroup which is elliptic in $T^\ast$ but not in $T$ is virtually abelian.
\end{enumerate}

Then the $JSJ$ deformation space of $G$ over $\mc{A}$ relative to $\mathcal H$ exists.
\end{thm}

\begin{rem}
In \cite{GuiLev2}, Theorem \ref{thm:acyJSJ} is proved with the hypothesis that $\mc{A}$ contains all $C$--virtually cyclic groups. However, in the proof this condition is only applied to $C$--virtually cyclic groups which act hyperbolically on some $(\mc{A}, \mc{H})$--tree. Such a group cannot belong to $\mc{H}$, and so they must belong to $\mc{A}$. Hence the proof given in \cite{GuiLev2} works as written to prove Theorem \ref{thm:acyJSJ}. 
\end{rem}

 The second part of \cite[Theorem 8.7]{GuiLev2} is a description of the flexible vertex groups.  This part is more complicated in our setting and the proof from \cite{GuiLev2} {\em does} use the assumption that $\mathcal A$ contains all $C$--virtually cyclic groups in an essential way.  We deal with the flexible vertex groups in the JSJ decomposition in Section \ref{ss:flexible} below.

\subsection{Flexible vertex groups} \label{ss:flexible}

For the rest of this section, we assume that $L$ admits no splitting over $\mc{A}_{\ufin}$ relative to $\mc{H}$. Combining Proposition \ref{prop:toc} and Theorem \ref{thm:acyJSJ}, we find that $L$ admits a $(\mc{A}, \mc{H})$--JSJ decomposition. Our next goal is to describe the flexible vertices of this JSJ decomposition. If $\mc{A}$ contained all $C$--virtually cyclic groups, then we could apply \cite[Theorem 8.7]{GuiLev2} directly to get that all flexible vertices are either virtually abelian or QH with finite fiber. However, this need not be the case in our situation. Nevertheless, we can get a description of the flexible vertices by applying essentially the same proof which appears in \cite{GuiLev2}. We start by briefly sketching the proof that flexible vertices are QH in the proof of \cite[Theorem 8.7]{GuiLev2}, then we explain how this proof can be modified in our situation.

Fix a flexible vertex $v$ of the JSJ decomposition. Since we can apply the tree of cylinders construction to make our splittings $(2,C)$--acylindrical, it suffices to assume that all splittings of $L_v$ that we consider are already $(2, C)$--acylindrical. Let $\mathrm{Inc}_v^{\mc{H}}$ denote the subgroups of $L_v$ which are conjugate into stabilizers of edges adjacent to $v$ or are contained in $\mc{H}$. Then any splitting of $L_v$ over $\mc{A}$ relative to $\mathrm{Inc}_v^{\mc{H}}$ extends to a $(\mc{A}, \mc{H})$--splitting of $L$. Since $v$ is a vertex of the JSJ decomposition, we may assume as in \cite[$\S$6.2]{GuiLev2} that $L_v$ is \emph{totally flexible}, which means that $L_v$ admits no splitting over $\mc{A}$ relative to $\mathrm{Inc}_v^{\mc{H}}$ with edge groups that are universally elliptic. This implies (by acylindricity) that every $(\mc{A}, \mathrm{Inc}_v^{\mc{H}})$--splitting of $L_v$ has $C$--virtually cyclic subgroups.

By Weidmann's acylindrical accessibility \cite[Theorem 1]{Weid-access}, $L_v$ has a ``maximal" $(2, C)$--acylindrical tree $U$, that is a tree with the maximal number of orbits of edges among all $(2,C)$--acylindrical splittings of $L_v$. By \cite[Proposition 6.28]{GuiLev2}, $L_v$ has another tree $V$ which is fully hyperbolic with respect to $U$. Let $R$ be the regular neighborhood of $U$ and $V$, described in \cite[Proposition 6.25]{GuiLev2}.

The tree $R$ is bipartite with vertices belonging to $\mc{S}\neq\emptyset$ and $\mc{V}$, with $\mc{V}$  possibly empty. Vertices in $\mc{S}$ have stabilizers which are QH while vertices in $\mc{V}$ have stabilizers which are elliptic in both $U$ and $V$. Furthermore, every subgroup in $\mathrm{Inc}_v^{\mc{H}}$ fixes a point in $R$, and edge stabilizers in $R$ are $C$--virtually cyclic. In the setting of \cite{GuiLev2}, this means that $R$ is a (minimal) $(\mc{A}, \mathrm{Inc}_v^{\mc{H}})$--tree. If $R$ contains an edge, then this produces a contradiction with the maximality of $U$. Hence $R$ is consists of a single vertex which belongs to $\mc{S}$, and so $L_v$ itself is QH.

In our setting, we obtain a contradiction if at least one of the edge stabilizers of $R$ belongs to $\mc{A}$. It follows that all edge stabilizers of $R$ are $C$--virtually cyclic and absolutely elliptic. Similarly, each vertex in $\mc{V}$ must have a stabilizer which does not split over $\mc{A}$ relative to $\mathrm{Inc}_v^{\mc{H}}$, otherwise we could refine $R$ at this vertex and obtain a contradiction with the maximality of $U$ as above. Hence these vertex groups are rigid in $(\mc{A}, \mc{H})$--trees. 

For each vertex $v\in\mc{S}$ and each essential simple closed geodesic $\gamma$ on the underlying orbifold, the corresponding subgroup $Q_\gamma\leq L_v$ belongs to $\mc{A}$. This follows from maximality of $U$.  Indeed, $U$ is obtained by refining $R$ at each $v\in\mc{S}$ and then collapsing all original edges of $R$. This refinement must correspond to a filling collection of geodesics on each orbifold by maximality of $U$, and hence the subgroup corresponding to each other geodesic acts hyperbolically on $U$ and hence does not belong to $\mc{H}$.

Thus we obtain the following.

\begin{lem}\label{l:flexible}
Let $v$ be a flexible vertex in the $(\mc{A}, \mc{H})$--JSJ decomposition of $L$ and suppose that $L_v$ is not virtually abelian. Then either $L_v$ is \QHAH with finite fiber or $L_v$ admits a splitting over $C$--virtually cyclic absolutely elliptic subgroups relative to  $\mathrm{Inc}_v^{\mc{H}}$ such that every vertex group in this splitting is either (i) rigid in $(\mc{A}, \mc{H})$--trees; or (ii) \QHAH with finite fiber.
\end{lem}

\begin{cor}\label{c:cut scc}
Let $v$ be a flexible vertex in the $(\mc{A},\mc{H})$--JSJ decomposition of $L$ and suppose that $L_v$ is not virtually abelian.  Suppose that $\Lambda$ is a nontrivial one-edge $(\mc{A},\mc{H})$--splitting of $L_v$ relative to $\mathrm{Inc}_v^{\mc{H}}$.  Then $\Lambda$ can be obtained from the splitting of $L_v$ in Lemma \ref{l:flexible} from cutting the surface corresponding to a \QHAH-vertex along a simple closed curve and then collapsing all other edges.
\end{cor}

Now that we have a description of the flexible vertices, we can apply results of \cite{GuiLev2} to get that the tree of cylinders of the $(\mc{A}, \mc{H})$--JSJ tree is compatible with every $(\mc{A}, \mc{H})$--tree.
\begin{defn}
 Let $G$ be a group. Two $G$--trees $T_1$ and $T_2$ are {\em compatible} if they have a common refinement: if there exists a tree $\hat T$ with collapse maps $\hat T \to T_i$.
\end{defn}

Let $\mc{A}_{nvc}$ denote the groups which belong to $\mc{A}$ that are not virtually cyclic. The proof of the following is essentially contained in \cite{GuiLev2}.

 \begin{thm} \label{t:JSJ}
Let $L$ be a $\mc{G}$--limit group which does not split over $\mc{A}_{\ufin}$ relative to $\mc{H}$. Let $T_a$ be a $(\mc{A}, \mc{H})$ JSJ-tree, and let $(T_a)_c$ be the corresponding tree of cylinders. Then $(T_a)_c$ is an $(\mc{A}, \mc{H}\cup\mc{A}_{nvc})$--JSJ tree which is compatible with every $(\mc{A}, \mc{H})$--tree.
\end{thm} 
\begin{proof}
Let $T$ be some $(\mc{A}, \mc{H})$--tree. Then by \cite[Lemma 2.8.(1)]{GuiLev2} $T_a$ has a refinement $S$ which dominates $T$.  By \cite[Lemma 7.14]{GuiLev2}, the tree of cylinders $S_c$ and $T$ have a common refinement $R$. By \cite[Lemma 7.15]{GuiLev2}, $S_c$ is a refinement of $(T_a)_c$, hence $R$ is a common refinement of $(T_a)_c$ and $T$.

The proof of \cite[Lemma 7.15]{GuiLev2} uses the fact that the flexible vertices of $(T_a)_c$ are either small in $(\mc{A}, \mc{H})$ trees or are QH with finite fiber. However, in our situation the ``flexible" part of flexible vertices are \QHAH by Corollary \ref{c:cut scc}, so the same proof works with only the obvious changes.
\end{proof}

\begin{defn} \label{d:JSJ}
 Suppose that $L$ is a $\mc{G}$--limit group which does not split over $\mc{A}_{\ufin}$ relative to $\mc{H}$.  The {\em JSJ-tree} of $L$ is the tree $(T_a)_c$ described in Theorem \ref{t:JSJ} above.  The graph of groups decomposition associated to the JSJ-tree is denoted $\AJSJ$ and referred to as the {\em JSJ-decomposition} of $L$. We also denote by $T_M$ the tree obtained from $(T_a)_c$ by taking the refinement of each non-(virtually abelian) flexible vertex group given by Lemma \ref{l:flexible}. We call $T_M$ the \emph{modular tree}, and the associated graph of groups decomposition is denoted by $\AMOD$ and referred to as the \emph{modular splitting} of $L$.
\end{defn}

Note that the vertex groups of the modular splitting may be characterized as \QHAH, virtually abelian and rigid  in the usual way.  Also note, however, that some rigid vertex groups of $\AMOD$ may in fact be subgroups of flexible vertex groups of $\AJSJ$. Additionally, it follows from the construction that every subgroup in $\mc{A}$ is either virtually cyclic or is elliptic with respect to the modular tree. 

\subsection{Modular automorphisms} \label{ss:shortening moves}

In this section, we introduce the automorphisms that are used in the shortening argument below and in Section \ref{s:short quotient}, which are called \emph{modular automorphisms}.  We prove that these automorphisms can all be ``seen" in the modular splitting defined in Definition \ref{d:JSJ} above.  This property justifies the use of the word ``modular" for this splitting. We start by describing modular automorphisms of splittings. 

We mostly follow the definitions and notation from \cite[Section 3]{ReiWei} here, since we reference their proof of the shortening argument.

In this subsection, we mostly consider graphs of groups instead of actions on trees. When we apply terminology from the previous section to a graph of groups we mean that the Bass--Serre tree dual to this graph of groups has the given property.

Suppose $G=A\ast_C B$ or $G=A\ast_C$ and $c\in Z_G(C)$. A \emph{Dehn twist by $c$} is an automorphism of $G$ which fixes $A$ and conjugates each element of $B$ by $c$ in the amalgamated product case, or an automorphism which fixes $A$ and sends the stable letter $t$ to $tc$ in the HNN--extension case. When $\mathbb A$ is a graph of groups decomposition and $e$ is an edge of $\mathbb A$, then a \emph{Dehn twist over $e$} is a Dehn twist in the one edge splitting corresponding to collapsing all edges of $\mathbb A$ other then $e$.

Let $G_v$ be a vertex group of a graph of groups decomposition of a group $G$. Then any automorphism of $G_v$ which fixes the adjacent edge group up to conjugacy can be extended to an automorphism of $G$ (see \cite[Definition 3.13]{ReiWei}). We call such an extension the {\em natural extension} of the automorphism. 

Let $L$ be a $\mc{G}$--limit group. Recall by Lemma \ref{aasub} that a virtually abelian subgroup $H$ of $L$ which is not absolutely elliptic contains a unique, maximal subgroup $H^+$ of index at most 2 such that $H^+$ is finite-by-abelian.

Suppose $\mathbb A$ is a splitting of $L$ and $L_v$ is a vertex group of $\mathbb A$ which is virtually abelian and not absolutely elliptic. Let $E(L_v)$ denote the subgroup of $L_v$ generated all edge groups adjacent to $v$, and let $E(L_v^+)=E(L_v)\cap L_v^+$. Let $\Hom_E(L^+_v, \Z)=\{\phi\in \Hom(L^+_v, \Z)\;|\; \phi(g)=0 \;\forall\; g\in E(L_v^+)\}$. Let $P^+_v=\{g\in L^+_v\;|\; \phi(g)=0 \;\forall\; \phi\in \Hom_E(L^+_v, \Z)\}$. Then $L_v^+/P^+_v$ is a finitely generated free abelian group.
Furthermore, if $L_v$ is finitely generated, then $P^+_v$ is the minimal direct factor of $L_v^+$ which contains $E(L_v^+)$ and the torsion subgroup of $L_v^+$.

\begin{defn}\cite[Definition 3.15]{ReiWei}
Let $L$ be a $\mc{G}$--limit group and let $\mathbb A$ be a virtually abelian splitting of $L$. The modular automorphism group of $L$ relative to $\mathbb A$, denoted $\Mod_{\mathbb A}(L)$, is the subgroup of $\Aut(L)$ generated by the following types of automorphisms:
\begin{enumerate}
\item Inner automorphisms.

\item Dehn twists over finite-by-abelian, non-(absolutely elliptic) edge groups $L_e$ of $\mathbb A$ by an element $c\in Z(M^+_{L_e})$.

\item Natural extensions of automorphisms of non-(absolutely elliptic) maximal virtually abelian vertex groups $L_v$ of $\mathbb A$ which fix $P^+_v$ and which restrict to conjugation on each subgroup $U\leq A_v$ with $U^+=P_v^+$.

\item Natural extensions of automorphisms of \QHAH-vertex groups of $\mathbb A$ induced by homeomorphisms of the underlying orbifold which fix the boundary. 
\end{enumerate}
\end{defn}

For the remainder of the section, fix a divergent $\mc{G}$--limit group $L$ which does not split over $\mc{A}_\ufin$ relative to $\mc{H}$, and let $\AJSJ$ and $\AMOD$ be the graph of groups decompositions of $L$ as in Definition \ref{d:JSJ}. By Theorem \ref{t:JSJ}, $\AJSJ$ is compatible with every $(\mc{A}, \mc{H})$--decomposition of $L$.

\begin{thm} \label{t:all the splittings}
For any $(\mc{A}, \mc{H})$--spitting $\mathbb A$ of $L$, $\Mod_{\mathbb A}(L)\leq \Mod_{\AMOD}(L)$. Moreover, the same is true if $\mathbb A$ is any splitting of $L$ obtained from a limiting action on $\R$--trees as in Proposition \ref{prop:splitting}.
\end{thm}

The proof of Theorem \ref{t:all the splittings} is split into a number of smaller results, depending on the
type of modular automorphism.  However, we first make the following observation.

\begin{lem}
$\Mod_{\AJSJ}(L)\leq \Mod_{\AMOD}(L)$.
\end{lem}
\begin{proof}
This follows from the fact that the collapse map from $\AMOD$ to $\AJSJ$ only collapses absolutely elliptic edges of $\AMOD$ into the flexible vertices of $\AJSJ$ which are not virtually abelian or \QHAH-vertices of $\AJSJ$. 
\end{proof}
In general the inclusion in the above lemma may be strict, as the automorphisms arising from the \QHAH vertices of $\AMOD$ which arise as in Corollary \ref{c:cut scc} may not be modular automorphisms of $\AJSJ$, since these can lie in flexible vertices of $\AJSJ$ which are not \QHAH.

\begin{lem}\label{l:type 2}
Suppose $\alpha$ is a Dehn twist corresponding to a one edge $(\mc{A}, \mc{H})$--splitting of $L$. Then $\alpha\in \Mod_{\AMOD}(L)$. 
\end{lem}
\begin{proof}
Let $\mathbb B$ denote the one edge splitting corresponding to $\alpha$. By Theorem \ref{t:JSJ}, there is a refinement $\mathbb A^\prime$ of $\AJSJ$ and a collapse map $\mathbb A^\prime\to\mathbb B$. Clearly $\alpha\in \Mod_{\mathbb A^\prime}(L)$.  If we choose $\mathbb A^\prime$ to be the minimal common refinement of $\AJSJ$ and $\mathbb B$ then either $\mathbb A^\prime = \AJSJ$ or else the collapse map from $\mathbb A^\prime$ to $\AJSJ$ collapses a single edge.  In the first case, we have already observed that $\alpha \in \Mod_{\AJSJ}(L) \le \Mod_{\AMOD}(L)$.

It remains to consider the case where the collapse map from $\mathbb A^\prime$ to $\AJSJ$ collapses a single edge, which is necessarily the edge along which $\alpha$ twists.  In this case, the edge group $L_e$ is a subgroup of a flexible vertex group $L_v$ of $\AJSJ$, so $L_v$ is either virtually abelian or else as described in Lemma \ref{l:flexible}.
In the first case $\alpha$ is a modular automorphism for $\AJSJ$ of type (3) and so $\alpha\in \Mod_{\AJSJ}(L) \le \Mod_{\AMOD}(L)$. In the second case, Corollary \ref{c:cut scc} implies that the refinement $\mathbb A^\prime$ corresponds to a splitting of a \QHAH-vertex of $\AMOD$, hence the corresponding Dehn twist $\alpha$ is an automorphism of type (4) in $\Mod_{\AMOD}(L)$.
\end{proof}

\begin{lem} \label{l:type 3}
Let $\mathbb A$ be an $(\mc{A}, \mc{H})$--splitting of $L$ and let $\alpha\in \Mod_{\mathbb A}(L)$ be a modular automorphism of type (3). Then $\alpha\in \Mod_{\AJSJ}(L)$.
\end{lem}
\begin{proof}
Let $L_v$ be the vertex group of $\mathbb A$ corresponding to $\alpha$. First, since we are assuming that $L$ does not split over $\mathbb A_\ufin$, so if $L_v$ is virtually cyclic then the edge groups adjacent to $v$ are infinite and hence finite index in $L_v$. It follows that $P_v^+$ is a subgroup of index at most 2 in $L_v$, so $\alpha$ must act as conjugation on $L_v$ by an element of the centralizer of $P_v^+$, which naturally extends to a Dehn twist of $L$, which belongs to $\Mod_{\AJSJ}$ by (the proof of) Lemma \ref{l:type 2}.

 Now suppose $L_v$ is not virtually cyclic. Since $L_v$ is a maximal virtually abelian subgroup of $L$, there is some vertex $u$ of $\AJSJ$ such that  $L_v=L_u$. Hence it suffices to show that $P^+_u\subseteq P^+_v$. 

Since $\AJSJ$ is compatible with every $(\mc{A}, \mc{H})$--splitting of $L$, there is a refinement $\mathbb A^\prime$ of $\AJSJ$ and a collapse map $\mathbb A^\prime\to\mathbb A$. We assume that $\mathbb A^\prime$ is the minimal such refinement. Let $u^\prime$ be a vertex of $\mathbb A^\prime$ which maps to $v$. 

First consider the collapse map $\mathbb A^\prime\to \mathbb A$. Suppose that there is an edge $e$ of $\mathbb A^\prime$ connecting some vertex $w$ to $u^\prime$ such that $e$ is collapsed under this map. Then the vertex group obtained from $L_{u^\prime}$ by collapsing $e$ is virtually abelian since it is contained in $L_v$, this means that $L_w=L_e$. However, it follows from the construction of the tree of cylinders that edges connecting virtually abelian vertex groups do not occur in $\AJSJ$, so $e$ is also collapsed under the map $\mathbb A^\prime\to\AJSJ$. But then the splitting obtained from $\mathbb A^\prime$ by collapsing $e$ is also a refinement of both $\mathbb A$ and $\AJSJ$, which contradicts the minimality of the refinement $\mathbb A^\prime$. Hence we can conclude that there are no edges adjacent to $u^\prime$ which are collapsed under the map $\mathbb A^\prime\to \mathbb A$, thus $P_{u^\prime}^+=P_v^+$.

Now consider the collapse map $\mathbb A^\prime\to\AJSJ$. Clearly $u^\prime$ has to map to $u$ since $L_{u^\prime}=L_v=L_u$. Suppose that is there is an edge $e$ connecting $u^\prime$ and some vertex $w$ such that $e$ gets collapsed under the map from $\mathbb A^\prime$ to $\AJSJ$. As before, this means that we must have $L_w=L_e\leq L_{u^\prime}$. After collapsing the edge $e$, the new edge groups which become attached to the vertex $u^\prime$ correspond to edges adjacent to $w$ other then $e$. So the corresponding edge groups are contained in $L_w=L_e$, which means that the peripheral structure of $L_{u^\prime}$ could only decrease under such a collapse. Furthermore, collapsing edges of $\mathbb A^\prime$ which are not adjacent to $u^\prime$ has no effect on the peripheral structure of $L_{u^\prime}$. Therefore, $P^+_{u} \subseteq P^+_{u^\prime}=P_v^+$.

\end{proof}

\begin{lem} \label{l:Surface short}
Let $\mathbb A$ be a virtually abelian splitting of $L$.  Then any $\alpha \in \Mod_{\mathbb A}(L)$ of type (4) belongs to $\Mod_{\AMOD}(L)$.
\end{lem}
\begin{proof}
This follows quickly from the definition of \QHAH-subgroups, Lemma \ref{l:type 2} and the fact that the mapping class group is generated by Dehn twists.
\end{proof}

We have now proved the first assertion of Theorem \ref{t:all the splittings}.

The second assertion of Theorem \ref{t:all the splittings}, where $\mathbb A$ is a splitting arising from a limiting action on an $\R$--tree is now straightforward, since automorphisms of type (2) and (3) are associated to $(\mc{A},\mc{H})$--splittings and Lemma \ref{l:Surface short} was proved for any virtually abelian splitting.

In light of Theorem \ref{t:all the splittings}, when $L$ is a divergent $\mc{G}$--limit group we define $\Mod(L):=\Mod_{\AMOD}(L)$.
\subsection{The Shortening Argument}\label{s:shortarg}
\begin{defn}
Let $L$ be a divergent $\mc{G}$--limit group. Then $\phi\in \Hom(L, \mc{G})$ is called \emph{short} if for all $\alpha\in \Mod(L)$, $\|\phi\|\leq\|\phi\circ\alpha\|$.
\end{defn}

\begin{thm}\label{shortarg}
Suppose $G$ is a finitely generated group which does not split over any finite subgroup of order at most $2C$. Let $\phi_i\colon G\to\Gamma_i$ be such that $\Gamma_i \in \mc{G}$,  $\ker^\omega(\phi_i)=\{1\}$, and $\lim^\omega \|\phi_i\|=\infty$. Then $\omega$--almost surely $\phi_i$ is not short.
\end{thm}

The proof of this theorem is essentially the same as the proof which appears in \cite[Proposition 4.6]{ReiWei}. The details of the proof are quite technical but they are carefully worked out in \cite[Section 4.2]{ReiWei}, so we only give a brief sketch of how the proof works and refer to \cite{ReiWei} for the details.

\begin{proof}[Sketch proof]

Suppose $\phi_i\colon G\to\Gamma_i$ such that $\Gamma_i \in \mc{G}$,  $\ker^\omega(\phi_i)=\{1\}$, and $\lim^\omega \|\phi_i\|=\infty$. Note that this means that $G$ itself is a divergent $\mc{G}$--limit group since $G=G/\ker^\omega(\phi_i)$. In particular, $\Mod(G)$ is well-defined. Throughout the proof, we denote $X_{\Gamma_i}$ by $X_i$. Let $T$ be the limiting $\R$--tree associated to the sequence $(\phi_i)$ as in Theorem \ref{t:limiting R-tree}. Recall that $T$ is a minimal subtree of $\lim^\omega (X_i/ \|\phi_i \|, o_i)$, where $o_i\in X_i$ such that $\max\limits_{s\in S} d_i(o_i, \phi_i(s)o_i)\leq \|\phi_i\|+\frac1i$. Let $o=\lim^\omega(o_i)$.

We let $\mathbb A$ be the splitting of $G$ associated to the action of $G$ on $T$ described in Proposition \ref{prop:splitting}. Then the modular automorphism we find which shorten the actions of $G$ are going to be elements of $\Mod_{\mathbb A}(G)$ which are hence contained in $\Mod(G)$ by Theorem \ref{t:all the splittings}.

Let $S$ be a finite generating set of $G$. Then for some $g\in S$, $[o, go]$ is a non-trivial arc in $T$. By Theorem \ref{t:Rmach applies} and Lemma \ref{l:goa-tran}, $T$ has a transverse covering by pieces of either Seifert, axial, or simplicial type. Then the interval $[o, go]$ must have non-degenerate intersection with at least one piece $Y$ of the transverse covering, and $Y$ is either Seifert, axial, or simplicial type. Each case is considered separately.

Suppose $Y$ is either a Seifert or an axial piece in the transverse cover of the tree $T$. Then there is a modular automorphism $\alpha$ such that for any $h\in S$, if $[o, ho]$ intersects $Y$ in a non-degenerate segment then $d(o, \alpha(h)o)<d(o, ho)$, and otherwise $d(o, \alpha(h)o)=d(o, ho)$. This allows us to shorten the action of $g$ on the limit tree $T$ without affecting the other generators in the cases where $[o, go]$ intersects a Seifert or an axial piece. Since the actions of $\phi_i(G)$ on $X_{i}$ are converging to the action on the limit tree, $\alpha$ $\omega$--almost surely also shortens the action of $g$ on the spaces $X_i$. If $Y$ is an axial piece, $\alpha$ is an automorphism of type (3) and if $Y$ is an Seifert piece, then $\alpha$ is an automorphism of type (4). Finding the automorphism in the axial case is described in \cite[Section 4.2.1]{ReiWei}, and finding the automorphism in the Seifert case is described in \cite[Section 4.2.2]{ReiWei}.

Suppose now that $Y$ is a simplicial piece. In this case, we do not shorten the action of $g$ on the limit tree $T$, but instead we directly shorten the action of $g$ on $X_i$ $\omega$--almost surely. That is, given an edge $e$ in $Y$, for an $\omega$--large set of $i$ we can find a modular automorphism of type (1), i.e. a Dehn twist $\alpha_i$ over the edge $e$ such that $\alpha_i$ shortens the action of each generator $h\in S$ on $X_i$ for which $e\subseteq [o, ho]$ and does not affect the actions of the other generators. Finding the Dehn twist in the simplicial case is in \cite[Section 4.2.3]{ReiWei}.

Now, for each generator $g\in S$ either $go=o$ or there is an $\omega$--large set of $i$ such that there is a modular automorphism $\alpha_i$ which shortens the action of $g$ on $X_i$ and which either shortens or does not affect the actions of the other generators. Since $S$ is finite, we can take $\beta_i$ to be the product of one such modular automorphism for each generator, and hence $\omega$--almost surely $\|\phi_i\circ\beta_i\|<\|\phi_i\|$.

Observe that in \cite{ReiWei} the definition of $\|\phi\|$ is slightly different. We defined $\| \phi_i \|$ to be $\inf_{x\in X_i}\max_{s\in S} d_i(x, \phi_i(s)x)$, while in \cite{ReiWei} they replace the maximum over the generators with a sum over the generators. The only difference this makes to the proof is that they can use an automorphism which shortens a single generator, while we take a product of one shortening automorphism for each generator which does not fix the basepoint .

It is also worth noting that we are only assuming that $L$ doesn't split over $\mc{A}_\ufin$ relative to $\mc{H}$, so it may be that $L$ is not one-ended.  However, the only time those splittings arise in our construction are times when finite edge stabilizers correspond to points in the limiting $\R$--tree, and so do not need to be shortened.   Therefore, we can safely ignore splittings of $L$ over finite groups.
\end{proof}

\section{Shortening quotients and reduction to non-divergent limit groups} \label{s:short quotient}

In this section, we continue to fix $\mc{G}$, a uniformly acylindrically hyperbolic family of groups with hyperbolicity constant $\delta$ and acylindricity constants $R_\e$, $N_\e$.

We start by recording some properties of non-divergent $\mc{G}$--limit groups. When $\mc{G}$ consists of a single hyperbolic group $\Gamma$ acting on its Cayley graph, the non-divergent $\Gamma$--limit groups are all isomorphic to subgroups of $\Gamma$. This is no longer true in the acylindrically hyperbolic setting, but we can say a few things.

\begin{prop}
Suppose $L$ is a non-divergent $\mc{G}$--limit group. Then $L$ admits an acylindrical action on a hyperbolic metric space with the same acylindricity and hyperbolicity constants as $\mc{G}$. 
\end{prop}
\begin{proof}
Suppose that $(\Gamma_i, X_i) \in \mc{G}$ for $i \ge 1$, such that $\phi_i \co G \to \Gamma_i$ and that $L = G / \ker^\omega(\phi_i)$.  Fix basepoints $o_i\in X_i$ as before. Then $\lim^\omega(\|\phi_i\|)<\infty$ implies that $L$ has a well-defined action on $X_\infty:=\lim^\omega(X_i, o_i)$, which is $\delta$--hyperbolic by Lemma \ref{lem:hypultralimit}. Note that in the non-divergent case, we are taking the ultralimit of the spaces $X_i$ without rescaling the metrics.

 Let $\e>0$ and let $x_\infty=\lim^\omega x_i$ and $y_\infty=\lim^\omega y_i$ be distinct points of $X_\infty$ with $d(x_\infty, y_\infty)> R_\e$. Then $\omega$--almost surely, $d_i(x_i, y_i)>R_\e$. If $g\in L$ and $\widetilde{g}\in\phi_\infty^{-1}(g)$ such that $d(gx_\infty, x_\infty)<\e$, then $\omega$--almost surely $d_i(\phi_i(\widetilde{g})x_i, x_i)<\e$. Now if $F$ is a finite set of distinct elements of $L$, then $\omega$--almost surely $\phi_i(\widetilde{g})\neq\phi_i(\widetilde{h})$ for all $g, h\in F$. It follows that from these observations that if $F\subseteq L$ is finite and 
\[
d(x_\infty, gx_\infty)<\e, d(y_\infty, gy_\infty)<\e
\]
for all $g\in F$, then there exists an index $i$ such that for all $g, h\in F$, $\phi_i(\widetilde{g})\neq\phi_i(\widetilde{h})$,  $d(x_i, \phi_i(\widetilde{g})x_i)<\e$, $d(y_i, \phi_i(\widetilde{g})y_i)<\e$, and $d_i(x_i, y_i)>R_\e$. By the acylindricity of the action of $\Gamma_i$ on $X_i$, we have $|F|\leq N_\e$.
\end{proof}

\begin{lem}\label{lem:non-div aasub}
Let $L$ be a $\mc{G}$--limit group defined by $L=G/\ker^\omega(\phi_i)$. If $\lim^\omega\|\phi_i\|<\infty$, then every virtually abelian subgroup in $\mathcal{NLSE}(L)$ is virtually cyclic. 
\end{lem}
\begin{proof}
Let $M=\lim^\omega \|\phi_i\|$. It suffices to prove that any abelian subgroup in $\mc{NLSE}(L)$ is virtually cyclic.  Let $H \in \mc{NLSE}(L)$ and let $H_1, H_2, \ldots$ be in increasing sequence of finitely generated subgroups of $H$ so that $H = \cap H_k$.  For each $k$, let $\widetilde{H}_k \le G$ be a lift of $H_k$ so that $\widetilde{H}_{k-1} \le \widetilde{H}_{k}$.

By Lemma \ref{aasub} $\omega$--almost surely there exist $g_i \in \Gamma_i$ so for each $k$ $\phi_i(\widetilde{H}_k) \le E_{\Gamma_i}(g_i)$.  By Lemma \ref{ufin} there are only finitely many isomorphism types of subgroups of the form $E_{\Gamma}(g)$ with $\Gamma \in \mc{G}$ and $g$ loxodromic.  Therefore, $\omega$--almost surely $E_{\Gamma_i}(g_i)$ is isomorphic to some fixed virtually cyclic group $E$.  Let $Z$ denote a finite generating set for $E$, let $E \to E_{\Gamma_i}(g_i)$ be some isomorphism, and let $Z_i$ be the image of $Z$ under this isomorphism.  Since $H_k$ is abelian, it must have abelian image in $E_{\Gamma_i}(g_i)$, so we may assume without loss of generality that $E$ has a finite normal subgroup with quotient infinite cyclic.

By  \cite[Lemma 2.2]{Bow2} there exists a constant $d$ which depends only on the hyperbolicity and acylindricity constants so that for any loxodromic element $g \in \Gamma \in \mc{G}$ and any $x\in X_\Gamma$ we have $\tau(g) := \lim_{n \to \infty} \frac{1}{n} d_{X_\Gamma}(x,g^nx) > d$.  In particular, we have
$d_{X_{\Gamma_i}}(o_i, g_i^no_i) \ge nd$.  If follows that there are constants $C$ and $D$ (depending only on the hyperbolicity and acylindricity constants) so that for any $h \in E_{\Gamma_i}(g_i)$ we have
$d_i(o_i,ho_i) \ge C|h|_{Z_i} - D$.

First consider elements of $H_k$ of finite order.  They map to elements of $E_{\Gamma_i}(g_i)$ of finite order, and there are a uniformly bounded number of these.  From this, it follows that there are a uniformly bounded number of elements of $H_k$ of finite order, and these form a finite normal subgroup of $H_k$.

Suppose $a,b \in H_k \le L$ have infinite order, and choose lifts $\widetilde{a}$ and $\widetilde{b}$ such that  $|\widetilde{a}|_S=|a|_{\phi_\infty(S)}$ and $|\widetilde{b}|_S=|b|_{\phi_\infty(S)}$. Then $d_i(o_i, \phi_i(\widetilde{a})o_i) \leq |\widetilde{a}|_S(\|\phi_i\|+1)$ and $d_i(o_i, \phi_i(\widetilde{b})o_i) \leq |\widetilde{b}|_S(\|\phi_i\|+1)$. 

 Therefore,
\[
|\phi_i(\widetilde{a})|_{Z_i}\leq \frac{1}{C} d_i(o_i, \phi_i(\widetilde{a})o_i) +D\leq \frac{M+1}{C}|a|_{\phi_\infty(S)}+D
\]
and
\[
|\phi_i(\widetilde{b})|_{Z_i}\leq \frac{1}{C} d_i(o_i, \phi_i(\widetilde{b})o_i) +D\leq \frac{M+1}{C}|b|_{\phi_\infty(S)}+D
\]

Since the bound on the word length of these elements in $E$ is independent of $i$, there exist some $r, s\in\mathbb Z$ such that $\omega$--almost surely  $\phi_i(\widetilde{a}^r)=\phi_i(\widetilde{b}^s)$, and hence $a^r=b^s$.  It follows that $a$ and $b$ belong to the same cyclic subgroup of $H_k$. Furthermore, if $x\in H_k$ such that $x^n=a$, then $\omega$--almost surely $\phi_i(\widetilde{x}^n)=\phi_i(\widetilde{a})$. However, such an $x$ can only exist for a bounded set of $n$, where the bound depends only on $E$ and $|\phi_i(\widetilde{a})|_{Z_i}$. It follows that  $a$ can have only finitely many roots in $H_k$, and hence $H_k$ must contain a primitive element $h$. Since any two elements of $H_k$ of infinite order belong to a common cyclic subgroup, we must have that all infinite order elements of $H_k$ belong to $\langle h\rangle$. Since there are a uniformly bounded number of elements of $H_k$ of finite order, and all infinite order elements lie in a cyclic subgroup, $H_k$ is virtually cyclic, with a finite normal subgroup of uniformly bounded rank.  

Now consider the sequence of groups $H_k \le H_{k+1} \le \ldots$.  This is an increasing sequence of virtually cyclic groups with finite normal subgroups of bounded rank.  Therefore, $\omega$--almost surely the finite normal subgroups are all the same subgroup $N$ and we have cyclic groups $H_k/N \le H_{k+1}/N \le \ldots$.  The same argument as above bounds the size of the set of finite order elements of $H_j$ independently of $j$, and also, for infinite order $a$ in $H_k$, bounds the index of a root of $a \in H_k$ in any $H_j$ for $j \ge k$.  This implies that the ascending sequence of cyclic groups stabilizes, which implies that $H$ is virtually cyclic, as required.
\end{proof}

If we knew that $\mc{G}$--limit groups were fully residually $\mc{G}$, then we could define a divergent $\mc{G}$--limit group $L$ via a sequence of homomorphisms $(\phi_i \co L \to \Gamma_i)$ with $\ker^\omega(\phi_i) = \{ 1 \}$.  We could then define an equivalence relation on homomorphisms from $L \to \mc{G}$ by precomposing with elements of $\Mod(L)$ and post-composing by conjugation in the element of $\mc{G}$.  If we take a sequence $(\eta_i \co L \to \Gamma_i)$ where each $\eta_i$ is shortest in its equivalence class, the shortening argument implies that $\ker^\omega(\eta_i) \ne \{ 1 \}$, i.e. $L/\ker^\omega(\eta_i)$ is a proper quotient of $L$, which we call a \emph{shortening quotient}.

However, we do not know that $\mc{G}$--limit groups are fully residually $\mc{G}$, so in order to construct shortening quotients in general we have to proceed in a more complicated way.

The following is essentially the same as \cite[Lemma 6.1]{ReiWei}, and the proof which appears there works in our situation without any changes.  The idea of this construction comes from \cite{sela:dio1}, and a similar construction also appears in \cite{Gro}.
\begin{lem} \label{l:ugly hack}
Let $G$ be finitely presented and let $L=G/\ker^\omega(\phi_i)$ be a divergent $\mc{G}$--limit group which does not split over $\mc{A}_{\ufin}$ relative to $\mc{H}$. Let $\AMOD$ be the modular splitting of $L$ from Definition \ref{d:JSJ}. Then there exists a sequence of finitely presented groups $G=W_0, W_1,...$ and epimorphisms $f_i\colon W_i\to W_{i+1}$ and $h_i\colon W_i\to L$ for $i\geq 0$ such that:
\begin{enumerate}
\item $\phi_\infty=h_0$.
\item $h_i=h_{i+1}\circ f_i$. for all $i\geq 1$
\item\label{eq:direct limit}  $L$ is the direct limit of the sequence $G\twoheadrightarrow W_1\twoheadrightarrow...$.  Equivalently, 
$$\ker^\omega(\phi_i)=\bigcup_{k=1}^\infty\ker(f_{k-1} \circ \ldots \circ f_0)$$
\item Each $W_i$ for $i\geq 1$ has a graph of groups decomposition $\mathbb A_i$ whose underlying graph is isomorphic to the underlying graph of $\AMOD$. 
\item If $V$ is a vertex group of $\AMOD$ and $V_i$ is the corresponding vertex group of $\mathbb A_i$, then $h_i(V_i)\subseteq V$. Furthermore,
\begin{enumerate}
\item $V=\bigcup\limits_{i=1}^\infty h_i(V_i)$
\item If $V$ is \QHAH--vertex group, then $h_i|_{V_i}$ is an isomorphism onto $V_i$
\item If $V$ is a virtually abelian vertex group, then $h_i|_{V_i}$ is injective.
\end{enumerate}

\item If $E$ is an edge group of $\AMOD$ and $E_i$ is the corresponding edge group of $\mathbb A_i$, then $h_i$ maps $E_i$ injectively into $E$ and $E=\bigcup_{i=1}^\infty h_i(E_i)$.
\end{enumerate}

\end{lem}

Let $\xi_i\colon G\to W_i$ be the natural map, that is $\xi_i=f_{i-1} \circ...\circ f_0$. Note that for a fixed $j$, since $W_j$ is finitely presented $\ker(\xi_j)\subseteq \ker(\phi_i)$ for an $\omega$ large set of $i$. Hence after passing to a subsequence of $(W_j,\xi_j)$ and re-indexing, we can assume that for all $i$ we have $\ker(\xi_i)\subseteq\ker(\phi_i)$. This means that the map $\phi_i$ factors through $\xi_i$, that is 
\[
\phi_i=\lambda_i\circ\xi_i
\]
for some $\lambda_i\colon W_i\to \Gamma_i$.

Note that if a vertex group of $\AMOD$ is virtually abelian (respectively, a \QHAH--vertex group), then the corresponding vertex group of $\mathbb A_i$ is virtually abelian (respectively, \QHAH). We refer to all other vertex groups of $\mathbb A_i$ as rigid since they correspond to rigid vertex groups of $\AMOD$. Hence we can define an equivalence relation on $\Hom(W_i, \Gamma_i)$ using modular automorphism with respect to $\mathbb A_i$ in the same way as before.  Replacing each $\lambda_i$ with the shortest element of its equivalence class $\widehat{\lambda}_i$ gives a new sequence
\[
\eta_i\colon=\widehat{\lambda}_i\circ\xi_i.
\]
Define $Q=G/\ker^\omega(\eta_i)$. Since by Lemma \ref{l:ugly hack}.\eqref{eq:direct limit} and construction we have $\ker^\omega(\phi_i)=\ker^\omega(\xi_i)\subseteq\ker^\omega(\eta_i)$, there is a natural map $\pi\colon L\to Q$, and $Q$ is called a \emph{shortening quotient} of $L$.

\begin{lem}\label{lem:shortq1}
 Let $L$ be a divergent $\mc{G}$--limit group which does not split over $\mc{A}_{\ufin}$ relative to $\mc{H}$ and let $\pi\colon L\to Q$ be a shortening quotient, via the construction (and with the notation) described above. Then $\pi$ is injective on the rigid vertex groups of the modular splitting of $L$. Furthermore, if $\lim^\omega(||\eta_i||)=\infty$ (so $Q$ is a divergent $\mc{G}$--limit group), then $\pi$ is a proper (non-injective) quotient map. 

\end{lem}

\begin{proof}
Suppose $g\in L\setminus\{1\}$ belongs to a rigid vertex group of $L$ and $\widetilde{g}\in\phi_{\infty}^{-1}(g)$. Then $g_i:=\xi_i(\widetilde{g})$ belongs to the corresponding vertex group of $W_i$ for sufficiently large $i$. This means that $\widehat{\lambda_i}(g_i)$ is conjugate to $\lambda_i(g_i)$, since modular automorphisms act by conjugation in $W_i$ on rigid vertex groups. Since $\widetilde{g}\notin\ker^\omega(\phi_i)$ and $\phi_i=\lambda_i\circ\xi_i$, we get that $\omega$--almost surely $\lambda_i(g_i)\neq 1$ and hence $\omega$--almost surely $\widehat{\lambda_i}(g_i)\neq1$. Since $\eta_i=\widehat{\lambda}_i\circ\xi_i$, we get that $\eta_{\infty}(\widetilde{g})\neq 1$, hence $\pi(g)=\eta_{\infty}(\widetilde{g})\neq 1$.

If $L=Q=G/\ker^\omega(\eta_i)$ and $\|\eta_i\|$ is unbounded, then  a natural adaptation of the shortening argument (see Section \ref{s:shortarg}) can be applied to the action of $L$ on the limiting $\R$--tree coming from the sequence $(\eta_i)$, and this contradicts the shortness of the $\widehat{\lambda_i}$ for an $\omega$--large set of $i$.
\end{proof}

\begin{lem}\label{lem:shortq2}
 Let $\pi\colon L\to Q$ be a shortening quotient of a divergent $\mc{G}$--limit group $L$ which does not split over $\mc{A}_{\ufin}$ relative to $\mc{H}$, and suppose every non-(absolutely elliptic) virtually abelian subgroup of $L$ is finitely generated. If $\eta_i$ $\omega$--almost surely factors through $\eta_\infty$, then $\phi_i$ $\omega$--almost surely factors through $\phi_\infty$.
\end{lem}
\begin{proof}
By assumption, all edge groups and hence all vertex groups of the modular decomposition of $L$ are finitely generated. Since \QHAH--vertex groups are finitely presented and finitely generated virtually abelian groups are finitely presented, if follows that $L$ is finitely presented relative to rigid vertex groups $P_1,...,P_n$. Since each $P_j$ is finitely generated, it has a lift $\widetilde{P_j}$ such that $\xi_i(\widetilde{P}_j)$ is contained in the corresponding vertex group of $\mathbb A_i$ for all sufficiently large $i$. Since $\phi_i$ and $\eta_i$ only differ by conjugation on these vertex groups, we see that $\omega$--almost surely $\ker(\eta_i|_{\widetilde{P}_j})=\ker(\phi_i|_{\widetilde{P}_j})$.

Now, $\ker(\phi_\infty)\subseteq \ker(\eta_\infty)$ by construction and by hypothesis $\omega$--almost surely we have $ \ker(\eta_\infty)\subseteq\ker(\eta_i)$. Hence $\ker(\phi_\infty)\subseteq\ker(\eta_i)$ $\omega$--almost surely. Thus, for each $P_j$, $\omega$--almost surely so $\ker(\phi_\infty|_{\widetilde{P}_j})\subseteq\ker(\eta_i|_{\widetilde{P}_j})=\ker(\phi_i|_{\widetilde{P}_j})$. Now Lemma \ref{lem:relfpen} implies that $\phi_i$ $\omega$--almost surely factors through $\phi_\infty$ .
\end{proof}

Suppose now that $L$ does split over $\mc{A}_\ufin$ relative to $\mc{H}$. In this case, we start by taking a Linnell decomposition $\mathbb D_L$, which is a $(\mc{A}_\ufin,\mc{H})$ JSJ-decomposition in the sense of Guirardel-Levitt (see \cite{Linnell} and \cite{GuiLev2}). In particular, each vertex group of $\mathbb D_L$ does not split over any subgroup of $\mc{A}_\ufin$ relative to $\mc{H}$. Hence we can apply the above construction of shortening quotients to each vertex group of $\mathbb D_L$, then ``glue" the shortening quotients together in the natural way to form a shortening quotient of $L$, and this shortening quotient of $L$ is a $\mc{G}$--limit group. This procedure is the same as that described after the proof of \cite[Proposition 6.2]{ReiWei}. 

The proof of the following lemma is contained in the proof of \cite[Proposition 6.3]{ReiWei}.
\begin{lem}\label{l:Linnell decomp}
Using the above description of shortening quotients, Lemmas \ref{lem:shortq1} and \ref{lem:shortq2} hold without the assumption that $L$ does not split over $\mc{A}_\ufin$ relative to $\mc{H}$.
\end{lem}

We next want to understand when $\mc{G}$--limit groups satisfy the hypothesis of Lemma \ref{lem:shortq2}, that is when all non-absolutely elliptic virtually abelian subgroups are finitely generated.

\begin{lem}\label{l:infgen subgroup}
Suppose a $\mc{G}$--limit group $L$ contains a non-absolutely elliptic virtually abelian subgroup which is not finitely generated. Then there exists an infinite sequence of divergent $\mc{G}$--limit groups
\[
L\overset{\alpha_1}{\onto} S_1\overset{\alpha_2}{\onto} \cdots ,
\]
such that each $\alpha_j$ is a proper quotient map.
\end{lem}

\begin{proof}
 Let $A$ be a non-(absolutely elliptic) virtually abelian subgroup of $L$ which is not finitely generated. We first show that some rigid vertex group in the modular splitting of $L$ contains an infinitely generated non-(absolutely elliptic) virtually abelian subgroup.  By Theorem \ref{t:JSJ} and the construction of the modular splitting,  after replacing $A$ with a conjugate we can assume that $A$ is contained in a vertex group $L_v$ of the modular splitting of $L$. \QHAH--vertex groups do not contain infinitely generated virtually abelian subgroups, so $L_v$ is either virtually abelian or rigid. If $L_v$ is virtually abelian, it must be infinitely generated since it contains an infinitely generated subgroup $A$. Since $L$ is finitely generated, $L_v$ is finitely generated relative to the adjacent edge groups, so there is some edge $e$ adjacent to $v$ such that $L_e$ is infinitely generated. By construction of the modular splitting, the other vertex $u$ adjacent to $e$ must correspond to a rigid vertex group $L_u$. Hence, in either case $L$ must have a rigid vertex group which contains a non-absolutely elliptic virtually abelian subgroup $B$ which is not finitely generated.

Let $(\phi_i)$ from  $\Hom(F_n, \mc{G})$ be a divergent defining sequence for $L$ such that $B\in\mathcal{NLSE}(L)$ with respect to the sequence $(\phi_i)$. Note that $(\phi_i)$ must be a divergent sequence by Lemma \ref{lem:non-div aasub}. Let $S_1$ be a shortening quotient of $L$ constructed as above starting with the defining sequence $(\phi_i)$. Let $\alpha_1\colon L\onto S_1$ be the corresponding quotient map. Then $B$ embeds in $S_1$ by Lemmas \ref{lem:shortq1} and \ref{l:Linnell decomp}. Furthermore, because shortening moves act by conjugation on rigid vertex groups, $B\in\mc{NLSE}(S_1)$ with respect to the constructed defining sequence for $S_1$.  By Lemma \ref{lem:non-div aasub}, this defining sequence must be divergent, and hence $\alpha_1$ is a proper quotient map by Lemmas \ref{lem:shortq1} and \ref{l:Linnell decomp}.

Now $B$ is a non-(absolutely elliptic) virtually abelian subgroup of $S_1$ which is not finitely generated, and so we can repeat the above argument inductively to obtain a sequence

\[
L\overset{\alpha_1}{\onto} S_1\overset{\alpha_2}{\onto}...
\]
where each $S_j$ is a divergent $\mc{G}$--limit group and each $\alpha_j$ is a proper quotient map.
\end{proof}

 In order to ensure that there are no infinite descending sequences of $\mc{G}$--limit groups we need an assumption about the non-divergent $\mc{G}$--limit groups. Such an assumption is required for our approach because techniques using limiting $\R$--trees cannot be applied to non-divergent $\mc{G}$--limit groups.

The next theorem is similar to \cite[Theorem 13]{JalSel}.  See also \cite[Theorem 1.12]{sela:dio7}.

\begin{thm}\label{dcc}
Suppose that for any finitely generated group $G$ and any non-divergent sequence $(\phi_i)$ from $\Hom(G, \mc{G})$, $\phi_i$ factors through the limit map $\phi_\infty$ $\omega$--almost surely. Then there is no infinite sequence of $\mc{G}$--limit groups
\[
L_1\overset{\alpha_1}{\onto} L_2\overset{\alpha_2}{\onto}...
\]
such that each $\alpha_i$ is a proper quotient map.
\end{thm}
\begin{rem}
In Theorem \ref{t:EN up to non-div} below we prove that if $\mc{G}$ is a uniformly acylindrical family of groups satisfying the hypothesis of Theorem \ref{dcc} then $\mc{G}$ is an equationally noetherian family of groups.  The converse follows from Proposition \ref{eqnfamily}.
\end{rem}
\begin{proof}
Suppose towards a contradiction that some infinite descending sequence of $\mc{G}$--limit groups exists as above. By Lemma \ref{l:newgenset}, we can assume that there exists a sequence of $\mc{G}$--limit groups $R_1\overset{\beta_1}{\onto} R_2...$ with each $R_n=F_n/\ker^\omega(\phi_i^n)$ and $\phi^n_\infty\circ\beta_n=\phi^{n+1}_\infty$ for all $n$. We further assume that, for each $n\geq 1$, $R_{n+1}$ is chosen such that if $R_n\onto S\onto...$ is any infinite descending sequence of $\mc{G}$--limit groups with $S=F_n/\ker^\omega(\rho_i)$, then
\[
|\ker(\rho_\infty)\cap B_n|\leq |\ker(\phi^{n+1}_\infty)\cap B_n|
\]

where $B_n$ denote the ball of radius $n$ in $F_n$ with respect to the word metric.

Now choose a diagonal sequence $(\psi_i=\phi^i_{n_i})$, where $n_i$ is chosen such that $\ker(\psi_i)\cap B_i=\ker(\phi^i_\infty)\cap B_i$ and for some $g\in\ker(\phi^{i+1}_\infty)$, $g\notin\ker(\psi_i)$. Let $R_\infty=G/\ker^\omega(\psi_i)$; by construction, $\ker(\psi_\infty)=\bigcup_{j=1}^\infty\ker(\phi^j_\infty)$, that is $R_\infty$ is also the direct limit of the sequence  $F_n\onto R_1\onto R_2\onto...$.

Next we show that every descending sequence of $\mc{G}$--limit groups of the form 
\[
R_\infty\onto L_1\onto L_2...
\]
with proper quotient maps is finite. Indeed, if some element $g\in B_n$ maps trivially to $L_1$ but not to $R_\infty$, then $g$ must map non-trivially to $R_{n+1}$. But there is an infinite descending sequence $R_n\onto L_1\onto...$, so this contradicts our choice of $R_{n+1}$.

Since there is no infinite descending chain which starts at $R_\infty$, any sequence of proper shortening quotients of $R_\infty$ is finite. Thus, $R_\infty$ admits a finite sequence of proper shortening quotients
\[
R_\infty\twoheadrightarrow S_1...\twoheadrightarrow S_n.
\]
Clearly there can be no infinite descending chain starting at any $S_j$ for $1\leq j\leq n$, so by Lemma \ref{l:infgen subgroup} we get that all non-(absolutely elliptic) virtually abelian subgroups of $R_\infty$ and of each $S_j$ are finitely generated. Since $S_n=G/\ker^\omega(\eta_i^n)$ is non-divergent, by assumption $\eta^n_i$ $\omega$--almost surely factors through $\eta^n_\infty$. Hence we can inductively apply Lemma \ref{lem:shortq2} to get that $\psi_i$ $\omega$--almost surely factors through $\psi_\infty$. But then there exists $i$ such that $\ker(\psi_i)\subseteq\ker(\psi_\infty)\subseteq\ker(\phi^{i+1}_\infty)$, which contradicts our construction of $\psi_i$.

\end{proof}

As noted in the introduction, the following is our main technical theorem.

\begin{reptheorem}{t:EN up to non-div}
 Suppose that $\mc{G}$ is uniformly acylindrically hyperbolic family. Furthermore, suppose that whenever $G$ is a finitely generated group and $\eta_i \co G \to \mc{G}$ is a non-divergent sequence of homomorphisms, $\eta_i$ $\omega$--almost surely factors through the limit map $\eta_\infty$.  Then $\mc{G}$ is equationally noetherian.
\end{reptheorem}
\begin{proof}
Let $(\phi_i \co G \to \mc{G})$ be a divergent sequence of homomorphisms, and let $L$ be the associated limit group. By Theorem \ref{dcc}, any sequence of proper shortening quotients of $L$ eventually terminates in a limit group $S_n$ which is non-divergent by Lemmas \ref{lem:shortq1} and \ref{l:Linnell decomp}. If $(\eta_i)$ is the corresponding non-divergent sequence of homomorphisms, then $\eta_i$ factors through the limit map $\eta_\infty$ $\omega$--almost surely by assumption. Lemma \ref{l:infgen subgroup} together with Theorem \ref{dcc} imply that any non-(absolutely elliptic) virtually abelian subgroup of any $\mc{G}$--limit group is finitely generated. Hence we can inductively apply Lemmas \ref{lem:shortq2} and \ref{l:Linnell decomp} to the sequence of shortening quotients of $L$ to get that $\phi_i$ factors through $\phi_\infty$ $\omega$--almost surely, so $\mc{G}$ is equationally noetherian by Theorem \ref{eqnfamily}.   
\end{proof}

We begin with a straightforward application of Theorem \ref{t:EN up to non-div} to maps to free products.

\begin{defn}
Let $\mc{G}$ be a family of groups and let 
$\mc{G}^{\ast n}=\{\Gamma_1\ast...\ast\Gamma_n \mid \Gamma_1,...,\Gamma_n\in\mc{G}\}$.
 Let $\mc{G}^\ast=\bigcup\limits_{i=1}^\infty\mc{G}^{\ast n}$.
\end{defn}

By considering the action of an element of $\mc{G}^\ast$ on the Bass--Serre tree of the free splitting, we have the following obvious result.
\begin{lem} \label{lem:G^ast UAH}
The set $\mc{G}^\ast$ is a uniformly acylindrically hyperbolic family of groups.
\end{lem}

The following result is very similar in nature to the results in the first part of \cite{JalSel}, though this result does not appear there.
\begin{repcorollary}{t:G^ast EN}
If $\mc{G}$ is an equationally noetherian family of groups, then $\mc{G}^{\ast}$ is an equationally noetherian family of groups.
\end{repcorollary}

\begin{proof}
 Let $G$ be a finitely generated group (with a fixed finite generating set).  As noted in Lemma \ref{lem:G^ast UAH} above, $\mc{G}^\ast$ is a uniformly acylindrically hyperbolic family.
Hence, we can apply Theorem \ref{t:EN up to non-div}, and we only need to deal with non-divergent sequences of homomorphisms from $\Hom(G,\mc{G}^\ast)$.  

Let $(\phi_i)$ be a non-divergent sequence from $\Hom(G,\mc{G}^\ast)$.  The space of (projectivized) $G$--trees is compact, and so there is a limiting action of $L = G/ \ker^\omega(\phi_i)$ on a tree.  Since $\phi_i$ is non-divergent, this action is simplicial, and it is clear that edge stabilizers are trivial.

This implies that there is an induced free splitting of $L$.  Since $L$ is finitely generated, the vertex groups of this splitting are finitely generated, and so they are $\mc{G}$--limit groups.  It is now clear that the $\phi_i$ must $\omega$--almost surely factor through $\phi_\infty \co G \to L$, as required.
\end{proof}

As suggested in \cite{JalSel}, it is tempting to try to generalize Corollary \ref{t:G^ast EN} to $k$--acylindrical actions with vertex stabilizers in an equationally noetherian family.  Theorem \ref{t:EN up to non-div} still applies in this setting. However, it is hard to control the edge groups in the limiting tree in the non-divergent case, and in particular we do not know how to prove they are finitely generated.

In the next section, we provide a much more involved application of Theorem \ref{t:EN up to non-div} to relatively hyperbolic groups.

\section{Relatively hyperbolic groups}\label{s:rel hyp}

\subsection{Definitions and properties}
Let $\Gamma$ be a group and $\{P_1,...,P_n\}$ a collection of subgroups of $\Gamma$, called \emph{peripheral subgroups}. Let $\mc{P}=\bigcup_{i=1}^n(P_i\setminus\{1\})$. Let $A\subseteq\Gamma$ such that $\langle A\cup\mc{P}\rangle=\Gamma$.

Let $F(A)$ be the free group on the set $A$, and let $F=F(A)\ast P_1\ast...\ast P_n$ then there is a natural surjective homomorphism
\[
F\onto\Gamma
\]
which is the identity when restricted to $A$ and to each $P_i$. If the kernel of this homomorphism is normally generated by a set $R\subseteq F$, then we say that
\begin{equation}\label{relpres} \tag{$\dagger$}
\langle A, P_1,...,P_n\;|\; R\rangle
\end{equation}
is a \emph{presentation for $\Gamma$ relative to $\{P_1,...,P_n\}$}, or simply a relative presentation for $\Gamma$ if the peripheral subgroups are understood. We say that $\Gamma$ is \emph{finitely generated relative to  $\{P_1,...,P_n\}$} if there exists a finite set $A\subseteq \Gamma$ such that $\langle A\cup\mc{P}\rangle=\Gamma$. We say that $\Gamma$ is \emph{finitely presented relative to  $\{P_1,...,P_n\}$} if $\Gamma$ has a relative presentation of the form (\ref{relpres}) with both $A$ and $R$ finite.

Note that the sets $A$, $P_1$,...,$P_n$ are not necessarily disjoint in $\Gamma$, but they are disjoint when considered as subsets of $F$. We use the notation $A\sqcup\mc{P}$ when we are considering $A$ and $\mc{P}$ as subsets of $F$ to emphasize this disjointness. In particular, when considering words in the alphabet $A\sqcup\mc{P}$, we consider elements which belong to $A$ as distinct letters from elements that belong to any peripheral subgroup even when they represent the same element in $\Gamma$. Similarly, elements which belong to distinct peripheral subgroups in $F$ are considered as distinct letters in $A\sqcup\mc{P}$ even then they represent the same element of $\Gamma$.

Suppose $\Gamma$ has a relative presentation of the form (\ref{relpres}). Let $W$ be a word in the alphabet $A\sqcup\mc{P}$ such that $W=_\Gamma 1$. Then by the definition of $R$, there exists $r_1,...,r_k\in R$ and $f_1,...,f_k\in F$ such that
\[
W=_F\prod\limits_{i=1}^kf_i^{-1}r_if_i.
\]
The \emph{relative area of $W$}, denoted $\Area^{\rel}(W)$, is defined as the minimal $k$ such that $W$ can be expressed as above.
\begin{defn}\cite{Osi06a}
Let $\Gamma$ be a group and $\{P_1,...,P_n\}$ be a collection of subgroups of $\Gamma$. Then $\Gamma$ is {\em hyperbolic relative} to $\{P_1,...,P_n\}$ if $\Gamma$ has a finite relative presentation of the form (\ref{relpres}) and there exists a constant $K$ such that for any word $W$ in $A\sqcup\mc{P}$ with $W=_\Gamma 1$,
\[
\Area^{\rel}(W)\leq K|W|,
\]
where $|W|$ denotes the word length of $W$.
\end{defn}

Given a word $W$ in the alphabet $A\sqcup\mc{P}$, $W=_\Gamma 1$ if and only if there exists a \emph{van Kampen diagram} for $W$ over (\ref{relpres}), that is an oriented, connected, simply connected, planar $2$--complex $\Delta$ such that each edge of $\Delta$ is labeled by an element of $A\sqcup\mc{P}$, $\Lab(\partial\Delta)\equiv W$, and for each $2$--cell $\Pi$  of $\Delta$ either $\Lab(\Pi)$ belongs to $R$ or there exists $1\leq i\leq n$ such that each edge of $\partial\Pi$ is labeled by an element of $P_i$ and $\Lab(\partial\Pi)=_{P_i}1$. The cells whose labels belongs to $R$ are called $R$--cells. It is easy to see that $\Area^{\rel}(W)$ is equal to the minimal number of $R$--cells in a van Kampen diagram for $W$.

\begin{lem}\cite[Lemma 2.15]{Osi06a}\label{lem:mindiagram}
For any word $W$ in $A\sqcup\mc{P}$ with $W=_\Gamma 1$, there exists a van Kampen diagram $\Delta$ over (\ref{relpres}) such that every internal edge of $\Delta$ belongs to the boundary of an $R$--cell and the number of $R$--cells of $\Delta$ is equal to $\Area^{\rel}(W)$. 
\end{lem}

Let $\Cay(\Gamma, A\sqcup\mc{P})$ denote the Cayley graph of $\Gamma$ corresponding to the alphabet $A\sqcup\mc{P}$. Each edge of $\Cay(\Gamma, A\sqcup\mc{P})$ is labeled by an element of $A\sqcup\mc{P}$, and we denote the label of an edge $e$ by $\Lab(e)$.  Note that this graph may have multiple edges (with different labels) when distinct elements of $A\sqcup\mc{P}$ represent the same element of $\Gamma$. The first part of the following theorem is \cite[Theorem 1.7]{Osi06a}, and is part of proving that the above definition of relative hyperbolicity is equivalent to Farb's definition from \cite{Farb}.  The second part of the following theorem is \cite[Proposition 5.2]{Osi13}.

\begin{thm}\cite{Osi06a,Osi13} \label{t:RH AH}
Suppose $\Gamma$ is hyperbolic relative to $\{P_1,...,P_n\}$ and $A\subseteq \Gamma$ is a finite relative generating set. Then
\begin{enumerate}
\item $\Cay(\Gamma, A\sqcup\mc{P})$ is a hyperbolic metric space.
\item The action of $\Gamma$ on $\Cay(\Gamma, A\sqcup\mc{P})$ is acylindrical.
\end{enumerate}
\end{thm}

When $\Gamma$ is hyperbolic relative to $\{P_1,...,P_n\}$, we can define an associated metric $\hat{d}_j\colon P_j\to[0, \infty]$ on each $P_j$, called the \emph{relative metric}.  Note that we allow these metrics to take infinite values. For each $P_j$, the Cayley graph $\Cay(P_j, P_j \smallsetminus \{ 1 \})$ is naturally embedded as a complete subgraph of $\Cay(\Gamma, A\sqcup\mc{P})$. For $f, g\in P_j$, let $\hat{d}_j(f, g)$ equal the length of the shortest path from $f$ to $g$ in $\Cay(\Gamma, A\sqcup\mc{P})$ which contains no edges in the subgraph $\Cay(P_j, P_j \smallsetminus \{ 1 \})$, or $\hat{d}_j(f, g)=\infty$ if no such path exists. For $g\in P_j$, we let $|g|_j=\hat{d}_j(1, g)$.

Fix a sufficiently large constant $D$ (see \cite[Equation 29]{Osi13}). Let $p$ be a geodesic in $\Cay(\Gamma, A\sqcup\mc{P})$. We say that $p$ \emph{penetrates} a coset $xP_j$ if $p$ decomposes as $p=p_1ep_2$ where $e$ is an edge such that $\Lab(e)\in P_j$ and $\Lab(p_1)=_\Gamma x$. In this case we say that $e$ is \emph{inside} the coset $xP_j$. Furthermore, we say that $p$ \emph{essentially penetrates} the coset $xP_j$ if in addition $|\Lab(e)|_j\geq D$. 

Given $g\in\Gamma$, let $S(g)$ be the set of cosets of peripheral subgroups which are essentially penetrated by some geodesic in $\Cay(\Gamma, A\sqcup\mc{P})$ from $1$ to $g$. Cosets in $S(g)$ are called \emph{separating cosets}. 

The following is  proved in the more general context of hyperbolically embedded subgroups as \cite[Lemma 4.8]{Osi13}.
\begin{lem}\label{sepco}
For any $g\in\Gamma$, the set $S(g)$ can be ordered as $S(g)=\{C_1,..., C_m\}$ such that for any geodesic $p$ in $\Cay(\Gamma, A\sqcup\mc{P})$ from $1$ to $g$, $p$ decomposes as \[
p=p_1e_1p_2e_2...e_mp_{m+1}
\]
where each $p_j$ is a (possibly trivial) subpath of $p$ which essentially penetrates no cosets and each $e_j$ is an edge of $p$ which is inside $C_j$.
\end{lem}
The point is that any two geodesics with the same endpoints essentially penetrate the same cosets in the same order. Note also that for any $p_j$ in the above decomposition of $p$, the label of each edge of $p_j$ belongs to a finite set $B=A\sqcup(\sqcup_{i=1}^n\{g\in P_i\;|\; |g|_i\leq D\}$, where $D$ is the constant fixed above.

\subsection{Equationally noetherian}
In this section we prove Theorem \ref{thm:rel hyp eqn}.  For the convenience of the reader, we recall the statement.

\begin{reptheorem}{thm:rel hyp eqn}
If $\Gamma$ is hyperbolic relative to equationally noetherian groups, then $\Gamma$ is equationally noetherian. 
\end{reptheorem}
\begin{proof}

Suppose $\Gamma$ is hyperbolic relative to $\{P_1,...,P_n\}$ and that each $P_j$ is equationally noetherian. Suppose $\Gamma$ has a finite relative presentation of the form (\ref{relpres}), and let $X=\Cay(\Gamma, A\sqcup\mc{P})$. For $g\in\Gamma$, let $|g|=d_X(1, g)$ and for $g\in P_j$, recall that $|g|_j=\hat{d}_j(1, g)$. Let $B=A\sqcup\left(\sqcup_{j=1}^n\{g\in P_j \mid |g|_j\leq D\}\right)$.  By Theorem \ref{t:RH AH}, $X$ is $\delta$--hyperbolic for some $\delta$ and the action of $\Gamma$ on $X$ is acylindrical.  Thus, consider the set $\mc{G} = \left\{ ( \Gamma, X ) \right\}$.

 Let $G$ be a finitely generated group and $(\phi_i)$ a sequence from $\Hom(G, \Gamma)$ with $\lim^\omega(\|\phi_i\|)<\infty$.  By Theorem \ref{t:EN up to non-div}, to prove Theorem \ref{thm:rel hyp eqn} it suffices to show that $\omega$--almost surely $\phi_i$ factors through $\phi_\infty$. Let $L=G/\ker^\omega(\phi_i)$ be the associated (non-divergent) $\Gamma$--limit group.
 
Let $S=\{s_1,...,s_m\}$ be a finite generating set for $G$. After conjugating each $\phi_i$, we can assume that $\| \phi_i \|=\max_{s\in S}|\phi_i(s)|$. Fix $1\leq j\leq m$. Since $\omega$--almost surely we have $|\phi_i(s_j)|\leq\|\phi_i\|\leq\lim^\omega(\|\phi_i\|)<\infty$, there is some fixed $0\leq k_j\leq\lim^\omega(\|\phi_i\|)$ such that $\omega$--almost surely $|\phi_i(s_j)|=k_j$. Let $p^j_i$ be a geodesic in $X$ from $1$ to $\phi_i(s_j)$. Since $\omega$--almost surely the number of separating cosets $|S(\phi_i(s_j))|\leq\l(p_i^j)=k_j$, we see that $|S(\phi_i(s_j))|$ is $\omega$--almost surely constant.  Denote this ($\omega$--almost surely) constant value for $|S(\phi(s_j))|$ by $m_j$. Now $p_i^j$ has a decomposition as in Lemma \ref{sepco}, so we can write

\[
p_i^j=q^j_{i,1}e^j_{i,1}q^j_{i,2}e^j_{i,2}...e^j_{i,m_j}q^j_{i,m_j+1}
\]
where each $q^j_{i,k}$ is a (possibly trivial) subpath of $p_i^j$ with each edge labeled by an element of $B$ and each $e^j_{i,k}$ is an edge of $p^i_j$ which is inside a separating coset. Since $\omega$--almost surely $\sum_{k=1}^{m_j+1}\l(q^j_{i, k})\leq k_j$, there are $\omega$--almost surely only finitely many possible configurations of the list $(\l(q^j_{i,1}), \l(q^j_{i,2}),...,\l(q^j_{i,m_j+1}))$.  Therefore each $\l(q^j_{i, k})$ is fixed $\omega$--almost surely, and since its label belongs to the finite set $B$, in fact the label of each $q^j_{i, k}$ is $\omega$--almost surely independent of $i$. Similarly, for each $e^j_{i, k}$, the label of $e^j_{i, k}$ belongs to a subgroup $P_{t^j_k}$ which is $\omega$--almost surely independent of $i$; however, there are infinitely many choices for the label of $e^j_{i, k}$, so this stills depend on $i$.

In particular, $\omega$--almost surely we have $\phi_i(s_j)=_\Gamma W_1^jg^j_{i, 1}W^j_2...g^j_{i, m_j}W^j_{m_j+1}$, where each $W^j_k$ is a word in $B$ and each $g^j_{i, k}$ is an element of $P_{t^j_{k}}$.

We now define a set of abstract letters for each $1\leq j\leq m$, $\{h_1^j,...,h_{m_j}^j\}$ and also an abstract set of letters $\overline{B}$ together with a fixed bijection $\overline{B}\to B$. Let $\overline{A}\subseteq \overline{B}$ be the subset corresponding to $A\subseteq B$. Let $Z$ be the set of all such abstract letters, that is $Z=\overline{B}\sqcup\{h_k^j\;|\;1\leq j\leq m, 1\leq k\leq m_j\}$. For each $W^j_k$, let $U^j_k$ be the corresponding word in $\overline{B}$. There is a function $S\to F(Z)$, defined by sending each $s_j$ to $U_1^jh^j_{1}U^j_2...h^j_{m_j}U^j_{m_j+1}$. Let $N\leq F(Z)$ be the normal subgroup generated by the image of the relations in $G$.  Then we get a homomorphism $\rho\colon G\to H$, where $H = F(Z)/N$.  

Recall that for each $e^j_{i, k}$, there is a corresponding peripheral subgroup which contains the label of $e^j_{i, k}$ and which is $\omega$--almost surely independent of $i$. Let $J_t$ be the subgroup of $H$ which is generated by the image of all abstract letters corresponding to $P_t$. More precisely, $J_t$ is generated by the image of all abstract letters $\bar{b}\in\overline{B}\setminus\overline{A}$  such that the corresponding element $b\in B\setminus A$ belongs to $P_t$ together with the image of all abstract letters of the form $h^j_k$ such that the peripheral subgroup corresponding to $e^j_{i, k}$ is $\omega$--almost surely equal to $P_t$.

 Now we define a sequence of homomorphisms $(\hat{\phi_i}\colon H\ast(\ast_{t=1}^n P_t)\to \Gamma)$ which are the identity on each $P_t$ and which send each element of $\overline{B}$ to the corresponding element of $B$ and each $h_k^j$ to $\Lab(e^j_{i, k})$.  It follows easily from the construction that $\phi_i=\hat{\phi}_i\circ\rho$.

We now define the $\Gamma$--limit group $W=(H\ast(\ast_{t=1}^n P_t))/\ker^\omega(\hat{\phi_i})$, and note that $\hat{\phi}_\infty\circ\rho\colon G\to W$. Since $\phi_i=\hat{\phi}_i\circ\rho$, we get
\[
\ker(\phi_\infty)=\ker^\omega(\phi_i)=\ker^\omega(\hat{\phi_i}\circ\rho)=\ker(\hat{\phi}_\infty\circ\rho)
\]

Thus, there is a natural injective homomorphism $\iota\colon L\hookrightarrow W$. 

\[\begin{tikzcd}
 & H\ast(\ast_{t=1}^n P_t) \arrow{d}{\hat{\phi}_i}\arrow[twoheadrightarrow]{dr}{\hat{\phi}_\infty} \\
G \arrow{ur}{\rho} \arrow{r}{\phi_i}\arrow[twoheadrightarrow]{dr}[swap]{\phi_\infty}& \Gamma \arrow[dashleftarrow]{r} \arrow[dashleftarrow]{d} & W\\
& L\arrow[hookrightarrow]{ur}[swap]{\iota} &
\end{tikzcd}
\]

Since $J_t\leq H$, there is a natural inclusion $J_t\ast P_t\hookrightarrow H\ast(\ast_{t=1}^n P_t)$. Define $Q_t:=\hat{\phi}_\infty(J_t\ast P_t)$. Note that each $Q_t$ is a $P_t$--limit group, since $J_t\ast P_t$ is finitely generated and $\hat{\phi}_i|_{(J_t\ast P_t)}\colon (J_t\ast P_t)\to P_t$ is a defining sequence of maps for $Q_t$. Also the restriction of $\hat{\phi_i}$ to $P_t$ is injective for all $i$, so we can identify $P_t$ with its image in $Q_t$. Our next goal is to argue that $W$ is hyperbolic relative to $\{Q_1,..., Q_n\}$. Let $\mc{Q}=\sqcup_{t=1}^n(Q_t\setminus\{1\})$. Note that by identifying $P_t$ with its image in $Q_t$ we get a natural inclusion $\mc{P}\subseteq\mc{Q}$.

Recall that $\Gamma$ has the finite relative presentation 
\begin{equation} \label{Gamma pres}
\langle A, P_1,...,P_n\;|\; R\rangle.  
\end{equation}
For each $r\in R$, let $\overline{r}$ be the word in $\overline{A}\sqcup\mc{Q}$ obtained by replacing each $A$--letter in $r$ with the corresponding letter of $\overline{A}$ and considering each $\mc{P}$ letter as a letter in $\mc{Q}$ via the inclusion of $\mc{P}$ into $\mc{Q}$. Since $\overline{r}$ is a word in $\overline{A}\sqcup\mc{P}$, it defines an element of $H\ast(\ast_{t=1}^n P_t)$. Clearly $\hat{\phi}_i(\overline{r})=1$ $\omega$--almost surely, so $\overline{r}=_W1$.

Let $\overline{R}=\{\overline{r}\;|\; r\in R\}$. Then we claim that 
\[
\langle \overline{A}, Q_1,...,Q_m\;|\;\overline{R}\rangle
\]
is a finite relative presentation for $W$.

First observe that $\overline{A}\sqcup\mc{Q}$ generates $W$. This is because $Z\sqcup\mc{P}$ generates $H\ast(\ast_{t=1}^n P_t)$, $\mc{P}\subseteq \mc{Q}$ and for all $Z\in Z\setminus\overline{A}$, $\hat{\phi}_\infty(z)\in Q_t$ for some $1\leq t\leq n$.

Suppose $U$ is a word in $\overline{A}\sqcup\mc{Q}$ such that $U=_W1$. Let $\widetilde{U}$ be the word in $\overline{A}\sqcup\left(\sqcup_{1\leq t\leq n}(J_t\ast P_t) \right)$ obtained by replacing each $Q_t$--letter $h$ in $U$ with $\widetilde{h}\in\hat{\phi}_\infty^{-1}(h)$. Let $U_i$ be the image of $\widetilde{U}$ under $\hat{\phi}_i$, that is the word obtained by applying $\hat{\phi}_i$ to each letter of $\widetilde{U}$; equivalently, $U_i$ is obtained from $U$ by replacing each $\overline{A}$--letter is by the corresponding element of $A$ in $\Gamma$ and replacing each $Q_t$ letter $h$ by $\hat{\phi}_i(\widetilde{h})$. Note that $U_i$ is a word in $A\sqcup\mc{P}$ and $U_i=_\Gamma 1$ $\omega$--almost surely. 

Now, by definition of relative hyperbolicity, $\omega$--almost surely there exists van Kampen diagrams $\Delta_i$ over \eqref{Gamma pres} such that $\Lab(\partial\Delta_i)\equiv U_i$ and the number of $R$--cells of $\Delta_i$ is bounded by $K\|U_i\|=K\|U\|$. By Lemma \ref{lem:mindiagram}, each $\Delta_i$ can be chosen such that every internal edge of $\Delta_i$ belongs to the boundary of an $R$--cell. Since there is a uniform bound on the number of $R$--cells in each $\Delta_i$ and a uniform bound on the number of edges in the boundary of any $R$--cell, we get a uniform bound on the total number of edges in $\Delta_i$. Since there are only finitely many (unlabeled) $2$--complexes with a fixed number of edges, $\omega$--almost surely the $\Delta_i$'s all have a fixed isomorphism type as (unlabeled) $2$--complexes. 

Now let $\overline{\Delta}$ be an abstract $2$--complex which $\omega$--almost surely has the same isomorphism type as $\Delta_i$. Let $\Pi$ be a $2$--cell of $\overline{\Delta}$, and let $\Pi_i$ be the corresponding $2$--cell of $\Delta_i$. Then there are two possibilities for $\Pi_i$. The first is that $\omega$--almost surely $\Pi_i$ is an $R$--cell in $\Delta_i$. In this case, since $R$ is finite, there is $\omega$--almost surely a fixed $r\in R$ such that $\Lab(\partial \Pi_i)\equiv r$. In particular, since every internal edge of $\Delta_i$ belongs to the boundary of an $R$--cell, this means that $\omega$--almost surely the label of every internal edge of $\Delta_i$ is fixed. The second possibility is that $\omega$--almost surely $\Pi_i$ is not an $R$--cell. In this case, there is $\omega$-almost surely a fixed $1\leq t\leq n$ such that, each edge of $\Pi_i$ is labeled by an element of $P_t$ and $\Lab(\partial \Pi_i)=_{P_t}1$.

For each internal edge $e$ of $\overline{\Delta}$, $e$ is $\omega$--almost surely labeled by a fixed element of $A\sqcup\mc{P}$ in $\Delta_i$, so we label $e$ in $\overline{\Delta}$ by the corresponding element of $\overline{A}\sqcup\mc{Q}$. For each external edge $e$, the label of $e$ occurs in a fixed position in each $U_i$, and we label $e$ in $\overline{\Delta}$ by the corresponding letter in $U$. 

Suppose $\Pi$ is a $2$--cell of $\overline{\Delta}$. If $\Pi_i$ $\omega$--almost surely occurs as an $R$--cell in $\Delta_i$, then the label of $\Pi$ in $\overline{\Delta}$ belongs to $\overline{R}$. If $\Pi_i$ $\omega$--almost surely occurs as an $P_t$--cell in $\Delta_i$, then the label of $\Pi$ in $\overline{\Delta}$  is a word in $Q_t$. Indeed, for each edge $e$ of $\Pi$ which is internal in $\overline{\Delta}$, the label of $e$ is an element of $P_t\subseteq Q_t$. For each external edge $e$, with label $h$, we have that $\omega$--almost surely $\hat{\phi}_i(\widetilde{h})\in P_t$, hence $\widetilde{h}\in J_t\ast P_t$ and $h\in Q_t$. 

Thus, in the second case the label of $\partial\Pi$ is a word in $Q_t$ whose image in $\Gamma$ is $\omega$--almost surely trivial since it maps to the boundary of $2$--cell in $\Delta_i$. Therefore, $\overline{\Delta}$ is a van Kampen diagram for $U$ over the relative  presentation $\langle \overline{A}, Q_1,...,Q_n\;|\;\overline{R}\rangle$. Hence this is a presentation for $W$, and furthermore $\Area^{\rel}(U)\leq \Area^{\rel}(\overline{\Delta})\leq K\|U\|$, which shows that $W$ is hyperbolic relative to $\{Q_1,..., Q_n\}$.

In particular, we have shown that $W$ is finitely presented relative to $\{Q_1,..., Q_n\}$. Since each $P_t$ is equationally noetherian, $\omega$--almost surely the maps $\hat{\phi}_i|_{(J_t\ast P_t)}\colon (J_t\ast P_t)\to P_t$ factor through the limit map $\hat{\phi}_\infty|_{(J_t\ast P_t)}$. Thus, by Lemma \ref{lem:relfpen}, $\hat{\phi}_i$ $\omega$--almost surely factors through $\hat{\phi}_\infty$.

Finally, we show that $\phi_i$ $\omega$--almost surely factors through $\phi_\infty$, concluding the proof. Recall that $\phi_i=\hat{\phi}_i\circ\rho$. Hence $\ker(\hat{\phi}_\infty)\subseteq\ker(\hat{\phi}_i)$ which implies that $\ker(\phi_\infty)=\ker(\hat{\phi}_\infty\circ\rho)\subseteq\ker(\hat{\phi}_i\circ\rho)=\ker(\phi_i)$. Therefore, $\phi_i$ $\omega$--almost surely factors through $\phi_\infty$.
\end{proof}

The above proof also shows the following.
\begin{cor}
Suppose $\Gamma$ is hyperbolic relative to $\{P_1,...,P_n\}$ and $L$ is a non-divergent $\Gamma$--limit group. Then there exists a group $W$ which is hyperbolic relative to $\{D_1,...,D_n\}$ such that $L$ embeds into $W$ and each $D_i$ is a $P_i$--limit group.
\end{cor}

\bibliographystyle{plain} 
\def\cprime{$'$}

\end{document}